\documentclass[12pt,reqno]{amsart} 
\usepackage[headings]{fullpage}
\usepackage{amssymb,amsmath,amscd,bbm,mathrsfs,extarrows}
\usepackage{graphicx}
\usepackage{texdraw}
\usepackage{pb-diagram}
\usepackage[all,cmtip]{xy}
\usepackage{url}
\usepackage[bookmarks=true,%
    colorlinks=true,%
    linkcolor=blue,%
    citecolor=blue,%
    filecolor=blue,%
    menucolor=blue,% 
    urlcolor=blue,%
    breaklinks=true]{hyperref}
\usepackage{slashed}    
\usepackage{tikz-cd}
\usepackage{stmaryrd}
\usepackage{cancel,mathtools}
\usepackage{standalone,tikz,tikz-cd}
\usetikzlibrary{positioning,knots,patterns,hobby,decorations.pathreplacing}

\usepackage{soul,xcolor,fixmath,verbatim}
\usepackage{marginnote}

\newcommand{\no}[1]{}

\setstcolor{red}
\providecommand{\abs}[1]{\lvert #1\rvert}

\DeclareMathOperator{\tr}{tr}

\DeclareMathOperator{\Fr}{Fr}
\DeclareMathOperator{\Gr}{Gr}

\DeclareMathOperator{\ord}{ord}

\DeclareMathOperator{\id}{id}
\DeclareMathOperator{\pr}{pr}

\DeclareMathOperator{\GKdim}{GK dim}
\newcommand{\qeq}{{\,\, \stackrel{(q)}{=}\, \,}}

\usepackage[headings]{fullpage}
\usepackage{amssymb,epic,eepic,epsfig,amsbsy,amsmath,amscd,color}
\usepackage[all]{xy}
\usepackage{graphicx,subcaption}
\usepackage{texdraw}
\usepackage{url}
\usepackage{bbm}
\usepackage{mathrsfs}
\usepackage{standalone}
\usepackage{tikz}
\usetikzlibrary{cd,positioning,knots,patterns,hobby}
\usepackage{accents}
\usepackage[normalem]{ulem}
\usepackage{comment}

\definecolor{ForestGreen}{RGB}{34,139,34}
\definecolor{dartmouthgreen}{rgb}{0.05, 0.5, 0.06}

\newcommand{\orange}[1]{{\color{orange}#1}}
\newcommand{\blue}[1]{{\color{blue}#1}}

\definecolor{darkviolet}{rgb}{0.58, 0.0, 0.83}

\newtheorem{theorem}{Theorem}[section]
\newtheorem{thm}{Theorem}
\newtheorem{lemma}[theorem]{Lemma}

\newtheorem{corrr}[thm]{Corollary}

\newtheorem{proposition}[theorem]{Proposition}
\newtheorem{conjecture}{Conjecture}
\newtheorem{question}{Question}

\newtheorem{definition}[theorem]{Definition}
\theoremstyle{definition}
\newtheorem{remark}[theorem]{Remark}
\newtheorem{example}[theorem]{Example}
\newtheorem{corollary}[theorem]{Corollary}

\newcommand{\bcon}{\begin{conjecture}}
\newcommand{\econ}{\end{conjecture}}
\newcommand{\bcor}{\begin{corollary}}
\newcommand{\ecor}{\end{corollary}}
\newcommand{\bdf}{\begin{definition}}
\newcommand{\edf}{\end{definition}}
\newcommand{\benu}{\begin{enumerate}}
\newcommand{\eenu}{\end{enumerate}}
\newcommand{\beq}{\begin{equation}}
\newcommand{\eeq}{\end{equation}}
\newcommand{\bexa}{\begin{example}}
\newcommand{\eexa}{\end{example}}
\newcommand{\bexe}{\begin{exercise}}
\newcommand{\eexe}{\end{exercise}}
\newcommand{\bfac}{\begin{fact}}
\newcommand{\efac}{\end{fact}}
\newcommand{\bite}{\begin{itemize}}
\newcommand{\eite}{\end{itemize}}
\newcommand{\blem}{\begin{lemma}}
\newcommand{\elem}{\end{lemma}}
\newcommand{\bmat}{\begin{pmatrix}}
\newcommand{\emat}{\end{pmatrix}}
\newcommand{\bprb}{\begin{problem}}
\newcommand{\eprb}{\end{problem}}
\newcommand{\bpro}{\begin{proposition}}
\newcommand{\epro}{\end{proposition}}

\newcommand{\bque}{\begin{question}}
\newcommand{\eque}{\end{question}}
\newcommand{\brem}{\begin{remark}}
\newcommand{\erem}{\end{remark}}
\newcommand{\bthm}{\begin{theorem}}
\newcommand{\ethm}{\end{theorem}}

\newcommand{\bpr}{\begin{proof}}
\newcommand{\epr}{\end{proof}}
\newcommand{\ignore}[1]{}
\newcommand{\xra}{\xrightarrow}

\def\onto{\twoheadrightarrow}

\newcommand{\cC}{\mathcal{C}}
\newcommand{\bm}{\mathbf{m}}

\def\BC{\mathbb C}
\def\BN{\mathbb N}
\def\BZ{\mathbb Z}
\def\BR{\mathbb R}

\def\cA{\mathcal A}

\def\cD{\mathcal D}

\def\Pd{\cP_\partial}
\def\Dd{{\Delta_\partial}}

\def\ft{\mathfrak t}

\def\la{\langle}
\def\ra{\rangle}

\def\cM{\mathcal M}

\def\tQ{\tilde Q}

\def\cS{\mathscr S}
\def\ot{\otimes}

\def\cE{\mathcal E}
\def\cF{\mathcal F}

\def\bk{{\mathbf k}}
\def\bn{{\mathbf n}}

\def\cP{\mathcal P}
\def\tD{\tilde \Delta}

\def\Id{\mathrm{Id}}
\def\fS{\mathfrak S}
\def\bfS{\overline{\fS}}
\def\pfS{\partial \fS}

\def\cY{\mathcal Y}

\def\embed{\hookrightarrow}

\def\sX{\mathfrak X}

\newcommand{\al}{\alpha}

\newcommand{\fm}{\mathfrak m}

 \def\bm{\mathbf{m}}

 \def\cC{\mathcal C}

\def\bl{{\mathbf{l}} }

\def\tr{\mathrm{tr}}

\def\uSS{{\underline{\cS}(\fS)}}
\def\uSSR{{\underline{\cS}(\fS;\cR)}}
\def\cI{{\mathcal I}}
\def\uS{{\underline{\cS}}}

\def\SSp{\cS^+(\fS)}
\def\SSpR{\cS^+(\fS;\cR)}

\def\al{\alpha}

\def\be { \begin{equation} }
\def\ee { \end{equation} }

\def\bcS{\overline{\mathscr{S}} }
\def\bD{{\bar \Delta }}
\def\bTheta{\overline \Theta  }

\def\bt{ {\mathbf t}}

\def\bT{\mathbb T}

\def\lt{\mathrm{lt}}

\def\nc{\newcommand}
\def\Fr{\mathrm{Fr}}
\def\Weyl{{\mathrm{Weyl}}}

                  \nc\FI[2]{\begin{figure}
    \begin{center}\input{#1.pstex_t}\end{center}
    \caption{#2}
    \lbl{#1}
  \end{figure}}
\nc\FIG[3]{\begin{figure}
    \includegraphics[#3]{#1.eps}
    \caption{#2}
    \lbl{fig:#1}
    \end{figure}}
\nc\FF[3]{\begin{figure}
    \includegraphics[#3]{#1.eps}
    \caption{#2}
    \lbl{#1}
    \end{figure}}
    \nc\FIGc[3]{\begin{figure}[htpb]
    \includegraphics[height=#3]{#1.eps}
    \caption{#2}
    \label{fig:#1}
    \end{figure}}
    
    \nc\FIGjpg[3]{\begin{figure}[htpb]
    \includegraphics[height=#3]{draws/#1.jpg}
    \caption{#2}
    \label{fig:#1}
    \end{figure}}

    \nc\FIGh[3]{\begin{figure}[htpb]
    \includegraphics[height=#3]{#1.eps}
    \caption{#2}
    \lbl{fig:#1}
    \end{figure}}

\def\ord{\mathrm{ord}}

     \def\tr{\mathrm{tr}}

\def\lt{\mathrm{lt}}

\def\cT{\mathcal T}

\def\cR{{\mathcal R}}
\def\BC{{\mathbb C}}

\def\Do{{\mathring \Delta}}
\def\bQ{{\overline{ \mathsf Q }}}
\def\hDd{{\hat \Delta_\partial}}
\def\DD{{ \tilde \Delta}  }
\def\mQ{{\mathsf Q}}
\def\bsX{{\overline{\sX}}}
\def\xar{\xrightarrow}

\def\EEE{{\mathsf E}}

\def\barm{{\bar m}}
\def\tmQ{{\tilde{\mQ}}}

\def\Cut{{\mathsf{Cut}}}

\def\stno{\no}

\newcommand\incl[2]{{\includegraphics[height=#1]{draws/#2.eps}}}

\def\Dbu{{ \ddot\Delta}}
\def\Pbu{{ \overline\mP}}
\def\Ad{{{\cA}^\diamond}}
\def\aaa{{\mathbold{\mathbbm a}}}
\def\kkk{{\mathbold{\mathbbm k}}}
\def\SSR{{\cS(\fS;\cR)}}

\title%[Degeneration of skein algebra]
{Degenerations of Skein Algebras and Quantum Traces}

\author{Wade Bloomquist}

\author%[H. Karuo]
{Hiroaki Karuo}
\date{}

\author%[T. L\^e]
{ Thang L\^{e}}
\begin{document}

\begin{abstract}

We introduce a joint generalization, called LRY skein algebras, of Kauffman bracket skein algebras (of surfaces) that encompasses both Roger-Yang skein algebras and stated skein algebras. We will show that, over an arbitrary ground ring which is a commutative domain, the LRY skein algebras are  domains and have degenerations (by filtrations) equal to monomial subalgebras of quantum tori.  
For surfaces without interior punctures, this integrality generalizes a result of Moon and Wong to the most general ground ring.
We also calculate the Gelfand-Kirillov dimension of LRY algebras and show they are Noetherian if the ground ring is. Moreover they are orderly finitely generated.

To study the LRY algebras and prove the above-mentioned results, we construct quantum traces, both the so-called $X$-version for all surfaces and also an $A$-version for a smaller class of surfaces. We also introduce a modified version of Dehn-Thurston coordinates for curves which are more suitable for the study of skein algebras as they pick up  the highest degree terms of products in certain natural filtrations.

\end{abstract}

\maketitle

\tableofcontents

\def\No{\mathring N}
\def\Tei{{Teichm\"uller} }
\def\RS{{ \cR }}
\def\SS{{  \mathscr{S}(\fS)}}
    
    %%%%%%%%%%%
\newcommand{\relemp}{
\begin{tikzpicture}[scale=0.8,baseline=0.3cm]
\fill[gray!20!white] (0,0)rectangle(1,1);
\draw (1,0)--(1,1);
\end{tikzpicture}\,
}
\newcommand{\relconn}{
\begin{tikzpicture}[scale=0.8,baseline=0.3cm]
\fill[gray!20!white] (0,0)rectangle(1,1);
\draw (1,0)--(1,1);
\draw [very thick] (0,0.67)..controls(0.8,0.67) and (0.8,0.33)..(0,0.33);
\end{tikzpicture}\,
}
\newcommand{\relup}[2]{
\begin{tikzpicture}[scale=0.8,baseline=0.3cm]
\fill[gray!20!white] (0,0)rectangle(1,1);
\draw[-stealth] (1,0)--(1,1);
\draw[very thick] (0,0.67)--(1,0.67) (0,0.33)--(1,0.33);
\draw[inner sep=1pt,right] (1,0.67)node{\tiny #1} (1,0.33)node{\tiny #2};
\end{tikzpicture}
}
\newcommand{\reldown}[2]{
\begin{tikzpicture}[scale=0.8,baseline=0.3cm]
\fill[gray!20!white] (0,0)rectangle(1,1);
\draw[-stealth] (1,1)--(1,0);
\draw[very thick] (0,0.67)--(1,0.67) (0,0.33)--(1,0.33);
\draw[inner sep=1pt,right] (1,0.67)node{\tiny #1} (1,0.33)node{\tiny #2};
\end{tikzpicture}
}
\newcommand{\relarc}[2]{
\begin{tikzpicture}[scale=0.8,baseline=0.3cm]
\fill[gray!20!white] (0,0)rectangle(1,1);
\draw[-stealth] (1,0)--(1,1);
\draw[very thick] (1,0.67)..controls(0.2,0.67) and (0.2,0.33)..(1,0.33);
\draw[inner sep=1pt,right] (1,0.67)node{\tiny #1} (1,0.33)node{\tiny #2};
\end{tikzpicture}
}
%%%%%%%%%%%%
    
\section{Introduction}

\subsection{What are the skein algebras studied in this paper?}

Roughly speaking,  the skein algebras we will study are closely related to the \Tei space of a punctured surface, where a puncture has either a ``holed" structure (a la Thurston and Fock) or a ``decorated" structure (a la Penner). When there is no interior decorated puncture this algebra recovers the stated skein algebra introduced by the third author \cite{Le:triangular}. On the partially opposite side, when there is no boundary and no holed interior punctures, this algebra recovers the Roger-Yang skein algebra \cite{RY}. In the  simplest case, when there is no boundary and no  decorated punctures, the algebra is the well known Kauffman bracket skein algebra.  Our results are new even for the ordinary Kauffman bracket skein algebra.

\subsection{The Kauffman Bracket Skein Algebra}

Let $\fS=\Sigma_g \setminus \cP$ where $\Sigma_g$ is the oriented closed surface of genus $g$ and $\cP$ is a finite collection of points.  Let the ground ring $\mathcal{R}$ be a commutative domain with a distinguished invertible element $q^{1/2}$. 
The Kauffman bracket skein algebra $\SSR$, introduced by Przytycki and Turaev \cite{Prz,Turaev}, is the  $\cR$-module freely spanned by isotopy classes of link diagrams on $\fS$ subject to the Kauffman bracket relations \cite{Kauffman}

\begin{align}
\label{eqSkein0}
\begin{tikzpicture}[scale=0.8,baseline=0.3cm]
\fill[gray!20!white] (-0.1,0)rectangle(1.1,1);
\begin{knot}[clip width=8,background color=gray!20!white]
\strand[very thick] (1,1)--(0,0);
\strand[very thick] (0,1)--(1,0);
\end{knot}
\end{tikzpicture}
&=q 
\begin{tikzpicture}[scale=0.8,baseline=0.3cm]
\fill[gray!20!white] (-0.1,0)rectangle(1.1,1);
\draw[very thick] (0,0)..controls (0.5,0.5)..(0,1);
\draw[very thick] (1,0)..controls (0.5,0.5)..(1,1);
\end{tikzpicture}
+q^{-1}
\begin{tikzpicture}[scale=0.8,baseline=0.3cm]
\fill[gray!20!white] (-0.1,0)rectangle(1.1,1);
\draw[very thick] (0,0)..controls (0.5,0.5)..(1,0);
\draw[very thick] (0,1)..controls (0.5,0.5)..(1,1);
\end{tikzpicture}\, , \quad 
\begin{tikzpicture}[scale=0.8,baseline=0.3cm]
\fill[gray!20!white] (0,0)rectangle(1,1);
\draw[very thick] (0.5,0.5)circle(0.3);
\end{tikzpicture}
=(-q^2 -q^{-2})
\begin{tikzpicture}[scale=0.8,baseline=0.3cm]
\fill[gray!20!white] (0,0)rectangle(1,1);
\end{tikzpicture}\, . 
\end{align}

The product is given by stacking. See Section \ref{secLRY} for details.

The skein algebra $\SSR$ and its analogs have played an important role in low dimensional topology and quantum algebra.  For this reason we want  to understand algebraic properties of the skein algebra $\SSR$, for example its representation theory.

Bonahon and Wong \cite{BW} proved that when $|\cP|\ge 1$, the skein algebra $\SSR$ can be embedded into a \emph{quantum torus} (the Chekhov-Fock algebra) by the  {\bf quantum trace map}
\[\tr: \SSR \embed \bT(Q;\cR).\]

 When $q=1$ the quantum trace  sends a curve to the $SL(2,\mathbb{R})$ trace  expressed in shear coordinates of Teichm\"uller space. 

 Here the quantum torus of an antisymmetric $r\times r$ integral matrix $Q$ is the $\cR$-algebra
\be \bT(Q;\cR):= \cR\la x_i^{\pm 1}, i=1, \dots, r\ra /(x_i x_j = q^{Q_{ij}} x_j x_i).
\label{eqQto}
\ee
Thus $\bT(Q;\cR)$ is the algebra of Laurent polynomials in the $r$ variables $x_i$ which might not commute but are $q$-commuting in the sense that $x_i x_j= q^{Q_{ij}} x_j x_i$. A quantum torus is a domain, and is the simplest type of noncommutative algebra.
The quantum trace is an essential tool used to understand the skein algebra algebraically, and opens possibilities to quantize Thurston's theory of hyperbolic surfaces to build hyperbolic topological field theory. % and to better understand the volume conjecture.
 
The case $\cP=\emptyset$ is more difficult as there is no known embedding of $\SSR$ into a quantum torus except when $g=1$ \cite{FrG}. We have the following

\begin{thm} [Part of Theorem \ref{thmMain2}] Assume $\fS= \Sigma_g$, where $g \ge 2$,  and
 the ground ring $\mathcal{R}$ is a commutative domain with a distinguished invertible element $q^{1/2}$.  \label{thmSg}  Let $r=3g-3$.

(a) There exists an algebra $\mathbb{N}$-filtration of $\SSR$ such that 
the associated graded algebra embeds into a quantum torus $\bT(\tmQ;\cR)$, where $\tmQ$ is an integral $2r\times 2r$ antisymmetric matrix.

(b) There exists an algebra $(\BN\times \BZ)$-filtration of $\SSR$ such that  the associated graded algebra  is isomorphic to a monomial subalgebra $\bT(\tmQ, \Lambda;\cR)$ of $\bT(\tmQ;\cR)$, for a submonoid $\Lambda\subset \BZ^{2r}$.
%, which is the $\cR$-spannad of  associated to a submonoid $\Lambda\subset \BZ^r$.
\end{thm}
Here for a submonoid $\Lambda\subset \BZ^{2r}$, the monomial subalgebra $\bT(\tmQ, \Lambda)$ is the $\cR$-span of
 monomials $x_1^{k_1} \dots x_{2r}^{k_{2r}}$ in the presentation \eqref{eqQto} with $(k_1,\dots, k_{2r})\in \Lambda$.
 
Let us describe $\tmQ$ and $\Lambda$.  Choose a  pants decomposition  of $\Sigma_g$ and a dual graph $\Gamma$, which is a trivalent graph in $\fS$, see Section \ref{secDT} for full details. Then $\Gamma$ has $r$ edges $e_1,\dots, e_{r}$. Define the antisymmetric $r\times r$ matrix $\mQ$ by
\begin{align*}
&{\mQ}(i,j)=
\# \left(\begin{array}{c}\includegraphics[scale=0.17]{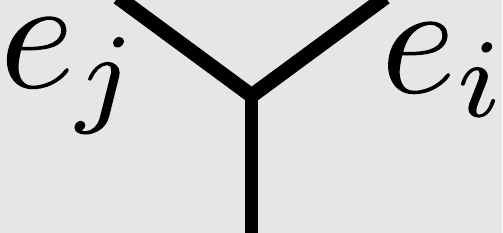}\end{array}\right)-\#\left(\begin{array}{c}\includegraphics[scale=0.17]{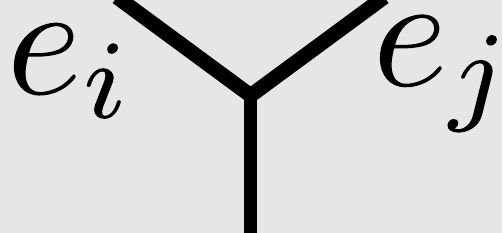}\end{array}\right).
\end{align*}

Here  the right hand side is the signed number of times a half-edge of $e_i$ meets a half-edge of $e_j$ at a vertex of $\Gamma$, where the sign is $+1$ if $e_i$ is clockwise to $e_j$, and $-1$ otherwise. Then $\tmQ$ is the {\bf symplectic double of $\mQ$}, defined by
\be 
\tmQ= \begin{bmatrix} \mQ &  \Id_{r} \\ -\Id_{r} & 0_{r}
\end{bmatrix}
\ee
where
$\Id_{r}$ and $0_{r}$ are respectively the  identity matrix and the  0 matrix of size $r$.  Additionally, $\Lambda$ is the submonoid of all possible Dehn-Thurston coordinates, see Section \ref{secDT}. w

\newcommand{\need}[1]{\orange {Need?: #1}}

\begin{remark} When we first announced our result,
    Detcherry and Santharoubane \cite{DS} also announced an embedding of $\cS(\Sigma_g)$ into a {\em localized} quantum torus (but not into a quantum torus). Neither Theorem \ref{thmSg} nor the result of \cite{DS} implies the other, and actually they complement each other. The techniques used in \cite{DS} are also quite different, and there the quantum torus is associated to the double of the trivial matrix.   {Our matrix $\mQ$ has a geometric meaning as it describes the symplectic structure of the \Tei space of a tubular neighborhood of the graph $\Gamma$.}

\end{remark}

\subsection{The Roger-Yang Skein Algebra}

\def\pbfS{\partial \bfS}
\def\tal{\tilde \al}
\def\Td{\cT^{\mathrm{dec}}}

\def\SP{(\Sigma, \cP)}
\def\SRY{\cS^{\mathrm{RY}}}
\def\SPT{\cS^{\mathrm{PT}}}
 
\def\SRYSR{\cS^{\mathrm{RY}}(\fS;\cR)}

Roger and Yang introduced the RY skein algebra $\SRYSR$ in an attempt to quantize the decorated \Tei space of a surface. This is in contrast to the above  $\SSR$ which is related to a quantization of the holed \Tei space (or enhanced \Tei space in the language of Bonahon and Wong).

In $\SRYSR$, where $\fS= \Sigma_g \setminus \cM$ with $|\cM| < \infty$, we allow not only link diagrams but also  arcs ending on points in $\cM$ with some height order, subject to the additional relations:
\begin{align}
&v \begin{array}{c}\includegraphics[scale=0.17]{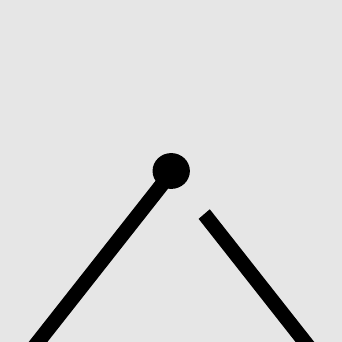}\end{array}=q^{1/2}\begin{array}{c}\includegraphics[scale=0.17]{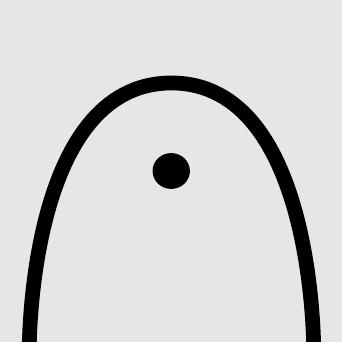}\end{array}+q^{-1/2}\begin{array}{c}\includegraphics[scale=0.17]{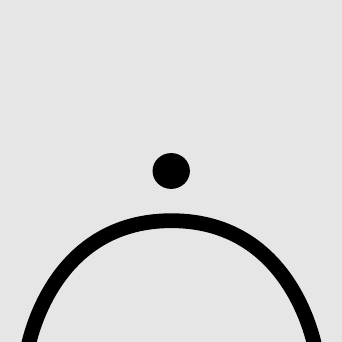}\end{array},\qquad  \begin{array}{c}\includegraphics[scale=0.17]{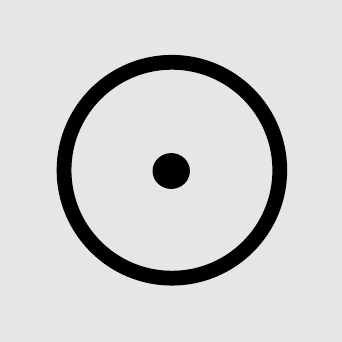}\end{array}=(q+q^{-1})\begin{array}{c}\includegraphics[scale=0.17]{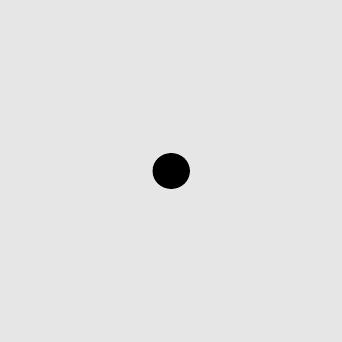}\end{array}. 
\label{eqEF}
\end{align}

Here we extend the ground ring to $\cR[ v^{\pm 1}, v \in \cM]$.  For details see Section \ref{secLRY}. One of our main results  is the following theorem about the structure of the Roger-Yang skein algebra.

\begin{thm} [See Theorem \ref{thmMain2}] \label{thmSRY} Assume $\fS= \Sigma_g \setminus \cM$,   {with $(g, |\cM|)\not \in \{ (1,0), (0,k) \mid k \le 3\} $}. Let  $\cR$ be a commutative domain with a distinguished invertible element $q^{1/2}$.

(a) There exists an algebra $\mathbb{N}$-filtration of $\SRYSR$ whose associated graded algebra embeds into a quantum torus $\bT(\tmQ;\cR)$, where $\tmQ$ is an integral $2r\times 2r$ antisymmetric matrix.

(b) The algebra $\SRYSR$ is a domain and is orderly finitely generated. Moreover $\SRYSR$ is Noetherian whenever $\cR$ is.

(c) There exists an algebra $(\BN\times \BZ^k)$-filtration of $\SRYSR$ such that  the associated graded algebra  is isomorphic to a monomial subalgebra of the quantum torus $\bT(\tmQ;\cR)$.

(d) The Gelfand-Kirillov dimension of $\SRY$ is $6g-6+3|\cM|$.
\end{thm} 

The matrix $\tmQ$ is determined by a  pants decomposition and its dual graph in a fashion similar to the case $\cP=\emptyset$.

One special case of (b) is 
\begin{corrr} The algebra $\SRY(\fS;\cR=\BC,q^{1/2}=1)$ is a domain.
\label{rIntgral0}
\end{corrr}
\def\Zq{{\BZ[q^{\pm 1/2}]}}
This fact is interesting because it is related to the origin of $\SRYSR$ as follows. Roger and Yang \cite{RY} introduced  relations \eqref{eqEF} specifically to encode Mondello's Poisson structure \cite{Mondello}  on the decorated Teichmuller space, where $v$  is the length of the horocycle about the corresponding puncture. They  showed that there is a Poisson algebra homomorphism 
\be \SRY(\fS;\cR=\BC, {q^{1/2}=1})\to \BC[\Td(\fS)],\label{eq.RY1} 
\ee
where $\BC[\Td(\fS)]$ is the ring of (complex valued) regular functions on $\Td(\fS)$. To have a meaningful theory of quantization, Roger and Yang conjectured that  the map \eqref{eq.RY1} is injective. This conjecture follows fairly easily from Corollary \ref{rIntgral0}, as observed in \cite{MW}, where  Corollary~\ref{rIntgral0} was  proved  under some restriction on the topology of the surface. When we announced our result, Moon and Wong \cite{MW2} also announced a proof of Corollary \ref{rIntgral0} for general surfaces.   { Since $\SRYSR$ is a free module over $\cR$, Corollary~\ref{rIntgral0} implies $\SRYSR$ is a domain when $\cR= \Zq$; however it does not imply integrality in general case. For example  part (b) of our theorem  implies the integrality in the important case when $\cR=\BC$ and  $q^{1/2}$ is a root of 1, which is not covered by \cite{MW2}. }

The {\bf orderly} finite generation in part (b) means there are $x_1, \dots, x_k\in \SRYSR$ such that the set  $\{ x_1^{n_1}\dots x_k^{n_k} \mid n_i \in \BN\}$ spans the algebra. This is stronger than  finite generation of $\SRYSR$, which was shown previously \cite{Wong}, and \cite{Bull} for  the ordinary skein algebra.   Pzrytycki and Sikora \cite{PS2} showed that $\cS(\Sigma_g;\cR)$ is a domain, and Noetherian whenever $\cR$ is Noetherian.

\def\SLRY{\cS^\mathrm{LRY}}
\def\LRYSR{\cS^\mathrm{LRY}(\fS;\cR)}

\subsection{LRY Skein Algebras}  One efficient way to study skein algebras is to cut the surface into elementary pieces, study the pieces, then glue them back to obtain global results. This was done by the third author \cite{Le:triangular} where he reproved the existence of the Bonahon-Wong quantum trace, and many other results.
We can use the same method here. Even though we begin with surfaces without boundary, it is useful to extend the theory to involve boundary in a meaningful way so that we can glue back later. 

Thus, now $\fS$ is an oriented surface where each connected component of the boundary $\pfS$ is an open interval $(0,1)$. The  punctures (or ideal points) in the interior of $\fS$ are partitioned into two types: ``holed" and``decorated". Consider tangle diagrams on $\fS$ where the end points can go to  $\pfS$ or to decorated points, with height order, and states $\pm $ on $\pfS$. The LRY skein algebra $\LRYSR$ is the module freely spanned by isotopy classes of such tangle diagrams subject to Kauffman bracket relations \eqref{eqSkein0}, Roger-Yang Relations \eqref{eqEF}, and the additional relations involving the endpoints on $\pfS$. See Section \ref{secLRY} for details. This is a unification of the Roger-Yang skein algebra  and the stated skein algebra of the third author.

An important feature of the introduction of the boundary is the existence of the cutting homomorphism. Suppose $c$ is an interior ideal arc of $\fS$ whose endpoints are not at interior decorated points. By cutting $\fS$ along $c$ we get a new surface $\Cut_c(\fS)$. 

\begin{thm} [Part of Theorem \ref{t.splitting2}] There exists an algebra embedding, given by an explicit state sum, from $\LRYSR$ into $\SLRY(\Cut_c(\fS);\cR)$.
\end{thm}

By cutting $\fS$ along interior edges of a triangulation,  where edges are not allowed to end at interior decorated punctures (see Section \ref{secTrig}), we can prove the following result about the quantum trace and the structure of $\LRYSR$.

\begin{thm} [Part of Theorems \ref{thm.embed3e} and \ref{thm-noether1}] \label{thm4}
 Suppose $\fS$ has a triangulation $\Delta$ with $r$ edges.   Let the ground ring $\cR$ be a commutative domain with a distinguished invertible element $q^{1/2}$.
 
 (a) There is an algebra embedding of $\LRYSR $ into a quantum torus
 $$\tr_{\Delta;\cR}: \LRYSR \embed \bT(\mQ_\Delta;\cR)$$
 where $\mQ_\Delta$ is an antisymmetric integral $r\times r$ matrix depending on the triangulation $\Delta
 $. Consequently $\LRYSR$ is a domain.
 
 (b) There is an algebra $\BN$-filtration of $\LRYSR$ whose associated graded algebra is a monomial subalgebra $\bT(\mQ_\Delta,\Lambda;\cR)$, where $\Lambda$ is the monoid of all possible triangular coordinates of simple diagrams on $\fS$.
 
 (c)  
 The algebra $\LRYSR$ is orderly finitely generated, has  Gelfand Kirillov dimension~$r$; and  moreover it is Noetherian if the ground ring $\cR$ is. 
 
\end{thm} 
We also construct a reduced version of the quantum trace $\tr_\Delta$. \

When $\fS$ has no boundary and no decorated punctures, part (a) of Theorem \ref{thm4} was proved by Bonahon and Wong. When $\fS$ has no decorated punctures, Theorem \ref{thm4} was proved in \cite{LY2}, whose proof is adapted here: First cutting  $\fS$ along edges of $\Delta$,  we get a collection of ideal triangles and 1-decorated monogons (monogons with one decorated puncture in the interior). The reduced quantum trace for triangles was constructed already in \cite{Le:triangular}. New here is the 1-decorated monogon, for which we will construct a reduced quantum trace. Then we patch the quantum traces of the faces of the triangulation together to get a reduced quantum trace for $\fS$. A trick of gluing a triangle along each boundary edge used in \cite{LY2} allows to extend it to a full quantum trace and prove Theorem \ref{thm4}.

\subsection{Modified Dehn-Thurston coordinates and another quantum trace} Theorems \ref{thmSg} and \ref{thmSRY} correspond to cases when $\fS$ does not admit an ideal triangulation and as such Theorem  \ref{thm4} does not apply. Instead, a  pants decomposition is used for the proofs of these theorems.

We need to use another type of quantum trace, which is an extension of the  $A$-version developed in \cite{LY2} and  an earlier version of Muller \cite{Muller}. Thus we develop a quantum trace for the so-called boundary simplified skein algebra which was considered in \cite{Le:Qtrace,PS2}. We calculate this quantum trace for each pair of pants in the decomposition, see Theorem \ref{thmbtr}, and show that they can be glued together to get a global algebra map described in part ~a) of Theorems \ref{thmSg} and \ref{thmSRY}.

To parameterize simple diagrams on $\Sigma_g\setminus \cP$, we use a modified version of Dehn-Thurston (DT) coordinates. The reason is the ordinary DT coordinates do not behave well under the product of the skein algebra. In \cite{PS2, FKL} one has to restrict to the so-called triangular multicurves in order to have control over the behavior with respect to the product. To get deeper results, we need to consider all multicurves. The modified DT coordinates, developed in Section \ref{secDT}, help to extract the highest degree term of the product, and help to prove Theorems \ref{thmSg} and \ref{thmSRY}.

\newcommand{\eqdef}{\overset{\mathrm{def}}{=\joinrel=}}
\def\ZQ{{\BZ_q}}
\def\Zq{{\mathbb Z[q^{\pm1/2}]}}
\subsection{Organization of the paper}
{Section \ref{sec.alg} contains algebraic background material. In Section \ref{secLRY}, we define LRY skein algebras and their reduced versions. 
Section \ref{secTrig} recalls details on ideal triangulations and the cutting homomorphism with an eye towards embedding LRY skein algebras into quantum tori. 
In Section \ref{secQtraces} we give quantum traces to embed LRY skein algebras into quantum tori using ideal triangles and 1-marked monogons. 
In Section \ref{secMRY} we give a definition of boundary simplified skein algebras to define modified Dehn--Thurston coordinates for (non-compact) pairs of pants in Section \ref{secDT}. 
In Section \ref{secPPP} we embed the associated graded algebras of the LRY skein algebra of a (non-compact) pair of pants into a quantum torus. 
By combining the elementary cases in Section \ref{secPPP}, we embed the associated graded algebra of the LRY skein algebra of a surface into quantum tori in Section \ref{secDegen}.  }

\subsection{Acknowledgements} The authors would like to thank  F. Bonahon,
 F. Costantino, C. Frohman, J. Kania-Bartoszynska, H. Moon, A. Poudel,  A. Sikora, H. Wong, and T. Yang  for helpful discussions.   The authors warmly thank the referee for substantial comments and suggestions which significantly helped to improve the presentation of paper. W. B. is supported in part by NSF Grant DMS-1745583. H. K. is supported by JSPS KAKENHI Grant Numbers JP22K20342, JP23K12976. T. L. is supported by NSF grant DMS-2203255.

\section{Notation and algebraic preliminaries}\label{sec.alg}

We fix notation and review the theory of quantum tori and Gelfand-Kirillov dimension.

\subsection{Notation and conventions} Denote by $\BN, \BZ, \BC$ respectively the set of non-negative integers, the set of integers, and the set of complex numbers. Note that our $\BN$ contains 0.

We assume that all rings are associative and unital, and ring morphisms preserve the unit.

Following \cite{MR} a ring is a {\bf domain} if $ab=0$ implies $a=0$ or $b=0$.

Let $\ZQ:=\Zq$, the ring of Laurent polynomials with coefficients in $\BZ$ and an indeterminate $q^{1/2}$. By a 
{\bf ground ring} $\cR$ we mean a commutative $\ZQ$-domain, i.e. a commutative $\ZQ$-algebra which is a domain.   For each ground ring there is a $\BZ$-algebra map $f:\ZQ \to \cR$ given by $f(x)= x* 1_\cR$. By abuse of notation, we also denote by $q^{1/2}= f(q^{1/2})\in \cR$.

Throughout this section we fix a ground ring $\cR$. If $f: A \to B$ is a morphism between $\ZQ$-modules,  we usually denote $f_\cR:= f \ot_\ZQ \cR: A \ot_\ZQ \cR \to B \ot _\ZQ \cR$.

Two elements $x,y$ of an $\cR$-algebra $A$ are {\bf $q$-proportional}, denoted by $x\qeq y$, if there is $k\in \BZ$ such that $x = q^{k} y$.  Two elements $x,y\in A$ are {\bf $q$-commuting} if $xy$ and $yx$ are $q$-proportional.

 \no{Suppose to every ground ring $\cR$ there is defined a morphism $f_\cR: A_\cR \to B_\cR$ of $\cR$-modules. We say the family of morphisms $\{ f_\cR\}$ is {\bf ground ring universal} if $f_\cR$ is obtained from $f_\ZQ$ by the change of ground ring,   $f _\cR = f_\ZQ \ot_ \ZQ \cR$.

Some properties of a $\ZQ$-module $A$ will transfer to $A_\cR=A\ot_ \ZQ \cR$ for any ground ring. For example, if $A$ is a free  $\ZQ$-module, then $A_\cR$ is free over $\cR$. Many properties are not transferable. Notably, $A$ is a domain does not imply that $A_\cR$ is a domain.
}

\def\qq{{q^{1/2}}}

\def\term#1{{\bf #1}}

\subsection{Weyl normalization and algebras with reflection}  Assume $A$ is a $\ZQ$-algebra which is torsion free as a $\ZQ$-module. 
%If $x,y\in A$ are  $q$-proportional, then there is a unique $k\in \BZ$ such that $x = q^{k} y$, unless $ab=0$.}

 Suppose $x_1,x_2,\dots,x_n\in A$ are pairwise $q$-commuting elements. Since $A$ is $\Zq$-torsion free,  
 there is a unique
 $c_{ij}\in \mathbb{Z}$ such that $x_i x_j = q^{c_{ij}} x_j x_i$ unless $x_i x_j=0$. 
 The {\bf Weyl normalization} of the product $x_1x_2\dots x_n$ is
\[[x_1x_2\dots x_n]_\Weyl=q^{-\frac{1}{2}\sum_{i<j}c_{ij}}x_1x_2\dots x_n.\]
If $\sigma$ is a permutation of $\{1,2,\dots,n\}$, then $[x_1x_2\dots x_n]_\Weyl=[x_{\sigma(1)}x_{\sigma(2)}\dots x_{\sigma(n)}]_\Weyl$, which motivates the use of the Weyl normalization. 

If $\cR$ is a ground ring, i.e. a commutative $\Zq$-domain, then we usually use the notation $[ab]_\Weyl$ to denote the element $[ab]_\Weyl \ot 1 \in A \ot_\ZQ \cR$.
%\subsection{Algebra with reflection}

 \label{ss.reflection}
We say $A$ is  a \term{$\ZQ$-algebra with reflection} if it is equipped with a $\BZ$-linear anti-involution $\omega$, called the \term{reflection}, such that $\omega(\qq)=q^{-1/2}$. In other words, $\omega : A \to A$ is  $\BZ$-linear  such that for all $x,y \in A$,
\[\omega(xy)=\omega(y)\omega(x),\qquad \omega(\qq x)=q^{-1/2} \omega(x).\]

An element $z\in A$ is  \term{reflection invariant} if $\omega(z)=z$. If $B$ is another $\ZQ$-algebra with reflection $\omega'$, then a map $f:A\to B$ is \term{reflection invariant} if $f\circ \omega=\omega'\circ f$.

The following trivial statement is very helpful in many calculations.
% as it allows to do calculation up to powers of $q$ and fix the exact powers at the end.
\blem\label{rReflection} Let $a, b$ be reflection invariant elements of a $\ZQ$-algebra with reflection.

(a) If $a \qeq b$  then $a=b$.

(b) If $a$ and $b$ are $q$-commuting, then $[ab]_\Weyl$ is reflection invariant.

\elem

\subsection{Quantum tori}

Let $Q$ be an antisymmetric $r\times r$ integral matrix. 
The {\bf quantum torus} associated to $Q$ is the algebra
\[{\mathbb{T}(Q)}\eqdef \ZQ\langle x_1^{\pm1},\dots,x_r^{\pm1}\rangle/\langle x_ix_j=q^{Q_{ij}}x_jx_i\rangle.\]

For $\bk=(k_1,\dots, k_r)\in \BZ^r$, let
\[ x^\bk \eqdef [ x_1 ^{k_1} x_2^{k_2} \dots x_r ^{k_r} ]_\Weyl= q^{-\frac{1}{2} \sum_{i<j} Q_{ij} k_i k_j} x_1 ^{k_1} x_2^{k_2} \dots x_r ^{k_r}\]
be the Weyl normalized monomial. Then $\{ x^\bk \mid \bk \in \BZ^r\}$ is a free $\ZQ$-basis of $\bT(Q)$, and
\begin{align}
\label{eq.prod}
x^\bk x ^{\bk'}& = q ^{\frac 12 \la \bk, \bk'\ra_Q} x^{\bk + \bk'},\quad \text{where } \la \bk, \bk'\ra_Q := \sum_{1\le i, j \le r} Q_{ij} k_i k'_j.
%\label{eq.commu} x^\bk x ^{\bk'} & = q ^{ \la \bk, \bk'\ra_Q} x^{\bk '} x ^\bk.
\end{align}

Suppose $Q'$ is another antisymmetric $r'\times r'$ integral matrix such that $HQ' H^T= Q$, where $H$ is an $r\times r'$ integral matrix and $H^T$ is its transpose. Then the $\ZQ$-linear map $\bT(Q)\to \bT(Q')$, given on the basis by $x^\bk \mapsto x^{\bk H}$, is an algebra homomorphism, called a {\bf multiplicatively linear homomorphism}. Here $\bk H$ is the product of the row vector $\bk$ and the matrix $H$.

The quantum torus $\bT(Q)$ has a reflection anti-involution
\[\omega:\mathbb{T}(Q)\to\mathbb{T}(Q), \quad \text{given by} \quad
\omega(q^{1/2})=q^{-1/2},\quad \omega(x_i)=x_i.\]
All normalized monomials $x^\bk$ are reflection invariant, and all multiplicatively linear homomorphisms are reflection invariant.

If $\Lambda\subset \BZ^r$ is a submonoid, then the $\ZQ$-submodule
{$\bT(Q,\Lambda)\subset \bT(Q)$} spanned by $\{ x^\bk\mid \bk \in \Lambda \}$ is a $\ZQ$-subalgebra of $\bT(Q)$, called a {\bf monomial subalgebra}. When $\Lambda= \BN^r$, the corresponding subalgebra is called the  {\bf quantum space} associated to $Q$, denoted by $\bT_+(Q)$.

For a ground ring $\cR$, the quantum $\cR$-torus and its $\Lambda$-monomial subalgebra  are defined by
\[{\mathbb{T}(Q;\cR):=\mathbb{T}(Q)\otimes_{\ZQ}\cR}, \quad \bT(Q, \Lambda; \cR) = \bT(Q,\Lambda)\ot _\ZQ \cR.\]

An $\cR$-algebra $A$ is {\bf orderly finitely generated} if there is a finite list of its elements $y_1, \dots, y_n\in A$ such that the set $\{ y_1^{k_1}\dots y_{n}^{k_n} \mid k_i \in \BN\}$ spans $A$ over $\cR$.

\begin{lemma}[ Lemma 2.1 of \cite{LY2}  ]\label{r.mono} 

Let $\Lambda\subset \BZ^r$ be a submonoid finitely generated as an $\BN$-module, and $\cR$ be a ground ring, i.e. a commutative $\ZQ$-domain.

(a) The monomial algebra {$\bT(Q, \Lambda;\cR)$} is orderly finitely generated.

(b) If $\cR$ is Noetherian then {$\bT(Q,\Lambda;\cR)$} is a Noetherian domain.
\end{lemma}

%Note when $\Lambda$ is a subgroup, then $A(Q;\Lambda)$ is a quantum torus. 

\subsection{Gelfand-Kirillov dimension}
The {\bf Gelfand-Kirillov dimension} (or GK dimension) provides a noncommutative analog of the Krull dimension.
Let $A$ be a finitely generated algebra over a field $\kkk$, and let $V$ be a finite dimensional generating subspace, e.g., the span of a finite set of  generators. The \emph{Gelfand-Kirillov dimension}, or GK dimension,
is defined as
\[\GKdim A \eqdef \limsup_{n\to\infty}\frac{\log \dim_k (\sum_{i=0}^n V^i)}{\log n}.\]
The dimension is independent of the choice of $V$. We extend the definition to an $\mathcal{R}$-algebra $A$ using $\kkk:=\Fr(\mathcal{R})$, the field of fractions of $\cR$ by
\[\GKdim A \eqdef \GKdim(A\otimes_{\mathcal{R}} \Fr(\mathcal{R})).\]

\begin{lemma}\label{lemma-GKdim}
Let $A$ be a finitely generated $\cR$-algebra.
\begin{enumerate}
\item If $B$ is a finitely generated subalgebra or a quotient of $A$, then $\GKdim B\le\GKdim A$.
\item The GK dimension of the monomial subalgebra  {$\bT(Q,\Lambda;\cR)$} is the $\BZ$-rank of $\Lambda$.
\end{enumerate}
\end{lemma}
Here the $\BZ$-rank of $\Lambda$ is the rank of the free abelian group $\Lambda\ot_\BN \BZ$.
\begin{proof}
(1)  is \cite[Propositions 8.2.2, 8.6.5]{MR}, and 
(2) is \cite[Lemma 2.2]{LY2}
\end{proof}

\subsection{Filtrations and associated graded algebras} 
\label{ssFiltr} Let $\Gamma$ be a submonoid of $\BZ^r$ for some $r\in \BN$. The lexicographic order on $\BZ^r$ induces a linear order on $\Gamma$.

A {\bf $\Gamma$-filtered $\cR$-module} $A$
 is an $\cR$-module equipped with a  {\bf $\Gamma$-filtration}, which is a family $F= (F_k(A))_{k\in \Gamma}$ of $\cR$-submodules of $A$ such that $F_k(A)\subset F_l(A)$ if $k \le l$ and $\cup_{k\in \Gamma} F_k(A) = A$. The {associated graded module} of $F$ is 
$$
  \Gr^F(A):=\bigoplus_{k\in\Gamma}\Gr^F_k, \ \text{where $\Gr_k^F:=F_k/F_{<k}$ and $F_{<k}:=\displaystyle{\sum_{k'<k}F_{k'}}$. }
$$
An $\cR$-linear map $f: A \to A'$ between two $\Gamma$-filtered {modules} {\bf respects the $\Gamma$-filtrations} if $f( F_k(A)) \subset F_k(A')$ for all $k \in \Gamma$. Such a map induces the {\bf associated graded map}
$$ \Gr^F(f): \Gr^F(A) \to \Gr^F(A'),\   a  + F_{<k}(A) \to f(a) + F_{<k} (A')\ \text{for } \ a\in F_k(A).$$

 A $\Gamma$-filtration $F$ is {\bf good} if for every non-zero $a\in A$, there is $k\in \Gamma$ denoted by $\deg^F(a)$, such that $a \in F_k(A) \setminus F_{<k}(A)$. Then define the {\bf lead term} of $a$ by $\lt(a) = p_k(a)\in \Gr^F(A)$, where $p_k: F_k \onto \Gr^F_k$ is the natural projection. By convention $\lt(0)=0$. Note that $\lt(a) \neq 0$ unless $a=0$.
 
 \blem \label{rLift1}
  Let $A, A'$ be $\Gamma$-filtered $\cR$-modules and $f: A \to A'$ be an $\cR$-linear map respecting the $\Gamma$-filtrations. Assume the $\Gamma$-filtration of $A$ is good.
%(a) If $\sum _i \lt(a_i) =0$ and $\deg^F(a_i) \le k$ for all $i$ then $\sum_i a_i \in F_{<k}(A)$.
If $\Gr^F(f): \Gr^F(A) \to \Gr^F(A')$ is an isomorphism then $f^{-1}(F_k(A')) = F_k(A)$ and $f$ is injective.
\elem
\bpr% (a) follows right away from the definition. Let us prove (b). 
Since $f(F_k(A)) \subset F_k(A')$  it is clear that $f^{-1}(F_k(A')) \supset F_k(A)$.
Assume the contrary that there exists $a \in f^{-1}(F_k(A')) \setminus F_k(A)$.
Then $l:=\deg^F(a) >k$. Because $f(a) \in F_k(A') \subset F_{< l}(A')$ we have $\Gr^F(f)(a)=0$, contradicting the fact that $\Gr^F(f)$ is an isomorphism.
\epr

\no{ Assume $A$ is a free $\cR$-module. An $\cR$-basis $S$ of $A$ respects to a $\Gamma$-filtration $F$ if $S \cap F_k(A)$ is a free $\cR$-basis of $F_k(A)$ for all $k\in \Gamma$.}

When $A$ is an $\cR$-algebra, we say a  $\Gamma$-filtration $F$ {\bf respects the product}, or we call $F$ an {\bf algebra $\Gamma$-filtration}, if $1\in F_0(A)$ and $F_k(A) F_l(A) \subset F_{k+l}(A)$. Then $\Gr^F(A)$ has an $\cR$-algebra structure defined by
$
p_k(x)p_{k'}(y)=p_{k+k'}(xy).
$

Many properties of the associated graded algebra can be lifted to the original algebra.

\begin{proposition}\label{liftfacts} 

Let $A$ be a finitely generated $\cR$-algebra, equipped with an algebra $\Gamma$-filtration.

\begin{enumerate}
\item  One has  $\GKdim(\mathrm{Gr}(A))\leq \GKdim(A)$.

\item   Assume $\Gamma=\BN$. \\
(a) If $\mathrm{Gr}(A)$ is a domain, then $A$ is a domain.\\
(b)  If $\mathrm{Gr}(A)$ is Noetherian, then $A$ is Noetherian.\\
(c) If   $F_k$ is finite dimensional over $\Fr(R)$ for {all} $k \in \BN$ then $\GKdim\Gr(A)= \GKdim(A)$.\\
(d)  If $\Gr(A)$ is orderly finitely generated, then so is $A$.\\

\end{enumerate}

\end{proposition}
\bpr
(1)  is \cite[Lemma 2.1 (7)]{Zhang}.\\
(2) (a), (b,) and (c) are standard;  see Theorems  1.6.6, 1.6.9, and  8.6.5 of \cite{MR}.\\
(d)  follows inductively by lifting spanning sets of  $\bigoplus_{k'\leq k} G_{k'}$ to $F_k$.
\epr

\section{The LRY skein algebra}\label{secLRY}

In this section we give the definition of the LRY (L\^e-Roger-Yang) skein algebra and its reduced version,  and present basic facts, including a description of a basis of the LRY algebra and the cutting homomorphism.

 In the theory of decorated and holed Teichmuller spaces, an ideal point in the interior of the surface is equipped with either a horocycle (decoration) or a holed structure.  We call the first type a {\bf marked} puncture, and the second a {\bf holed} puncture. To be technically easier to describe  diagrams on the surfaces, we will present a holed puncture as still an ideal point, while a marked point as a {\bf circle boundary component}.
 
 Throughout $\cR$ is a ground ring, which is a commutative $\ZQ$-domain.

\subsection{Punctured surfaces} In this subsection we define and fix conventions concerning punctured surfaces, compact arcs, and ideal arcs.

\bdf  A {\bf punctured surface} 
 $\fS$ is a surface of the form $\mathfrak{S}=\bfS\backslash \mathcal{P}$ where $\overline{\mathfrak{S}}$ is a compact, oriented surface with (potentially empty) boundary $\pbfS$ and $\mathcal{P}$ is a  finite set. An element of $\cP$ is called a puncture, or ideal point, of $\fS$; it is a boundary puncture if it lies on the boundary of $\overline{\fS}$, otherwise it is an interior puncture.  

Each connected component of $\pfS$ which is an open interval is called a {\bf boundary edge}. 
\edf
Note that $\bfS$ is uniquely determined by $\fS$.  The set of all circle components of $\pfS$ is denoted by $\cM(\fS)$, or simply  $\mathcal{M}$ when the surface is clear from context.
In figures we will depict punctures by small squares.

An {\bf ideal arc} on $\fS$ is a smooth embedding $a: (0,1) \embed \fS$ which can be extended to an immersion $\bar a: [0,1] \to \bfS$ such that $\bar a(0), \bar a(1) \in \cP$. This ideal arc is {\bf trivial} if the image of $\bar a$ bounds a disk in $\bfS$, which forces  $\bar a(0)=\bar a(1)$. 

\def\tfS{ { \widetilde{\fS}  }}

\def\ptfS{ {\partial { \tfS }}}

A {\bf $\pfS$-arc} is a smooth  proper embedding $a: [0,1] \embed \fS$. A $\pfS$-arc is {\bf trivial} if it can be homotoped in $\fS$ relative its boundary points into $\pfS$.

A {\bf loop on $\fS$} is  a simple closed curve on $\fS$. A loop is {\bf trivial} if  it bounds a disk in $\fS$. A loop is {\bf peripheral} if  it is parallel to an element of $\cM$, i.e.  circle boundary component.

The {\bf thickening of $\fS$} is $\tfS := \fS \times (-1,1)$, an oriented 3-dimensional manifold with   boundary  $\ptfS = \partial\mathfrak{S}\times (-1,1)$. Each connected component of $\ptfS$ is $c\times (-1,1)$ where $c$ is either a boundary edge or an element of $\cM$. In the first case we call $c\times (-1,1)$ a {\bf boundary wall}, and in the latter case a {\bf  boundary well}. For a point $z=(x,t)\in \fS\times (-1,1)$ the number $t$ is called its {\bf height}. The tangent space $T_z(\tfS)$ can be identified with $T_x(\fS) \times (-1,1)$, and a tangent vector in $T_z(\tfS)$ is  {\bf vertical} if it is parallel to the component $(-1,1)$ and has the positive direction of $(-1,1)$.

We will often identify $\fS$ with the copy $\fS \times \{0\}\subset \tfS$.

\subsection{Tangles, diagrams, and states} To prepare for the definition of the LRY algebra, we define stated tangles and their diagrams.

\def\pal{{\partial \al}}
\bdf
%[Tangles over $\fS$] 

 (a) A tangle over $\fS$, or a $\ptfS$-tangle, is a  compact 1-dimensional non-oriented proper submanifold $\al \embed \tfS$ equipped with a normal vector field, called the framing, such that
\begin{itemize}
\item at each boundary point of $\al$ the framing vector is  vertical, and
\item boundary points of $\al$ on the same connected component of $\ptfS$ have distinct heights.
\end{itemize}
We will denote the set of endpoints of $\al$ by $\pal$.

(b) The height order on $\pal$ is the partial order which compares the height of endpoints lying on the same connected component of $\ptfS$.

(c) Two endpoints of $\al$ are consecutive if with respect to the height order there is no boundary point lying between them.

(d) The empty set, by convention, is a $\ptfS$-tangle.

(e) {\bf Isotopies of  $\ptfS$-tangles} are considered  in the class of $\ptfS$-tangles. 
\edf

 {Note that isotopies of $\ptfS$-tangles preserve the height order as exchanging consecutive endpoints would pass through a representative where the heights are not distinct.}

As usual, $\ptfS$-tangles are depicted by their diagrams on $\fS$, as follows.
Every $\ptfS$-tangle is isotopic to one with vertical framing.
Suppose a vertically framed $\ptfS$-tangle $\al$ is in general position with respect to
the standard projection $\pi: \fS \times (-1,1) \to \fS$, i.e.  the restriction $\pi |_{\al}:\al \to \fS$ is an immersion with transverse  double points as the only possible singularities and there are no double points on the boundary of $\fS$. 
Then $D=\pi (\al)$, together with the over/under passing information at every double point, is called a  {\bf  $\pfS$-tangle  diagram}, or a {\bf tangle diagram on $\fS$}.

Such a tangle diagram $D$ is {\bf boundary ordered} if it is equipped with a 
linear order on $D \cap b$ for each boundary component $b$ of $\fS$.
{\bf Isotopies} of  $\pfS$-tangle diagrams are  ambient isotopies in $\fS$.

A $\pfS$-tangle diagram $D$ of a $\ptfS$-tangle $\al$ inherits a boundary order from the height order of $\al$.
Clearly $D$, with this boundary order, determines the isotopy class of $\al$.
 When there is no confusion, we identify a boundary ordered $\pfS$-tangle diagram with the isotopy class of $\ptfS$-tangles.

\def\ori{{\mathfrak o}} A {\bf marked orientation}
 $\ori$ consists of 
 an {\bf orientation} of $\pfS$ and a point on each circle component of $\pfS$, called the {\bf initial point} of that component.  
 A $\partial \fS$-tangle diagram  $D$ is said to have   the {\bf $\ori$-order} if the height order is given by the direction of $\ori$. This means on a boundary component $b$ the height order  is  increasing when one goes along $b$ in  the direction of $\ori$, starting at the initial point if $b$ is a circle. 
 
 If the orientation $\ori$ of each component  is   induced from that $\fS$, the $\ori$-order is called a {\bf positive order}.
 
%  It is clear that every $\ptfS$-tangle, after an isotopy, can be presented by an  $\ori$-ordered $\pfS$-tangle diagram.

\def\pwal{\partial_w(\al)} 
 
 For a $\ptfS$-tangle $\al$,  or a $\pfS$-tangle diagram $\al$, the {\bf wall boundary} $\pwal$ is the set of all endpoints of $\al$ lying in the boundary walls of $\mathfrak{S}$.
    A {\bf  state}  of $\al$  is a function $s:\pwal\rightarrow\{\pm\}$.    Note that there are no states assigned to endpoints on the boundary wells.

 \subsection{The LRY Skein Algebra}  \label{ssLRY}

\bdf \label{defLRY}
 The LRY (L\^e-Roger-Yang) skein algebra $\SS$ of the punctured surface $\fS$ is the {$\ZQ$-module} freely spanned by the isotopy classes of stated $\ptfS$-tangles subject to the relations (A)-(F): \edf \begin{align}
\label{eq.skein}
\begin{tikzpicture}[scale=0.8,baseline=0.3cm]
\fill[gray!20!white] (-0.1,0)rectangle(1.1,1);
\begin{knot}[clip width=8,background color=gray!20!white]
\strand[very thick] (1,1)--(0,0);
\strand[very thick] (0,1)--(1,0);
\end{knot}
\end{tikzpicture}
&=q
\begin{tikzpicture}[scale=0.8,baseline=0.3cm]
\fill[gray!20!white] (-0.1,0)rectangle(1.1,1);
\draw[very thick] (0,0)..controls (0.5,0.5)..(0,1);
\draw[very thick] (1,0)..controls (0.5,0.5)..(1,1);
\end{tikzpicture}
+q^{-1}
\begin{tikzpicture}[scale=0.8,baseline=0.3cm]
\fill[gray!20!white] (-0.1,0)rectangle(1.1,1);
\draw[very thick] (0,0)..controls (0.5,0.5)..(1,0);
\draw[very thick] (0,1)..controls (0.5,0.5)..(1,1);
\end{tikzpicture}\, ,   \tag{A}\\
\label{eq.loop}
\begin{tikzpicture}[scale=0.8,baseline=0.3cm]
\fill[gray!20!white] (0,0)rectangle(1,1);
\draw[very thick] (0.5,0.5)circle(0.3);
\end{tikzpicture}
&=(-q^2 -q^{-2})
\begin{tikzpicture}[scale=0.8,baseline=0.3cm]
\fill[gray!20!white] (0,0)rectangle(1,1);
\end{tikzpicture}\, , \tag{B}\\
\label{eq.arcs}
\relarc{$+$}{$-$}&=q^{-1/2}\relemp,\quad
\relarc{$+$}{$+$}=0, \quad \relarc{$-$}{$-$}= 0,   \tag{C} \\
\label{eq.order}
\relup{$-$}{$+$}&=q^2\relup{$+$}{$-$}+q^{-1/2}\relconn\, ,  \tag{D}
\end{align}
\begin{align}
&\hspace{-10mm}\begin{array}{c}\includegraphics[scale=0.15]{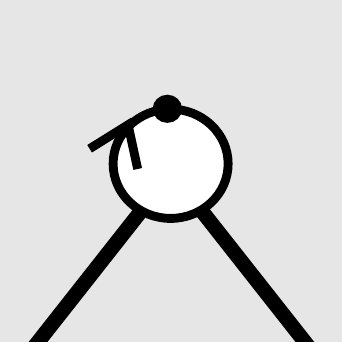}\end{array}=q^{1/2}\begin{array}{c}\includegraphics[scale=0.16]{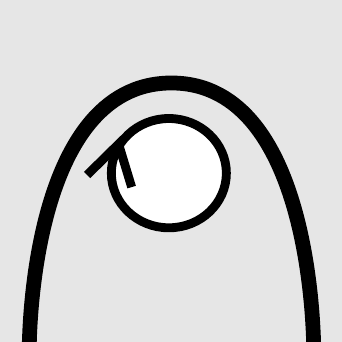}\end{array}+q^{-1/2}\begin{array}{c}\includegraphics[scale=0.16]{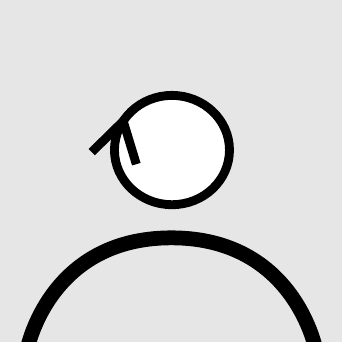}\end{array},  \tag{E} \label{eqE}\\
&\hspace{-10mm}\begin{array}{c}\includegraphics[scale=0.15]{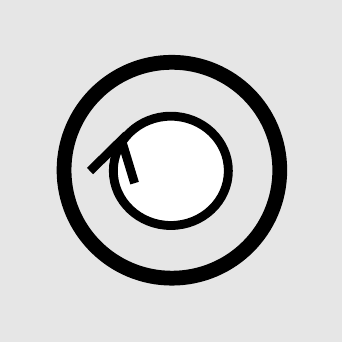}\end{array}=(q+q^{-1})\begin{array}{c}\includegraphics[scale=0.16]{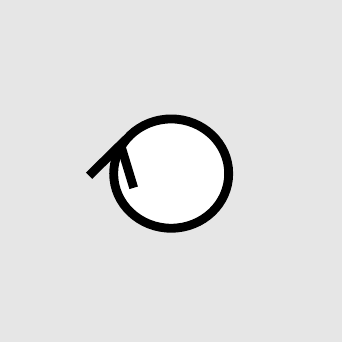}\end{array}, \tag{F} \label{eqF}
\end{align}
where multiplication is defined by stacking with respect to $(-1,1)$.

We adopt the following convention for interpreting these figures, 
as well as all other figures in the paper.
Each shaded part $S$ is a closed square, or square without an interior of a disk in \eqref{eqE} and \eqref{eqF}, subset of $\fS$, and the drawn diagram on $S$ is the diagram of $\al \cap (S \times J)$, where $J$ is a closed connected interval subset of  $(-1,1)$. We assume that $\al$ does not intersect $S \times \partial J$.
Each vertical line in Relations \eqref{eq.arcs} and \eqref{eq.order} is a part of $\pfS$; the endpoints on it are consecutive and ordered by the drawn direction. Each directed circle in \eqref{eqE} and \eqref{eqF} is a circle boundary component of $\fS$, with  the initial point at the top position. The circle component in (F) is not contractible.

We will call $S \times J$ a {\bf support} of the relation.

The skein module $\SS$ has an $\cR$-algebra structure, where the  product of two stated $\ptfS$-tangles $\al$ and $\beta$ is the result of stacking $\al$ above $\beta$.

The algebra $\SS$ admits a natural reflection $\omega$,  defined so that if 
$D$ is a stated boundary ordered $\pfS$-tangle diagram, then $\omega(D)$ is the result of switching all the crossing and reversing the height order on each boundary edge. It is easy to check that
 $\omega$ respects all the defining relations, and hence gives a well-defined anti-involution.

\brem There are several special cases:

\begin{itemize}\item  When  $\fS$ has no boundary, only \eqref{eq.skein} and \eqref{eq.loop} are used, and   $\SS$ is the usual Kauffman bracket skein algebra \cite{Prz,Turaev}.
 
\item When there is no circular boundary component, only \eqref{eq.skein}-\eqref{eq.order} are used, and $\SS$ is the  stated skein algebra introduced by the third author \cite{Le:triangular}. 
   
\item  When $\fS$ has no boundary edges, only  \eqref{eq.skein}, \eqref{eq.loop}, \eqref{eqE}, and \eqref{eqF} are used, and $\SS$ is almost the Roger-Yang skein algebra \cite{RY}, introduced to quantize Mondello's Poisson structure of {Penner's} decorated Teichm\"uller space.  See Subsection \ref{ssRY}.

\end{itemize}

  \erem

\subsection{Change of ground ring} For a ground ring $\cR$, the skein algebra $\SSR$ is defined identically to Definition \ref{defLRY}, with $\ZQ$ replaced by $\cR$. The right exactness of the tensor product implies 
\be \SSR=  \cS(\fS) \ot_\ZQ  \cR.
\label{eqRRR}
\ee

\subsection{A basis of {$\SSR$}} We now describe a basis of the $\cR$-module {$\SSR$}.

\bdf  \label{defSimpleDia}
  A {\bf simple diagram  on $\fS$} is a $\pfS$-tangle diagram $\al$ which has no crossing, no trivial $\pfS$-arc, no trivial loop, no peripheral loop, and additionally   each circle boundary component of $\fS$ intersects $\al$ at at most one point.
\edf

A state on a boundary ordered $\pfS$-tangle diagram $\al$ is {\bf increasing} if the state function $s: \pwal \to \{ \pm \}$ is an increasing function, i.e. $x \le y$ in the height order implies $s(x) \le s(y)$. Here as usual we identify $+$ with $1$ and $-$ with $-1$, so that $ - < +$.
  
\begin{theorem}\label{basisDiamond}
The set $B(\fS)$ of isotopy classes of increasingly stated, positively ordered, simple diagrams on a punctured surface $\fS$ is a free  $\cR$-basis for the $\cR$-module {$\SSR$}.
\end{theorem}
Note that for a simple diagram $\al$ we don't need to specify an order on $\pal \cap b$, where $b$ is a circle component, since $|\al \cap b | \le 1$.

\begin{proof}
 The proof is almost identical to the one in \cite{Le:triangular}, where the case when $\pfS$ has no circle component was considered. Actually all the technical points have been covered in \cite{Le:triangular,RY}. We will be brief.

We will use  the diamond lemma, in the language of confluence theory for graphs due to Sikora-Westbury \cite{SW}.  
Let $\tilde{B}$ be the set of all positively ordered stated $\partial\fS-$tangle diagrams.  Denote $\cR\tilde{B}$ the $\cR$-module freely spanned by $\tilde B$. Define the binary relation $\rightarrow$ on $\cR\tilde{B}$, by $D\rightarrow E$ if $D\in \tilde{B}$ and $E\in \cR\tilde{B}$ where $D$ is the left hand side of one of the LRY defining relations and $E$ is the corresponding right hand side.  This is then extended $\cR$ linearly from $\tilde{B}$ to $\cR\tilde{B}$.  If  $D\to E_1 \dots \to E_k=E$,  we say $E$ is a descendant of $D$.

If $\rightarrow$ is {\bf terminal} and {\bf locally confluent}, then the subset of $\tilde{B}$ of {\bf irreducible elements}, i.e. elements having no descendants, is a basis for the LRY skein algebra. Here terminality  means we cannot have an infinite sequence $E_1 \to E_2 \to E_3 \to \dots$. Additionally, local confluence means for $D\in \tilde B$, if $D\to E_1$ and $B\to E_2$, then $E_1$ and $E_2$ have a common descendant.

To prove terminality it is enough to introduce a complexity function $c: \tilde{B}\to \BN$,  such that if $D \to E$, where $D \in \tilde B$, then $E$ is an $\cR$-linear combination of diagrams with less complexity. It is easy to check that the following function satisfies this requirement:

$$c(D) =2 (\# \text{crossings}) + |D| + |\partial D| + \#\text{inversions} + m^2,$$
where $|D|$ is the number of components, an inversion is a pair of boundary points where a negative state is higher than a positive state, and $m$ is the number of $\partial D$ which are on circle components of $\pfS$.

Now we show that $\rightarrow$ is locally confluent on $\tilde{B}$. Assume  $D\rightarrow E_1$ and $D\rightarrow E_2$. 

For each defining relation the change of tangles happens in the  support.
Suppose  the support $P_1$ of $D \to E_1$ is disjoint from the support $P_2$ of $D\to E_2$. Then applying $P_2$ to $E_1$ is the same as applying $P_1$ to $ E_2$, and the result is a common descendant of $E_1$ and $E_2$.

Consider the case when $P_1$ and $P_2$ are not disjoint. A support of \eqref{eqE} or \eqref{eqF} can be shrunk so that  it is  disjoint from supports of \eqref{eq.skein}-\eqref{eq.order}. Hence either both $P_1$ and $P_2$ are from \eqref{eq.skein}-\eqref{eq.order} or both are from \eqref{eqE} or \eqref{eqF}. The former is the case of the stated skein algebra and in \cite{Le:triangular} it is proven that $P_1$ and $P_2$ are locally confluent. The latter is the case of RY algebra and in \cite{RY} it is proved that $P_1$ and $P_2$ are locally confluent.

 Finally,  we observe that $B(\fS)$ is  exactly the set of irreducible elements. 
\end{proof}

\bexa\label{ssExam}
 An important punctured surface is the {\bf  1-marked monogon} $\fm$, which is the  result of removing one puncture from the boundary $S^1 \times \{0\}$  of the annulus $S^1 \times [0,1]$.  Here $S^1$ is the standard circle. Theorem \ref{basisDiamond} shows that $\cS(\fm;\cR)$ has the free $\cR$-basis
\be 
B(\fm) = \{ b(k,l) \mid 0 \le l \le k \in \BN \}
\label{eq.basisfm}
\ee
where $b(k,l)$ is the positively ordered $\partial  \fm$-tangle diagram $b(k)$ depicted in Figure \ref{fig:uu_generator0} with negative states on the first lowest $l$ endpoints and positive states on other endpoints.

\begin{figure}[htpb!]
    \centering
    \includegraphics[width=120pt]{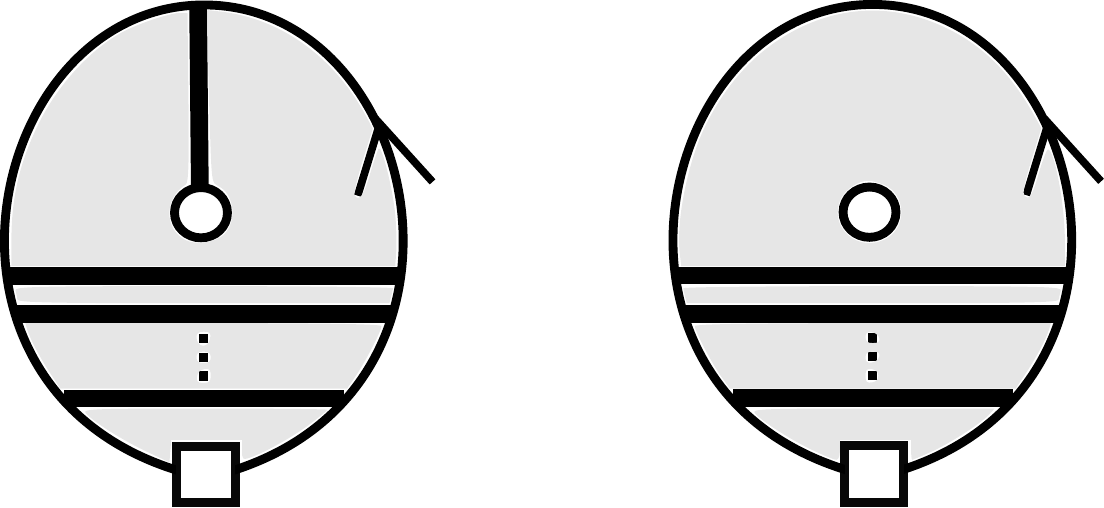}
    \caption{$b(k)$; left $k$ odd, right $k$ even. There are $\lfloor \frac{k}{2} \rfloor$ horizontal lines.} 
    \label{fig:uu_generator0}
\end{figure}

\eexa

\def\RY{{\mathrm{RY}}}
\subsection{Relation to the Roger-Yang skein algebra} \label{ssRY}

We now describe the precise relationship between our {$\SSR$} and the Roger-Yang skein algebra of \cite{RY}, which is defined only  when $\fS$ has no boundary.   The Roger-Yang skein algebra is almost identically the LRY skein algebra in this case with a specific choice of ground ring.

Assume $\fS$ has no boundary edge. Let $\cM$ be the set of circle boundary components.
 Let
$$\cR'_\fS= \cR[m^{\pm 1/2}, m\in \cM], \quad \cR_\fS= \cR[m^{\pm 1}, m\in \cM]$$
be respectively the ring of Laurent polynomials having coefficients in $\cR$ and variables $m^{1/2}$ (respectively $m$), with $m \in \cM$. 
The Roger-Yang skein algebra \cite{RY}, denoted by $\cS^\RY(\fS;\cR)$, is the $\cR_\fS$-algebra freely spanned by $\ptfS$-tangles subject to relations \eqref{eq.skein}, \eqref{eq.loop}, \eqref{eqE}',   and \eqref{eqF}. Here \eqref{eqE}' is the same as \eqref{eqE}, except the left hand side is multiplied by $m$, which is the  drawn circle component, see the first relation of Equation (\ref{eqEF}). Consider the slight extension
$$ \SRY(\fS; \cR)'  =  \SRY(\fS;\cR)\ot_{\cR_\fS} \cR'_\fS  .$$

\bpro \label{rRYRY}
 (a) There is an $\cR$-algebra  isomorphism
$$f:\SRY(\fS; \cR)' \xra{ \cong} \cS(\fS; \cR'_\fS)=\SS{\ot_\ZQ} \cR'_\fS$$
such that for all boundary ordered $\pfS$-tangle diagrams $\al$, 
\be f(\al)= (\prod_{m\in \cM}m^{-|\al \cap m|/2}) \al.
\label{eqSRY}
\ee

(b) If  {$\SSR$} is a domain then $\cS^\RY(\fS; \cR)$ is a domain.

(c) If  {$\SSR$} is Noetherian then $\cS^\RY(\fS; \cR)$ is Noetherian.

(d) If $\SS$ is orderly finitely generated, then $\SRYSR$ is orderly finitely generated.

(e) For the Gelfand-Kirillov dimension over $\cR$,
\be 
\GKdim({\SSR}) + |\cM| = \GKdim (\cS^\RY(\fS;\cR)).
\ee
\epro 
\emph{Proof}.  (a) From the defining relation it is clear that $f$ is a well-defined algebra homomorphism. By \cite{RY} the set $B(\fS)$ is a free $\cR'_\fS$-basis of $\cS^\RY(\fS; {\cR})'$, and by Theorem \ref{basisDiamond} it is a free $\cR'_\fS$-basis of $\cS(\fS, \cR'_\fS)$. It follows that $f$ is bijective.

(b) If {$\SSR$} is a domain  then by \cite[Theorem 1.2.9]{MR}, the  polynomial extension ${\SSR}\ot_\cR \cR'_\fS$ is a domain.  As  a subalgebra of a domain, $\SRY(\fS; {\cR})$ is  a domain.

(c) If {$\SSR$} is Noetherian   then by \cite[Theorem 1.2.9]{MR}, the  polynomial extension ${\SSR}\ot_\cR \cR'_\fS=\SRY(\fS {;\cR})'$ is Noetherian. 
Let $G=\{ g: \cM \to \BZ/2= \{0,1\}\}$, considered as  an abelian group. We have the $G$-gradings of rings:
\begin{align}
\cR'_\fS &= \bigoplus_{g\in G} (\cR'_\fS)_g, \quad (\cR'_\fS)_g := {\prod_{m\in \cM}} m^{g(m)/2  } \cR_\fS,\notag\\
\SRY(\fS {;\cR}) ' &=\bigoplus_{g\in G} (\SRY(\fS {;\cR})')_g, \ (\SRY(\fS {;\cR})')_g:= (\cR'_\fS)_g  \SRY(\fS {;\cR}).
\label{eqGrade2}
\end{align}
Here $g(m)\in \{0,1\}$ is considered as an element of $\BZ$.
As $\SRY(\fS {;\cR})$ is the 0-component, it is also Noetherian.

(d) Assume $\SS$ is orderly finitely generated. Then $\SRY(\fS;  {\cR})' =\SS{\ot_\ZQ} \cR'_\fS$ is 
orderly finitely generated, with a generating set $\{a_1, \dots, a_k\}$. Consider the grading \eqref{eqGrade2}. Replacing each $a_i$ by the list of all its homogeneous components, 
 we can assume that each $a_i$ is $G$-homogeneous. Assume $a_i \in  (\SRY(\fS {;\cR})')_{g_i}$.
Then  
$$b_i = \prod_{m \in \cM} m^{g_i(m)/2} a_i\in (\SRY(\fS {;\cR})')_0 = \SRYSR.$$

As each $m^{1/2}$ is invertible $\cR'_\fS$, the list $\{b_1,\dots, b_k\}$ also orderly generates $\SRY(\fS;  {\cR})'$ over $\cR_\fS'$, and we will show that it orderly generates $\SRYSR$ over $\cR_\fS$.

 Let $a\in \SRYSR$. As $a\in \SRY(\fS;  {\cR})'$ we have
\begin{align*}
a&= \sum c_{\bn} b_1^{n_1} \dots b_k^{n_k}, \quad c_\bn \in \cR'_\fS, \bn= (n_1, \dots, n_k)\in \BN^k.
\end{align*}
By taking the 0-homogeneous part,  in the above identity we can replace each $c_\bn$ by its 0-homogeneous components $(c_\bn) _0\in \cR_\fS$     This shows $\SRYSR$ is orderly finitely generated by $\{b_1, \dots, b_k\}$.

(e)
By \cite[ Corollary 8.2.15]{MR}, we have 
$$ \GKdim({\SSR}\ot_\cR \cR'_\fS) = \GKdim({\SSR}) + |\cM|.$$
Then from \cite[Proposition 8.2.9.ii]{MR}, we have that since $\SRY(\fS {;\cR})'$ is finite dimensional as a $\cS^\RY(\fS {;\cR})$-module that 
$$ \GKdim(\cS^\RY(\fS {;\cR})) = \GKdim({\SSR}) + |\cM|.\qed$$

\def\reordno{  \raisebox{-8pt}{\incl{.8 cm}{reord_no}} }

\def\reordnod{  \raisebox{-8pt}{\incl{.8 cm}{reord_nod}} }

\def\reordoneall{  \raisebox{-8pt}{\incl{.8 cm}{reord1all}} }

\def\reordoneallp{  \raisebox{-8pt}{\incl{.8 cm}{reord1allp}} }

\def\reordthree{  \raisebox{-8pt}{\incl{.8 cm}{reord3}} }
\def\Coeff{{\mathrm{Coef}}}

\subsection{Filtrations} \label{sec.fil}
\def\bqq{{\, \overset{(q)}= \, }}
Let $\cE$ be a finite set of ideal arcs and simple loops on $\fS$. For $k\in \BN$  define the $\RS$-submodule $F_k^\cE({\SSR}) \subset {\SSR}$ by
\be 
F_k^\cE({\SSR}) = \cR\text{-span of \{stated $\pfS$-tangle diagram $\al$},\quad  \sum _{e\in \cE} I(e, \al) \le k   \},
\ee
where $I(e, \al)$ is the geometric intersection number, i.e.
$$ I(e, \al) =\min\{  |\al \cap e| \quad  \mid \al' \in \ \text{isotopy {class} of} \ \al\}.$$
It is easy to see that the filtration $(F_k^\cE({\SSR})_{k=0}^\infty$ is compatible with the product, and
\be F_k^\cE({\SSR}) = F_k^\cE({\SS}) \ot_\ZQ \cR.
\ee

Assume now the components of $\cE$ are disjoint. It is  known that any $\pfS$-tangle diagram $\al$ can be isotoped to a {\bf taut}  position with respect to $\cE$, $\al'$ {so that} 
$$ |\al' \cap e| = I(\al', e) \quad \text{for all} \ e\in \cE .$$

From Theorem \ref{basisDiamond} we have
\bpro \label{r.basisGr}
The  set 
$B^\cE_{k}(\fS): = \{ \al \in B(\fS) \mid \sum _{e\in \cE} I(e, \al) \le k\}$ is  a free $\RS$-basis of $F_k^\cE({\SSR}) $. Hence the lead term maps $B(\fS)$ bijectively onto a free $\cR$-basis of $\Gr^\cE({\SSR})$.
\epro
The following will be useful.
\bpro 
\label{r.switch2} Let $\cE$ {be a} collection of boundary edges.
Suppose $\al$ is a stated  boundary ordered $\pfS$-tangle diagrams with $\sum_{e\in \cE} |\al \cap e|= k$,
 and $\al'$ is the result of changing arbitrarily the height order of $\al$. Then 
\be  \al \bqq \al' \mod F^\cE_{k-1}(\SS).
\label{eq.ex1}
\ee
In particular if $\al= \al_1 \sqcup \al_2$ then
\be  \al \bqq \al_1 \al_2 \mod F^\cE_{k-1}(\SS).
\label{eq.ex2}
\ee
\epro
\bpr

By the height exchange formula of \cite[Lemma 2.4]{CL}, for $\nu, \nu'\in \{\pm\}$,
\be
\label{eq.sign}
\begin{array}{c}\includegraphics[scale=0.7]{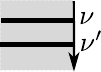}\end{array} =
 q^{\nu-\nu'-1}\,  \begin{array}{c}\includegraphics[scale=0.7]{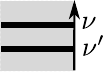}\end{array} + \delta_{\nu, -\nu'} \, \Coeff\,  \begin{array}{c}\includegraphics[scale=0.7]{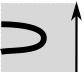}\end{array}, \quad \Coeff \in \Zq.
 \ee
 From here we have \eqref{eq.ex1}.  When $\al= \al_1 \sqcup \al_2$ the product $\al_1 \al_2$ is obtained from $\al$ by changing the height order. Hence we have \eqref{eq.ex2}.
\epr

\subsection{Grading} \label{sec.grade}
Let $e$ be a boundary edge of $\fS$. For a stated $\pfS$-tangle diagram $\al$ let
\begin{align*}
\deg_e(\al) &:=\sum _{u \in \pal \cap e} s(u) \in \BZ \\
\end{align*}

Since $\deg_e$ is preserved by the  relations \eqref{eq.skein}-\eqref{eqF}, it descends to a grading of the $\RS$-algebra
\be 
{\SSR} = \bigoplus_{k \in \BZ} {\SSR}_{\deg_e=k}.
\label{eq.deg1}
\ee
 where
 $$ {\SSR}_{\deg_e=k}:= \RS\text{-span of \{stated $\pfS$-tangle diagram $\al$},\quad  \deg_e(\al) = k   \}.$$

 \def\sS{{\mathscr S}}
\def\bSS{{\overline \sS(\fS)}}
\def\bSSR{{\overline \sS(\fS;\cR)}}

\subsection{Reduced LRY {skein} algebras and bad arcs} We define now the reduced version of ${\SSR}$, following \cite{CL}. 

 A {\bf corner arc} is a non-trivial $\pfS$-arc which cuts out a boundary puncture, as in {Figure}~\ref{fig:badarc}. A {\bf bad arc} is a stated corner arc with states  given in Figure~\ref{fig:badarc}.

\begin{figure}[htpb!]
    \centering
    \includegraphics[scale=.25]{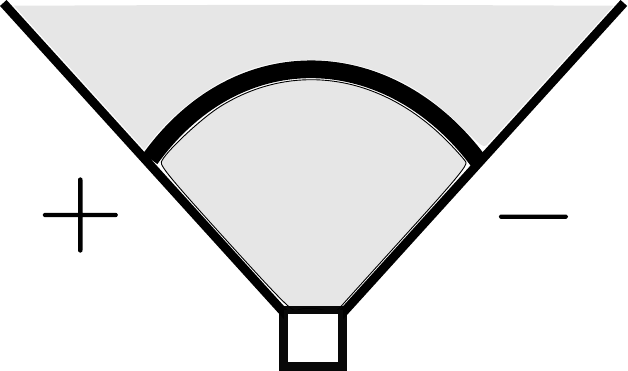}
    \caption{An example of a bad arc.}
    \label{fig:badarc}
\end{figure}

The {\bf reduced LRY skein algebra} is defined as
\[\bSS:=\SS/\mathcal{I}, \quad \bSSR := \bSS \ot_\ZQ \cR,\]
where $\mathcal{I} \subset {\SSR}$ is the two sided ideal generated by bad arcs. Since $\cI$ is invariant under the reflection $\omega$ defined in Subsection \ref{ssLRY}, $\omega$ descends to a reflection $\bar \omega$ of $\bSS$. 

 Following an identical argument to Theorem $7.1$ of Costantino-L\^e \cite{CL}, using height exchange relations on the boundary we have the following result.
\begin{theorem}
The set $\bar B(\fS)$ of isotopy classes of increasingly stated, positively ordered, simple diagrams  having no bad arcs is a free  $\cR$-basis for $\cR$-module {$\bSSR$}.
\end{theorem}

\subsection{Cutting homomorphism}\label{sec.cut}

We now present a main feature of our stated skein algebra $\SSR$: the cutting homomorphism.

Let $c$ be an ideal arc in the interior of a punctured surface $\fS $. The cutting $\Cut_c(\fS)$ is a punctured  surface having two boundary edges $c_1, c_2$ such that $\fS= \Cut_c(\fS)/(c_1=c_2)$, with $c=c_1=c_2$.

 Assume $\pfS$-tangle diagram $D$ is transverse to $c$. Let $h$ be a linear order on the set $D \cap c$. Let $p: \Cut_c(\fS) \to \fS$ be the natural projection map.
For a map $s: D \cap c \to \{\pm \}$, let $(D,h,s)$ be the stated tangle diagram over $\Cut_c(\fS)$ which is $p^{-1}(D)$ where the height order on $c_1 \cup c_2$ is induced (via $p$) from $h$, and the states on $c_1 \cup c_2$ are induced (via $p$) from $s$.

\begin{theorem}
\label{t.splitting2}  
Suppose $c$ is an interior ideal arc of a punctured  surface $\fS$ and $\cR$ is a ground ring.

(a)   There is a unique $\RS$-algebra embedding
\[\Theta_c\colon {\SSR} \embed \mathscr{S}(\Cut_c(\fS);{\cR})\]
such that if  $D$ is a stated $c$-transverse tangle diagram of a stated tangle $\al$ over $\fS$ and $h$ is any linear order on $D \cap c$, then
\begin{equation}\label{eq.cut00}
\Theta_c(\al) =\sum_{s: D \cap c \to \{\pm \}} (D, h, s).
\end{equation}

(b)  The  homomorphism $\Theta_c$ descends to an algebra embedding of reduced LRY skein algebras
\[\bar{\Theta}_c:{\bSSR})\embed \overline{\mathscr{S}}(\Cut_c(\fS);{\cR}).\]

% Additionally $\theta_c$ is an embedding of algebras and given $c_1$ and $c_2$ we have $\theta_{c_1}\circ \theta_{c_2}=\theta_{c_2}\circ \theta_{c_1}$. 

\end{theorem}

The proof of part (a) is identical to the similar theorem for the stated skein algebra, 
 Theorem $3.1$ of \cite{Le:triangular}, as the proof concerns only the boundary relations. Similarly, the proof part (b) is identical to the proof of the similar result, 
 Theorem 7.6 of \cite{CL}.

 \section{Ideal triangulations}
 \label{secTrig}
  \def\PP{{\mathbb P}}
  
  We will provide definitions and conventions concerning ideal triangulations. To each triangulation  we associate two antisymmetric integral matrices, giving rise to two quantum tori which are the target spaces of the quantum traces that will be constructed later. 
  We also show how to use ideal triangulations to parameterize the basis $B(\fS)$ of the skein algebra ${\SSR}$.

\subsection{Ideal triangulations}  We now define ideal triangulations of triangulable surfaces.

\bdf (a) An $n$-gon is the result of removing $n$ points from the boundary of the standard oriented closed disk. % When $n=3$ we call the $3-$gon an ideal triangle.

(b) A punctured surface $\fS$ is triangulable if each connected component of it has at least one puncture and is not one of the following  exceptional surfaces:

\begin{enumerate}

\item A sphere with $ \le 2$ punctures,
\item The $n$-gon with $n=1$ or $2$.
\end{enumerate}

 (c) An ideal triangulation, or simply a triangulation, of  a triangulable surface $\fS$ is  a maximal collection $\Delta$ of non-trivial ideal arcs which are pairwise disjoint and non-isotopic.
 
 (d) A triangulated surface $(\fS,\Delta)$ is a triangulable surface $\fS$ equipped with a triangulation $\Delta$.

\edf

%Usually we consider triangulation up to isotopies. 

It should be noted that the {\bf 1-marked monogon} $\fm$, which by definition is the  annulus $S^1 \times [0,1]$ with one puncture in $S^1 \times \{0\}$ removed, is considered triangulable by our definition. The only triangulation of $\fm$, up to isotopy,  consists of the  boundary edge.

An ideal arc of $\fS$ is {\bf boundary} if it is isotopic to a boundary edge. Every boundary edge is isotopic to exactly one element of the triangulation $\Delta$, since we already excluded the bigon. Let $\Dd\subset \Delta$ denote the subset of all boundary elements. Then $\Do= \Delta\setminus \Dd$ is the set of interior edges.

Cutting the surface along edges in  $\Do$  (see the definition of cutting in Subsection \ref{sec.cut}), we get a collection $\cF(\Delta)$ of {\bf faces of the triangulation}, where each face is either a triangle or a 1-marked monogon, see Figure \ref{fig:TrigMon}.

\begin{figure}[htpb!]
    \centering
    \includegraphics[scale=.25]{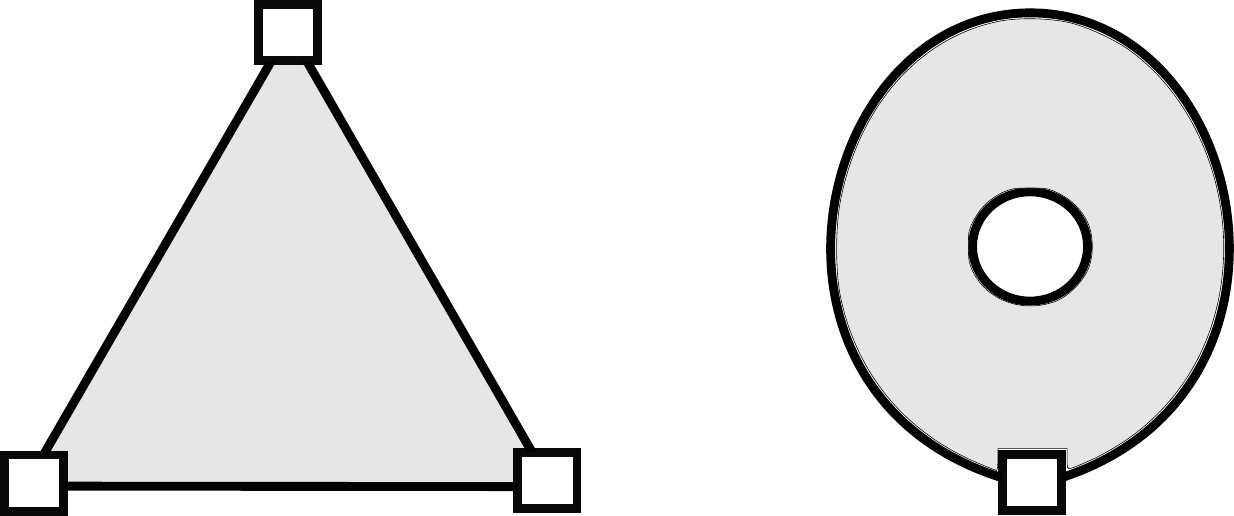}
    \caption{}
    \label{fig:TrigMon}
\end{figure}

There is a natural  projection
$\pr: \bigsqcup_{\tau\in\cF(\Delta)}\tau \onto \fS$, which identifies certain pairs of edges of the faces. If $a,b,c$ are the edges of a triangle $\tau$ in $\cF(\Delta)$, then we call $(\pr(a), \pr(b), \pr(c))$ a {\bf triangular triple}. 
Two of the edges may coincide. %, and the repeated edge is called a \emph{self-folded edge}.

\subsection{Quantum tori associated to triangulations} \label{secXtori}

\newcommand{\facedef}[2]{
\begin{tikzpicture}[scale=0.8,baseline=0.28cm]
\draw[fill=gray!20!white] (0,1)--(0.6,0)--(1,1);
\draw[inner sep=0pt] (0.1,0.5)node{\vphantom{$b$}#1} (1,0.5)node{\vphantom{$b$}#2};
\draw[fill=white] (0.6,0)circle(2pt);
\end{tikzpicture}
}

We now define the quantum $X$-torus associated to a triangulation. These spaces will serve as the target spaces of the quantum traces.
Fix a triangulated surface $(\fS,\Delta)$.

Let $\bQ=\bQ_\Delta: \Delta\times \Delta \to \BZ$ be the anti-symmetric function defined by
\begin{equation}\label{eq.Q}
\bQ_\Delta(a,b) = \#\left( \begin{array}{c}\includegraphics[scale=0.16]{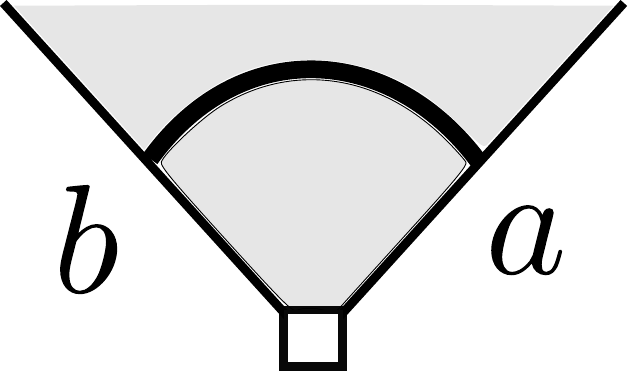}\end{array} \right)- \#\left( \begin{array}{c}\includegraphics[scale=0.16]{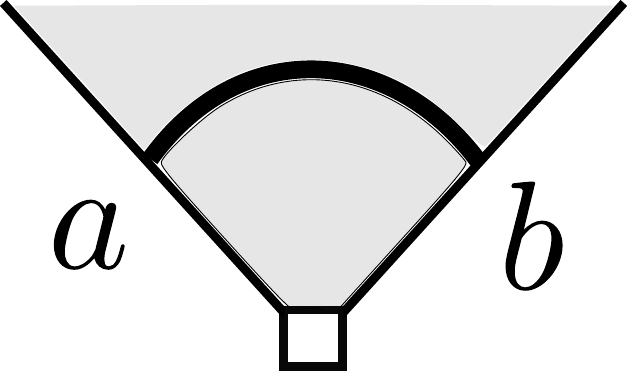}\end{array} \right).
\end{equation}
Here each shaded part is a corner of an ideal triangle. Thus, the right hand side of \eqref{eq.Q} is the number of corners where $b$ is counterclockwise to $a$ minus the number of corners where $a$ is clockwise to $b$, as viewed from the puncture. 

\def\bD{\DD}
\def\bx{{\bar x}}
\def\sXo{{{\sX}^\diamond}}

Let $\mQ_\Delta$ be the double of $\bQ_\Delta$ along $\Dd$. Explicitly, 
let $\hDd = \{ \hat e \mid e \in \Dd\}$ be a copy of $\Dd$ and
 $ \DD=\Delta\sqcup \hDd$. Then $\mQ_\Delta:\DD \times \DD \to \BZ$ is the extension of $\bQ_\Delta$, defined  so that the values of $\mQ_\Delta$ on the extension set $(\bD \times \bD) \setminus (\Delta\times \Delta)$ are 0 except
\[\mQ_\Delta(\hat e , e)=1=-\mQ_\Delta(e,\hat e) \text{ for all } e\in \Dd.\]
The geometric origin of $\bQ_\Delta$ will be given in Subsection~\ref{sec.extended}.

Let $\bsX(\fS,\Delta)$ and $\sX(\fS,\Delta)$ be the quantum $\ZQ$-tori associated to $\mQ$ and $\bQ$:
 \begin{align}
 \bsX(\fS,\Delta) &= \ZQ\la \bx_e^{\pm 1}, e \in \Delta \mid \bx_e \bx_{e'} = q^{ \bQ_\Delta(e,e') } \, \bx_{e'} \bx_{e'}\ra \\
 \sX(\fS,\Delta) &= \ZQ\la x_e^{\pm 1}, e \in \DD \mid x_e x_{e'} = q^{ \mQ_\Delta(e,e') } \, x_{e'} x_{e'}\ra ,
 \end{align}
 which are algebras with reflection. Let
  $\bsX(\fS,\Delta{;\cR})$ and $\sX(\fS,\Delta{;\cR})$ be the versions over $\RS$:
$$ \bsX(\fS,\Delta{;\cR}) = \bsX(\fS,\Delta) \ot_\ZQ \cR ,    \quad  \sX(\fS,\Delta{;\cR}) = \sX(\fS,\Delta) \ot_\ZQ \cR.  $$

There is an obvious embedding $\bsX(\fS,\Delta{;\cR})  \embed \sX(\fS,\Delta{;\cR}) $ given by $\bx_e \to x_e$ for all $e\in \Delta$.
 
 However,  there is no algebra projection $\sX(\fS,\Delta)  \onto \bsX(\fS,\Delta{;\cR}) $. So we will consider the $\RS$-subalgebra 
   $\sXo(\fS,\Delta{;\cR}) \subset  \sX(\fS,\Delta{;\cR})$  generated by $\{x_e^{\pm 1}, e \in \Delta\}$ and $\{x _{e}, e\in \hDd\}$. Note that the powers of $x_e$ with $e\in \hDd$ must be non-negative in $\sXo(\fS,\Delta{;\cR})$.

   There is an algebra surjective homomorphism $\pr: \sXo(\fS,\Delta{;\cR})\onto \bsX(\fS,\Delta{;\cR})$ given by 
 \be 
  \pr( x_e) = \begin{cases} \bx _e \quad &\text{if} \ e \in \Delta \\
 0 &\text{if} \ e \in \hDd.
 \end{cases}   
 \label{eq.pr1}
 \ee

 \brem (a) When $\fS$ has no circle boundary component,  the form $\bQ$, known as the Fock form, is related to the Weil-Petersson Poisson structure of the enhanced (or holed) Teichm\"uller space in shear coordinates, see e.g. \cite{Fock,BW,FG}. The matrix $\bQ_\Delta$ is called the face matrix in \cite{Le:Qtrace}. Also in this case $\bsX(\fS,\Delta{;\cR})$ is the square root version of the Chekhov-Fock algebra \cite{BW,CF}, or more generally the quantum $X$-space of Fock and Goncharov \cite{FG2}.

(b) When $\fS$ has no circle boundary component, the form $\mQ$ was introduced in \cite{LY2}. The generalization given here is straightforward.
\erem
\subsection{Cutting a triangulation} We show how the quantum torus $\bsX(\fS,\Delta{;\cR})$ behaves under the operation of cutting the surface along an edge in $\Do$. 
 
Let $e\in \Do$ be an interior edge. The cut surface $\Cut_e\fS$, with two boundary edges $e',e''$ which glue back to give $e$, has the triangulation $\Cut_e\Delta= \Delta \cup \{e', e''\} \setminus \{e\} $. From the definition of $\bQ_\Delta$ we have
\bpro 
\label{r.cuttri}
There is an embedding $\bar \Theta_e\colon\bsX(\fS, \Delta{;\cR}) \embed \bsX(\Cut_e\fS, \Cut_e\Delta{;\cR})$ given by 
$$ \bar \Theta_e(\bx_a)= \begin{cases} \bx_a \quad &\text{if} \ a\neq e\\
[\bx_{e'} \bx_{e''}]_\Weyl & \text{if} \ a=e.
\end{cases}
$$
Under this embedding $\bsX(\fS, \Delta)$ is equal to the monomial subalgebra of $\bsX(\Cut_e\fS, \Cut_e\Delta{;\cR})$ which is $\RS$-spanned by monomials $\{\bx^\bk  \mid \bk\colon\Cut_e\Delta\to \BZ,\ \bk(e') = \bk(e'')\}$.
\epro
 
 \brem There is no simple cutting operation for the full torus $\sX(\fS,\Delta{;\cR})$.

 \erem

 \subsection{Triangulation Coordinates
}
\label{marktricoord}  Given a triangulation $\Delta$, we show how to parameterize the set $B(\fS)$, the $\RS$-basis of the $\RS$-module ${\SSR}$, using  a submonoid of $\BZ^r$.

Recall that $B(\fS)$  consists of isotopy classes of increasingly stated, positively ordered, simple diagrams  on  $\fS$. For $\al\in B(\fS)$ define $\bn_\al: \DD \to \BN$ by
\be \begin{cases}
{\bn}_\alpha(e):=I(\alpha,e) & \text{if} \ e\in \Delta \\
{\bn}_\alpha(\hat{e})=2 (\# \{ \text{negative states on} \ \pal \cap e \} )& \text{if} \ e\in\Delta_\partial
\end{cases}.
\label{eq.bbn}
\ee

\def\LD{{\Lambda_\Delta}}

\begin{theorem} \label{thm.tri-coord}
Let $\Delta$ be an ideal triangulation of a triangulable  punctured surface $\fS$.  

(a) The map $\nu: B(\fS) \to \BN^\DD$ given by $\nu(\al)= \bn_\al$, maps $B(\fS)$ bijectively onto the  submonoid $\Lambda_\Delta\subset \mathbb{Z}^{\DD}$ consisting of $\bn\in\mathbb{N}^{\DD}$ such that 

\begin{enumerate}
 \item for any triangular triple $a,b,c$,  one has $\bn(a)+\bn(b)+\bn(c)\in 2\mathbb{N}$ and $\bn(a)\leq \bn(b)+\bn(c)$,
    \item for $e\in\Delta_\partial$ one has  
    $\bn(\hat{e})\in 2\mathbb{N}$ and 
    $\bn(\hat{e})\leq 2\bn(e)$.

\end{enumerate}

(b) The submonoid $\LD\subset \BN^\DD$ is a finitely generated $\BN$-module and the free abelian group generated by $\LD$ (in $\BZ^\DD$) has rank 
\be r(\fS)=|\DD| =2 \abs{\Pd}+ 2 |\cM|-3 \chi(\fS), \label{eqr}\ee where $\chi(\fS)$ is the Euler characteristic of $\fS$ and $\abs{\Pd}$ is the number of components of $\Pd$.

\end{theorem}
\def\he{{\hat e}}
\def\bbn{\bn}

\begin{proof} When $\cM=\emptyset$ this is \cite[Proposition 6.3]{LY2}, whose proof is  modified  here.

If $\al\in B(\fS)$ it is easy to show that $\bn_\al\in \Lambda_\Delta$.

Let $\bn\in \Lambda_\Delta$. We will show there is a unique $\al\in B(\fS)$ such that $\bbn_\al= \bn$. Let $\fS'$ be the result of cutting off all the 1-marked monogon faces from $\fS$. It is well known \cite{Matveev} that (1) implies that there is a unique (non-stated) simple $\pfS'$-tangle diagram $\al'$ such that $I(\al', a) = \bn(a)$ for $a\in \Delta$. In each 1-marked monogon $\tau$ with boundary edge $e$ let $\al_\tau$ be $b(\bn(e))$  of Figure \ref{fig:uu_generator0}, which has $\bn(e)$ points on $e$.

%The number of $+$ and $-$ states of $\al\cap e$ are also uniquely determined by $\bn(\he)$ given (b), and the states are completely fixed for increasingly stated diagrams. Thus for every $\bn\in \Lambda_\Delta$, there is a unique $\al\in B(\fS)$ such that $\bbn_\al= \bn$. 

We can patch all $\al_\tau$, for all 1-marked monogon faces $\tau$, together with $\al'$ to get a simple $\pfS$-tangle diagram $\al$, where the height order is the positive one. The number $\bn(\hat e)$ tells us how to uniquely state $\pal$:  on each boundary edge $e$, the first lowest  $\bn(\hat e)/2$ boundary points  have negative states, while the remaining ones have positive states.  This shows $\nu$ is bijective.

(b) As $\LD\subset \BZ^\DD$ is defined by a finite number of rationally linear inequalities, it is a finitely generated $\BN$-module. 

To compute the rank, we show that the group generated by $\Lambda_\Delta$ contains $(2\mathbb{Z})^\DD$. Let $\mathbf{2}\in\mathbb{Z}^{\DD}$ be the constant map $2$, and $\mathbf{d}_a\in\mathbb{Z}^{\DD}$ be the indicator function on $a\in\DD$. It is clear that $\mathbf{2}$ and $\mathbf{2}+2\mathbf{d}_a$ are in $\Lambda$ for all $a\in\DD$. Thus the difference $2\mathbf{d}_a$ is in the group generated by $\Lambda_\Delta$, and they span $(2\mathbb{Z})^{\DD}$.  Then the final equality is a standard Euler characteristic argument.
\end{proof}

 \section{Quantum traces for triangulated surfaces}
 \label{secQtraces}

 For a triangulation $\Delta$ of a punctured surface $\fS$ we will construct two quantum traces. The first embeds the reduced skein algebra ${\bSSR}$ into the quantum torus $\bsX(\fS,\Delta{;\cR})$, while the second embeds ${\SSR}$ into the quantum torus $\sX(\fS,\Delta{;\cR})$.

 We will show that ${\SSR}$ is orderly finitely generated and calculate the Gelfand-Kirillov dimension of ${\SSR}$. Additionally, we show that ${\SSR}$ and ${\bSSR}$, for any punctured surface (not necessarily triangulable), are Noetherian if the ground ring $\cR$ is.

\subsection{Quantum traces}  \label{ssQtraces}
\def\btr{{\overline{\tr}}}
We formulate the main results about the existence of quantum traces and their consequences. Recall that $\bSS$ and $\bsX(\fS,\Delta) $ are $\ZQ$-algebras with reflection.

\bthm   \label{thm.embed3}

Suppose $(\fS,\Delta)$ is a triangulated surface and the ground ring $\cR$ is a commutative $\ZQ$-domain. 
There exists a reflection invariant algebra homomorphism 
$$ \btr_\Delta: \bSS \to \bsX(\fS,\Delta)$$
with the following properties. 

(a)  The $\cR$-algebra homomorphism 
$$\btr_{\Delta; \cR}:= \btr_\Delta\ot_\ZQ \cR : \bSSR \to   \bsX(\fS,\Delta; \cR)$$
is injective. Consequently $\bSSR$ is a domain.

(b) The map $\btr_{\Delta; \cR}$ is compatible  with the cutting homomorphism in the following sense: For an edge $e\in \Do$, the diagram
\be
\begin{tikzcd}
{\bSSR}\arrow[r,hook,"\btr_{\Delta; \cR}"]
\arrow[d,"\bar \Theta_e"]  
& \bsX(\fS,\Delta{;\cR})
\arrow[d,"\bar \Theta_e"] \\
\bcS   (  \Cut_e\fS {;\cR})\arrow[r,hook,"\btr_{\Cut_e\Delta; \cR}"] & \bsX(\Cut_e\fS,\Cut_e\Delta{;\cR})
\end{tikzcd}
\label{eq.dia9}
\ee

\ethm

To formulate the result for the unreduced quantum trace recall that
 we  defined algebra filtrations $(F^\Delta_k({\SSR}))_{k=0}^\infty$ in  Subsection \ref{sec.fil}: 
 $$F^\Delta_k({\SSR})= \RS\text{-span of } \ \{\text{stated $\pfS$-tangle diagrams $\al$ with}\  \sum _{e\in \Delta} I(e, \al) \le k\}.$$

The quantum torus $\sX(\fS,\Delta{;\cR})$ has the algebra filtration  $(F_k^\Delta(\sX(\fS,\Delta{;\cR})))_{k=0}^\infty$ where
$$ F_k^\Delta(\sX(\fS,\Delta{;\cR}))= 
 \RS\text{-span of } \ \{ x^\bk \mid  \bk \in  \BZ^{ \tD  }, \  \sum_{e\in {\Delta}} \bk(e) \le k\}.
 $$
The associated graded algebra $\Gr^\Delta(\sX(\fS,\Delta{;\cR}))$ is canonically isomorphic to the quantum torus $\sX(\fS,\Delta{;\cR})$, and we identify them accordingly.

\bthm   \label{thm.embed3e}  Suppose $(\fS,\Delta)$ is a triangulated surface and the ground ring $\cR$ is a commutative $\ZQ$-domain. 
 There exists a reflection invariant algebra homomorphism 
$$ \tr_\Delta: \SS \to \sXo(\fS,\Delta) \subset \sX(\fS,\Delta)$$
with the following properties.

(a) The following $\cR$-algebra homomorphism is injective:
$$\tr_{\Delta; \cR}:= \tr_\Delta\ot_\ZQ \cR : \SSR \to   \sXo(\fS,\Delta; \cR).$$
Consequently $\SSR$ is a domain. Moreover, the following diagram is commutative
\be
\begin{tikzcd}
{\SSR}\arrow[r,hook,"\tr_{\Delta;\cR}"]
\arrow[d,"\pr"]  
& \sXo(\fS,\Delta{;\cR})
\arrow[d,"\pr"] \\
{\bSSR}\arrow[r,hook,"\btr_{\Delta;\cR}"] & \bsX(\fS,\Delta{;\cR})
\end{tikzcd}
\ee

(b)
The quantum trace $\tr_{\Delta;\cR}$ respects the filtrations $F^\Delta$ in the sense that for $k\in \BN$,
\be 
\tr_{\Delta;\cR} (F^\Delta_k({\SSR})  ) \subset  F^\Delta_k(\sX(\fS, \Delta {;\cR})).
\label{eq.in5}
\ee
Moreover, its associated  graded homomorphism gives an algebra isomorphism
\be  \Gr(\tr_{\Delta;\cR}): \Gr^\Delta({\SSR}) \xrightarrow{\cong} \bT(\mQ,\Lambda_\Delta{;\cR})
\ee
 where $\bT(\mQ,\Lambda_\Delta{;\cR})$ is the monomial subalgebra associated to the 
 monoid $\Lambda_\Delta$ of Theorem~\ref{thm.tri-coord}.

(c) The GK dimension of ${\SSR}$ over $\RS$ is $r(\fS)$
%$=|\DD| = 2 \abs{\Pd}+ 2 |\cM|-3 \chi(\fS)$, 
defined by \eqref{eqr}.

\ethm

The proofs of  Theorems \ref{thm.embed3} and \ref{thm.embed3e} are given in Subsections \ref{ssProof1}-\ref{ssProofe}. Here is a  corollary.

\begin{theorem}\label{thm-noether1}
Let $\fS$ be any punctured surface (not necessarily triangulable).
\begin{enumerate}
\item As an $\cR$-algebra, $\cS(\fS)$ is orderly finitely generated.
\item If $\cR$ is Noetherian, then ${\SSR}$ and ${\bSSR}$  are Noetherian domains.
\end{enumerate}
\end{theorem}

\bpr First assume that $\fS$ is triangulable, with a triangulation $\Delta$.

  By Theorem \ref{thm.embed3}(b), there is an algebra $\BN$-filtration of ${\SSR}$ such that the associated graded algebra is the monomial algebra $\bT(\mQ, \Lambda_\Delta{;\cR})$. 
  Since $\Lambda$ is a finitely generated $\BN$-module by Theorem~\ref{thm.tri-coord}, the algebra $\bT(\mQ, \Lambda_\Delta{;\cR})$ is orderly finitely generated by Lemma~\ref{r.mono}. By Proposition~\ref{liftfacts}, the algebra ${\SSR}$ is orderly finitely generated.
  
  Suppose  $\cR$ is Noetherian.  Then $\bT(\mQ, \Lambda_\Delta{;\cR})$ is Noetherian by Lemma \ref{r.mono}. By Proposition~\ref{liftfacts}, the algebra ${\SSR}$ is Noetherian.

Now assume $\fS$ is not triangulable. By removing a few punctures from $\fS$ we get a new punctured surface $\fS'$ which is triangulable. The embedding $\fS' \embed \fS$ induces a surjective algebra homomorphism $\cS(\fS'{;\cR}) \onto {\SSR}$. Since both orderly finite generation and Noetherianity are preserved under projections, the case of $\fS$ follows.

As ${\bSSR}$ is a quotient of ${\SSR}$, the statements concerning ${\bSSR}$ also hold true.
\epr

  \brem(a) When $\pfS=\emptyset$ we have $\btr_\Delta= \tr_\Delta$, and the map was constructed by Bonahon and Wong \cite{BW}. 
  
  (b) When $\cM=\emptyset$  Theorems \ref{thm.embed3} and  \ref{thm.embed3e} were proved in \cite{LY2}. 
  
 (c) The reduced quantum trace $\btr_\Delta$ is not injective on ${\SSR}$ while the unreduced one $\tr_\Delta$ is. The unreduced quantum trace helps to establish many properties of ${\SSR}$. Later when we define the boundary reduced skein algebra to prove the main theorems, it is the unreduced quantum trace that descends to this new version of skein algebra, but not the reduced quantum trace.  
  \erem

 \def\eqbu {{\qeq }}
\subsection{On the proof of Theorem \ref{thm.embed3}} \label{ssProof1}
The proof  is  parallel to that of the similar case when $\fS$ has no circle boundary component considered in \cite{Le:triangular,CL,LY2}. We first cut the surface into faces, construct the reduced quantum trace for each face, then patch the quantum traces of all faces to get the quantum trace for $\fS$. The new problem here is besides triangle faces we will also have  1-marked monogon faces. We will construct the reduced quantum trace for 1-marked monogon, then proceed similarly as in \cite{LY2}.

\no{
For two elements $x,y$ in an $\RS$-algebra we write
$ x \eqbu y$ if 
$$ x = (q^k v_1^{k_1} \dots v_r ^{k_r}) \, y, \ \text{where} \ v_1, \dots, v_r \in \cM;\  k,k_1, \dots, k_r\in \BZ.$$
}

\def\fT{{\mathfrak T}}
\def\bST{{\overline{\mathscr{S}}(\ft)}}
\def\bSTR{{\overline{\mathscr{S}}(\ft;\cR)}}
\def\BAT{{\bT^A(\ft)}}
\def\BXT{{\overline {\sX}(\ft)}}
\def\BXTR{{\overline {\sX}(\ft;\cR)}}
\def\btr{{ \overline{\tr}  }}
\subsection{Quantum trace for an ideal triangle} Here we recall the reduced quantum trace for triangles, already considered in \cite{Le:triangular,CL,LY2}.

Let $\ft$ be the ideal triangle, with boundary edges $a,b,c$ in counterclockwise order as in Figure \ref{fig:triangle}. Let  $\alpha,\beta,\gamma$ be the stated $\partial \ft$-arc given in \ref{fig:triangle}, with states $+$ at all endpoints. 

\begin{figure}[h]
    \centering
    \includegraphics[width=80pt]{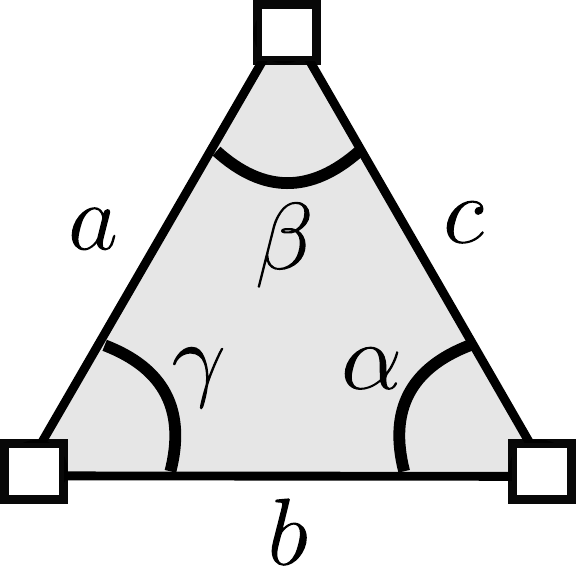}
    \caption{The ideal triangle $\ft$}
    \label{fig:triangle}
\end{figure}
Since $\ft$ has a unique triangulation, $\Delta=\{a,b,c\}$, we will drop $\Delta$ from the notation $\bsX(\ft,\Delta{;\cR})$. By definition, $\BXTR= \BXT\ot_\ZQ \cR$, where
\begin{align}
\BXT&=\ZQ\langle \bx_a^{\pm},\bx_b^{\pm},\bx_c^{\pm}\rangle/(q\bx_ a\bx_b=\bx_b\bx_a,q\bx _b\bx _c=\bx_c\bx_b,q\bx_c\bx_a=\bx_a\bx_c)
\end{align}
Recall that $\deg_e(D)$ is defined in Subsection \ref{sec.grade}.

%%%%%%%%%%%%%

\begin{theorem}[Theorem $7.11$ of \cite{CL} and  Lemma 6.10 of \cite{LY2}] \hfill
\label{thm.trtri}

(a) There is a unique reflection invariant $\ZQ$-algebra embedding $\btr_\ft\colon {\bST\embed \BXT}$ such that 
$$  \btr_\ft (\al) = [\bx_b\bx_ c]_\Weyl, \btr_\ft (\beta) = [\bx_ a\bx_ c]_\Weyl, \btr_\ft (\gamma) = [\bx_ a\bx_ b]_\Weyl.$$

Additionally, the map $\btr_{\ft, \cR}:= \btr_\ft \ot_\ZQ \cR$ is injective for any ground ring $\cR$.

(b) If $\al$ is 
stated $\partial \ft$-tangle diagram, then there exists $C\in \ZQ$ such that
\be 
\btr_\ft(\al) = C \,  (\bx_a)^{ \deg_a \al } \, (\bx_b)^{ \deg_b \al }\, (\bx_c)^{ \deg_c \al}. 
\ee

(c) If $\al$ is either (i) a product of non-bad stated $\partial \ft$-arcs, or (ii) a disjoint union of positively stated corner arcs, then
\be 
\btr_\ft(\al) \qeq  (\bx_a)^{ \deg_a \al } \, (\bx_b)^{ \deg_b \al }\, (\bx_c)^{ \deg_c \al}.
\label{eq.deg8a}
\ee

\end{theorem}

\def\pfm{{\partial \fm}  }
\def\bSM{{ \overline{\sS}(\fm)}}
\def\bSMR{{ \overline{\sS}(\fm;\cR)}}
\def\bXM{{ \overline{\sX}(\fm)}}
\def\bXMR{{ \overline{\sX}(\fm;\cR)}}
\subsection{Quantum trace for 1-marked monogon $\fm$}

 In this Subsection  we construct the reduced quantum trace for the  1-marked monogon $\fm$, by giving an explicit  presentation of the algebra $\cS(\fm{;\cR})$.

 Let $e$ be the boundary edge of $\fm$ and $u_\pm$ be the stated $ \pfm$-arc depicted in Figure \ref{fig:monogon}, with state $\pm$. Since $\fm$ has a unique triangulation, $\Delta=\{e\}$, we will drop $\Delta$ in notations.
From the definition, ${\bXM}= \ZQ[\bx_e^{\pm 1}]$, which is commutative.

\begin{figure}[h]
    \centering
    \includegraphics[width=50pt]{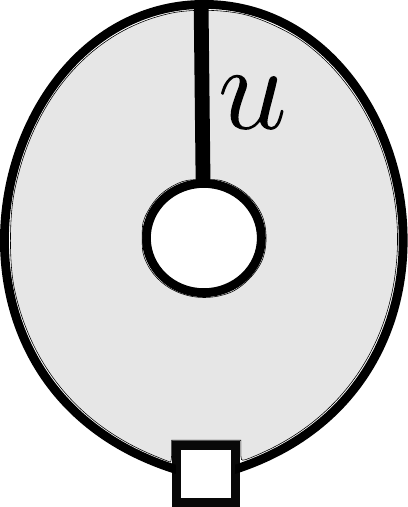}
    \caption{}
    \label{fig:monogon}
\end{figure}

\def\bSm{{\overline{\cS}(\fm)}}
\def\bSmR{{\overline{\cS}(\fm;\cR)}}
%Recall that we define $\deg_e(\al)$ in Subsection \ref{sec.grade}.
\begin{theorem} [Reduced quantum trace for 1-marked monogon]\label{thm.trmon} \hfill

(a) There is a unique reflection invariant $\ZQ$-algebra isomorphism  
\be 
 \btr_\fm\colon {\bSm\to \bXM}, \quad \btr_\fm(u_+) = \bx_e, \ \btr_\fm(u_-) = (\bx_e)^{-1}.
  \label{eq.trmon}
  \ee

(b) If $\al\in {\bSM}$ is  a 
stated $\partial \fm$-tangle diagram then there exists $C \in \ZQ$ such that
\be 
\btr_\fm(\al) = C \, (\bx_e)^{\deg_e \al  }. 
\label{eq.deg9}
\ee

(c) If $\al$ is a boundary ordered positively stated simple diagram with $k$ endpoints on $e$, then
 \be 
\btr_\fm(\al) \qeq (\bx_e)^{k  }.  \label{eq.deg9a}
\ee

\end{theorem}
For the proof we first find a presentation of $\cS(\fm)$.

\begin{theorem}\label{monogonpres}
The skein algebra $\cS(\fm)$ has the following presentation
\be 
\cS(\fm)=\ZQ\langle u_+,u_- \mid   qu_+ u_- -q^{-1} u_- u_+= q- q^{-1}\rangle.
\label{eq.present1}
\ee
\end{theorem}

\def\Sm{{\cS(\fm)}}
\def\SmR{{\cS(\fm;\cR)}}
\begin{proof} Let $Y$ be the algebra given by the right hand side of \eqref{eq.present1}, with $u_+, u_-$ replaced respectively by $y_+, y_-$.

Let $F_k(\Sm)\subset \Sm$ be the $\ZQ$-submodule spanned by stated $\pfm$-tangle diagram   $\alpha$ with $|\partial \alpha\cap e|\leq k$.
 By Theorem \ref{basisDiamond} and Proposition \ref{r.basisGr}, the sets
$$ B(\fm) = \{ b(k,l) \mid 0\le l\le k \in \BN \}, \ B_k: =\{ b(k',l) \mid 0\le l \le k'\le k \}$$
are respectively free 
$\ZQ$-bases of ${\Sm}$ and $F_k({\Sm})$, where $b(k,l)$ is given in Example~\ref{ssExam}. Note that $| B_k | = k(k+1)/2$.

{\bf (Step 1)}.
A simple calculation using the defining relations shows that in ${\Sm}$
\begin{align}
 qu_+ u_- -q^{-1} u_- u_+&= q- q^{-1}.
\label{eq.11a}
\\
b(2,0) &= q^{1/2} (u_+)^2, \label{eq.11b}\\
b(2,1)&=  q^{1/2}(  u_+ u_- -1  ) , \label{eq.11c} \\
b(2,2)&= q^{1/2} (u_-)^2  \label{eq.11cc}
\end{align}

Identity \eqref{eq.11a} implies the map $f: Y \to {\Sm}$ given by $f(y_+)= u_+, f(y_-)= u_-$ is a well-defined $\cR$-algebra homomorphism. 

 {\bf (Step 2)}. Let us prove that $F_k({\Sm}) = (F_1({\Sm}))^k$. 
 
 As $F_k F_l \subset F_{k+l}$ we have $(F_1({\Sm}))^k\subset F_k({\Sm})$. We will prove the converse inclusion by induction on $k$.
 The statement for $k=2$ is true since the basis $B_2$ of $F_2({\Sm})$ consists of $b(1,0)=u_+, b(1,1)=u_-, b(2,0), b(2,1), b(2,2)$, and Identities \eqref{eq.11b}--\eqref{eq.11c} show that they are all in $F_1({\Sm})^2$.

Let $k \ge 3$. Consider an element $b(k,l)$ of the basis $B_k$ of $F_k({\Sm})$. Let
$a_1, \dots, a_r$ be the stated connected components of $b(k,l)$. Then $k= 2r+ \epsilon$ where $\epsilon=0$ or 1.  Proposition \ref{r.switch2} shows that
\be  b(k,l) \bqq  a_1 \dots a_ r\mod F_{k-1}({\Sm}). \label{eq.166}
\ee
Each $a_i$ has at most two end points and hence belongs to $F_2({\Sm})$.
 
When $\epsilon=0$ we have $k=2r$ and $a_1 \dots a_ r\in (F_{2}({\Sm}))^r = (F_1({\Sm}))^{2r}= (F_1({\Sm}))^{k} $. 

When $\epsilon=1$ one of the $a_i$ is $u_{\pm }$ which is in $F_1({\Sm})$. Hence we also have $a_1 \dots a_ r\in (F_1({\Sm}))^{k}$. By induction,  the right hand side of \eqref{eq.166} is in $(F_1({\Sm}))^{k}$. 

{\bf (Step 3)}. Let us prove that $f(F_k(Y)) =F_k({\Sm})$,  where $F_k(Y)\subset Y$ be the $\cR$-span of all monomials in $y_+, y_-$ of degree $ \le k$.
 
 Since $F_1(Y)$ is $\cR$ spanned by $y_+$ and $y_-$, we have $F_k(Y) = (F_1(Y) )^k$. By definition $f(F_1(Y))= F_1({\Sm})$. Hence
 $$ f( F_k(Y) ) = f ( (F_1(Y) )^k  ) = (F_1({\Sm}))^k = F_k({\Sm}).$$
 
{\bf (Step 4)} From Step 3, let denote $f_k\colon F_k(Y) \to F_k({\Sm})$ the restriction of $f$ on $F_k(Y)$, with target space $F_k({\Sm})$.  Step 3 shows that  $f_k$ is surjective. 
 We now prove $f_k$ is bijective.

From the defining relation it is easy to see that $F_k(Y)$ is free over $\cR$ with basis $(y_+)^{k_1}(y_-)^{k_2}$ with $k_1 + k_2\le k$. It follows that $F_k(Y)$ is free over $\cR$ with rank $k(k+1)/2$, the same as that of $F_k({\Sm})$.  As $f_k\colon F_k(Y)\onto F_k({\Sm})$ is a surjective $\cR$-linear map between two free $\cR$-modules of the  same finite rank, we conclude that $f_k$ is an isomorphism.
  
  Thus $f$ is bijective, and hence an algebra isomorphism.
 \end{proof}

\def\bad{{\mathrm{bad}}}
\def\bSm{{\overline\cS(\fm)}}
\begin{proof}[Proof of Theorem \ref{thm.trmon}]

(a)  The only bad arc is $b(2,1)$.
 By  \eqref{eq.11c} we have
$ b(2,1)= q^{1/2}(  u_+ u_- -1  )$.
Hence from the presentation of ${\Sm}$ given by Theorem \ref{monogonpres}, we get
\be 
{\bSm} = \ZQ\langle u_+,u_- \mid  q u_+ u_- -q^{-1} u_- u_+= q- q^{-1},\ u_+ u_- -1 \rangle \cong\ZQ[u^{\pm 1}], \label{eq.pres2}
\ee
where the last isomorphism is given by $u_\pm \to  u^{\pm1}$.
Thus,  there is a unique $\ZQ$-algebra isomorphism $\btr_\fm\colon {\bSm} \to \ZQ[\bx_e^{\pm 1}]$ given by $ \btr_\fm(u_+)=\bx_e$. Then  $ \btr_\fm(u_-)= (\bx_e)^{-1}$.

Both $u_+$ and $\bx_{e}$ are reflection invariant. Hence $\btr_\fm$ is reflection invariant. 

(b)  
As ${\Sm}$ is $\ZQ$-spanned by $(u_+)^k (u_-)^l$ and $\deg_e(u_\pm)= \pm 1$, the diagram $\al$ is a $\ZQ$-linear combination of $(u_+)^k (u_-)^l$ with $k-l=\deg_e \al$. Hence from \eqref{eq.trmon} we get $\btr_\fm(y)\in  (\bx_e)^{\deg_e(\al)} \, \ZQ$.

(c) Each of the components $\al_1, \dots, \al_k$ of $\al$ is either $u_+$ or $b(2,0)$. By Identity \eqref{eq.11b} we have $b(2,0)= q^{1/2} (u_+)^2$. The height exchange relation of \cite[Lemma 2.4]{CL} shows that $\al \qeq \al_1 \dots \al_k $. Hence 
$\al  \qeq (u_+)^{m},$ for some $m\in \BN$.
\end{proof}

\subsection{Proof of Theorem \ref{thm.embed3}}

\def\bsX{{\overline{\sX}}}
\def\Do{{\mathring \Delta}}
\def\TR{{\overline{\mathrm{TR}}}}
\def\Zq{{\BZ[q^{\pm 1/2}]}}
\def\bphi{{\overline{\varphi}}}
\def\Eo{{\mathring E}}
\label{reducedqtrace}

Let $\Do \subset \Delta$ be the set of interior edges and $\Eo  :=\bigcup_{e\in \Do} e$.
By cutting $\fS$ along edges in  $\Do$ we get $\fS'$, which is the disjoint union of all faces in $\cF= \cF(\Delta)$.
 We identify $\bcS(\fS'{})$ with $\bigotimes_{\tau \in \cF} \bcS(\tau{}) $. Every edge $e\in \Do$ is cut into two edges $e'$ and $e''$ of $\fS'$. Using Proposition \ref{r.cuttri}, we  identify  
 $\bsX(\fS, \Delta{})$ with the $\RS$-submodule of $\bigotimes_{\tau \in \cF} \bsX(\tau{}) $ spanned monomials $z$ with property $\deg_{\bx_{e'}}(z) = \deg_{\bx_{e''}}(z)$ for all $e\in \Do$.

 Consider the composition
\be  \bphi_\Delta\colon {\bSS} \xhookrightarrow{\Cut  }  \bigotimes_{\tau \in \cF} \bcS(\tau{}) \xhookrightarrow{\TR } \bigotimes_{\tau \in \cF} \bsX(\tau{}), \quad \text{where} \  \TR = \bigotimes_{\tau \in \cF} \btr_\tau .
\label{eq.bvphi0} 
\ee
 We now show the image of $\bphi_\Delta$ is in $\bsX(\fS,\Delta{})$.
Assume $y\in {\bSS}$ is represented by a stated $\pfS$-tangle diagram $D$ which is transverse to $\Eo$. It is enough to show that $\bphi_\Delta(y) \in \bsX(\fS,\Delta{})$.

Fix an orientation $h$ of $\Eo$, and use it to order elements in each $e\in \Do$. By definition,  
\be \bphi_\Delta(y)= \sum_{s: D \cap \Eo \to  \{\pm \}  }  \TR ((D,h,s)),
\label{eq.bvphi} 
\ee
where $(D,h,s)\in \bcS(\fS'{})$ is obtained by cutting $D$ along all $e\in \Do$ and the state for the newly created boundary points are given by $s$. Note that for $e\in \Do$ we have
$$ \deg_{e'}(D,h,s) =\deg_{e''}(D,h,s)$$
because both are equal to 
\be
d(e,s):= \sum_{u\in  D \cap e} s(u) . \label{eq.degee}
\ee
By Theorems \ref{thm.trtri}(b) and \ref{thm.trmon}(b) we have that 
$\TR(( D,h,s))$ is homogeneous in both $\bx_{e'}$ and $\bx_{e''}$ of the same degree $d(e,s) $. Thus $\bphi_\Delta(y) \in \bsX(\fS,\Delta{})$. This shows  $ \text{Image}(\bphi_\Delta)\subset \bsX(\fS,\Delta{})$. 

Let  $\btr_\Delta$ be the same as $\bphi_\Delta$ but with the target space restricted to $\bsX(\fS,\Delta{})$:
$$ \btr_\Delta\colon {\bSS} \embed \bsX(\fS,\Delta{}).$$
By construction, $\btr_\Delta$ is compatible with the cutting homomorphism. All the component homomorphisms in \eqref{eq.bvphi0} are reflection invariant; hence so is $\btr_\Delta$. When changing the ground ring from $\ZQ$ to $\cR$, all component homomorphisms in \eqref{eq.bvphi0} are still injective; hence so is $\btr_{\Delta, \cR}$.

\def\inpp#1{\input{draws/#1}}
\def\Ybl{{ \bsX}}
\def\surface{\fS}
\def\cZ{{\mathcal Z}}
\def\he{{\hat e}}

\newcommand{\extY}{\bar{\mathcal{Y}}}
\newcommand{\extYbl}{\bar{\mathcal{Y}}^\mathrm{bl}}
\newcommand{\extX}{\bar{\mathfrak{X}}}

\newcommand{\rd}{{\mathrm{rd}}}
\def\bY{{\bar {\cY}}}
\def\bYbl{\bY^{\mathrm{bl}}}
\def\bal{{\mathrm{bl}}}
\def\bYbl{\bY^\bal}
\def\sXd{{\overset\diamond{\sX}}}

\subsection{Proof of Theorem \ref{thm.embed3e}(a)}  \label{sec.extended}

We now define the unreduced quantum trace $\tr_\Delta$. The construction and proof are almost identical to those in \cite{LY2}: The idea is to embed $\fS$ into a bigger surface $\fS^\ast$ and use the reduced quantum trace of $\fS^\ast$.

 Let $\fS^\ast$ be the result of attaching an ideal triangle $\ft_e$ to each boundary edge $e$ of $\fS$ by identifying $e$ with an edge of $\ft_e$. Denote the other two edges of $\ft_e$ by $\hat e ,\hat e'$ as in Figure~\ref{fig:extSurf}.

There is a smooth embedding $\iota: \fS \embed \fS^\ast$ which maps $e$ to $\hat e$ and is the identity outside a small neighborhood of $e$ for every boundary edge $e$, see Figure~\ref{fig:extSurf}. In particular $\iota(a)=a$ for all $a\in \Do$.

\begin{figure}[h]
\centering
\includegraphics[width=300pt]{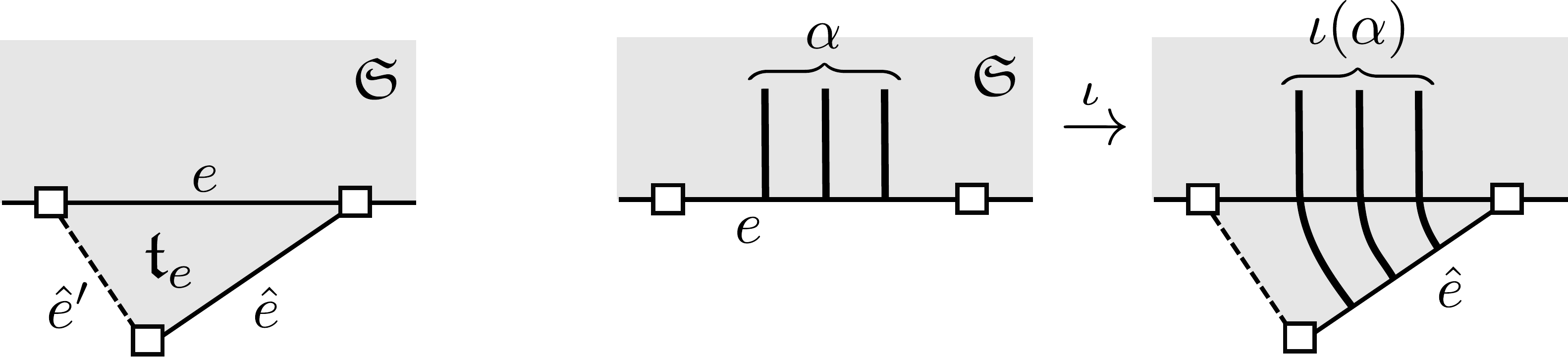}
\caption{Left: attaching triangle $\ft_e$. Right: the embedding $\iota:\fS\embed \fS^\ast$..}\label{fig:extSurf}
\end{figure}

 Note that $\iota$ induces a $\RS$-algebra homomorphism $\iota_*\colon {\SSR} \to \cS(\fS^\ast{;\cR})$, sending a stated $\pfS$-tangle diagram $\al$ to $\iota(\al)$.

Then $\Delta^\ast=\Delta\cup (\bigcup_{e\in \Dd} \{\hat e , \hat e'\})$ is an ideal triangulation of $\fS^\ast$.  We will show that the following composition
\begin{equation}\label{eq.phiD}
\phi_\Delta\colon {\SS} \xrightarrow{\iota_\ast} \cS(\fS^\ast{}) \xrightarrow{p} \bcS(\fS^\ast{}) \xhookrightarrow{\btr_{\Delta^\ast}} \Ybl(\fS^\ast;\Delta^\ast{})
\end{equation}
defines the quantum trace $\tr_\Delta$ we are looking for. 
It is important to note that under $\iota$ the image of a $\pfS$-arc is never a corner arc, and a fortiori it is not a bad arc. Let $\phi_{\Delta, \cR}= \phi_\Delta \ot_\ZQ \cR$. The composition $p\circ \iota_\ast$ maps the $\RS$-basis $B(\fS)$ of ${\SSR}$ injectively into the $\RS$-basis $\overline{B} (\fS^\ast)$ of $\bcS  (\fS^\ast{;\cR})$. This shows $p\circ \iota_\ast$ is injective. Hence the composition $\phi_{\Delta,\cR}$ is injective.

Since $\bQ_{\Delta^\ast}(e, \hat e)=1=-\mQ_\Delta(e,\hat e)$, we can identify the quantum torus $\sX(\fS,\Delta{})$ as a subalgebra of $\bsX(\fS^\ast;\Delta^\ast{})$ via the embedding 
\begin{equation}
\label{eq.yz}
\begin{cases}
x_a \mapsto \bx_a, & a\in \Do, \\
x_e \mapsto [\bx _e \bx _{\hat e}], &e\in \Dd,\\
x_{\hat e}\mapsto \bx _{\hat e}^{-1}, &e\in \Dd.
\end{cases}
\end{equation}
Under this identification $\sX(\fS,\Delta{})$ is $\RS$-spanned by monomials not involving $\bar x_{\hat e'}$ for all $e\in \Dd$.

\blem The image of $\phi_\Delta$ is in $\sXd(\fS,\Delta{})$.
\elem
\bpr Let $D$ be  stated $\pfS$-tangle diagram.
The analog of \eqref{eq.bvphi} is

\be \phi_\Delta(D)= \sum_{s: D \cap E \to  \{\pm \}  }  \TR ((\iota(D),h,s)),
\label{eq.vphi} 
\ee
where $(\iota(D),h,s)$ is the stated tangle diagram over $\bigsqcup_{\tau \in \cF(\Delta^\ast)} \tau$  obtained by cutting $\iota(D)$ along $e\in E$ and the state for the newly created boundary points are given by $s$.

 Since $\iota(D)\cap \he '= \emptyset$, none of the $\TR ((\iota(D),h,s))$ 
  involve $x_{\he'}$. Hence $\phi_\Delta(D)\in \sX(\fS,\Delta{})$. 
  
Let $e\in \Dd$. Then  $(\iota(D),h,s)\cap \ft_e$ consists of 
 several  stated  arcs $a_1, \dots, a_r$ which are parallel to each other. Then  $(\iota(D),h,s;e)\qeq a_1 \dots a_r$ in $\bcS(\ft_e{})$. To be non-zero,  each $a_i$ cannot be a bad arc, and this forces the state of $a_i \cap e$ is greater than or equal to that of $a_i \cap \he$. It then follows from the definition \eqref{eq.yz} that  the degree of $x_{\he}$ in $\btr_{\ft_e}((\iota(D),h,s) \cap \ft_e )$ is  non-negative. Hence the image of $\phi_\Delta$ is in $\sXd(\fS,\Delta{})$.
\epr
Define $\tr\colon {\SS} \to \sXd(\fS,\Delta{})$ as $\phi_\Delta$, with the target space restricted to $\sXd(\fS,\Delta{})$.

The projection $\pr\colon \sXd(\fS,\Delta{}) \onto \bsX(\fS,\Delta{})$ kills all the $x_{\he }$ with $e\in \Dd$. 
By removing all the terms of the right hand side of \eqref{eq.vphi} which has positive power of $x_{\he }$  for some $e\in \Dd$, we get the right hand side of \eqref{eq.bvphi}. Hence  $\pr (\phi_\Delta(D))= \bphi_\Delta(D)$. The commutativity of diagram \eqref{eq.dia9} follows. 

All the component maps in \eqref{eq.phiD} are reflection invariant, hence so is $\tr_\Delta$. As $\phi_{\Delta,\cR}$ is injective, so is $\tr_{\Delta, \cR}$.
 This completes the proof of Theorem \ref{thm.embed3e}(a).

\subsection{Top degree term of $\tr_\Delta$} We now calculate the top degree part of $\tr_\Delta(\al)$ for $\al\in B(\fS)$.
Again we proceed by adapting the argument of \cite{LY2}.

\begin{lemma}\label{r.phi1} Let  $\al\in B(\fS)$
 with $\sum_{e\in \Delta} \bn_\al(e)=k$, where
 $\bn_\al$ is defined by \eqref{eq.bbn}. Then
\begin{equation}
\label{eq.ss0}
\tr_\Delta(\al)\,  \qeq\,   x^{{\bn}_\al}   \mod F^\Delta_{k-1} ( \sX(\fS,\Delta)).
\end{equation}
\end{lemma}

\def\Coeff{{\mathrm{Coeff}}}
\def\ma{{\mathsf{max}}}
\begin{proof} We assume  $\al$ is represented by $D$ which is taut with respect to $\Delta$, so that $\bn_\al(e)= |D\cap e|$ for all $e\in \Delta$.
In \eqref{eq.vphi} the degree of $x_e$ in $\TR(\iota(D), h,s)$ is 
$$ d(s,e)= \sum _{u \in \iota(D) \cap e} s(u) \le |e \cap D|= \bn_\al(e).$$
The maximal of $d(s,e)$ is $\bn_\al(e)$,  achieved when $s=s_+$, which takes value $+$ at every point of intersection between $D$ and $\Delta$. 

Consider the term $(\iota(D), h,s_+ )$. Let $\tau$ be a face of $\Delta^\ast$. 
Then $(\iota(D), h,s_+ )\cap \tau$ consists of several stated $\partial \tau$-arcs $a_1, \dots, a_r$.

First assume $\tau$ is not one of $\ft_e, e \in \Dd$. Then all edges of $\tau$ are in $\Delta$. Hence all states of $a_i$  are positive.
By \eqref{eq.deg8a} and \eqref{eq.deg9a},
\be 
\btr_\tau( (\iota(D), h,s_+ )\cap \tau  ) \, \eqbu\,  \prod_{a} (\bx_a)^{\bn_\al(a) }
\label{eq.tr7}
\ee
 where $a$ runs in the set of edges of $\tau$.

Now consider the case  $\tau=\ft_e, e\in \Dd$. Then each $a_i$ is a corner arc connecting $e$ and $\he$. Since the states on $e$ are positive, each $a_i$ is not a bad arc. From \eqref{eq.deg8a} we get 
\be 
\btr_\tau( (\iota(D), h,s_+ )\cap \tau  )\, \eqbu\,  (\bx_e \bx_{\he})^{\bn_\al(e) } (\bx_{\he})^{-\bn_\al(\he) }
\label{eq.tr8}
\ee

Combining all the faces of $\Delta^\ast$ and using the embedding \eqref{eq.yz}, we get \eqref{eq.ss0}.
\end{proof}

\def\bYt{ A^{(2)}}
\def\bYtp{ A'^{(2)}}
\def\trD{{\tr_\Delta}}
\def\trDR{{\tr_{\Delta;\cR}}}
\def\FD{F^\Delta}
\def\XS{{\sX(\fS, \Delta)}}
\def\XSR{{\sX(\fS, \Delta;\cR)}}
\def\LT{{\mathsf{lt}}}

\subsection{Proof of Theorem \ref{thm.embed3e}(b)}
Let us prove \eqref{eq.in5}, which says $\trDR$ respects the filtrations.

By Proposition \ref{r.basisGr} the set
$$ B_k: = \{ \al\in B(\fS) \mid \sum_{e\in \Delta} \bn_\al(e) \le k \}$$
is an $\RS$-basis of $F_k^\Delta({\SSR}))$.  From Lemma \ref{r.phi1} we have
\be
\trDR(\al) \, \eqbu \,x^{\bn_\al} \mod \FD_{k-1}({\XSR}) \quad \text{ for } \ \al \in B_k \setminus B_{k-1}.
\label{eqTopdeg}
\ee
It follows that for all $\al \in B_k$ we have $\trDR(\al) \in \FD_k( {\XSR})$. Hence we have \eqref{eq.in5}.

Consider the associated graded homomorphism of $\trDR$:
$$ \Gr(\trDR)\colon \Gr^\Delta({\SSR}) \to \Gr^\Delta({\XSR}).$$
Recall the lead term $\lt(\al)$ is defined in Subsection \ref{ssFiltr}.  The set 
$ \lt(B):= \{ \lt(\al) \mid \al \in B(\fS)\}$
is a free $\RS$-basis of $\Gr^\Delta({\SSR})$. From \eqref{eqTopdeg} we have
\be 
\Gr(\trDR)(\LT(\al))\, \eqbu \, x^{\bn_\al} \quad \text{for all} \ \al \in B(\fS). \label{eqiso4}
\ee

By Theorem \ref{thm.tri-coord} the  map $B(\fS) \to \Lambda_\Delta$ given by $\al \to \bn_\al$ is a bijection. Besides the set $\{x^\bn  \mid \bn \in \Lambda_\Delta\}$ is a free $\RS$-basis of $\bT(\mQ,\Lambda_\Delta;\cR)$. Hence $\Gr(\trDR)$ maps the $\RS$-basis $\lt(B)$ of $\Gr^\Delta({\SSR})$ isomorphically onto the $\RS$-basis $\{x^\bn  \mid \bn \in \Lambda_\Delta\}$ of $\bT(\mQ, \Lambda_\Delta; \cR)$. It follows that $\Gr(\trDR)$ is an {isomorphism}. This completes the proof of Theorem \ref{thm.embed3e}(b).

\subsection{Proof of Theorem~\ref{thm.embed3e}(c)} \label{ssProofe}

 By Theorem \ref{thm.tri-coord} the submonoid $\LD$ is finitely generated $\BN$-module which has rank $r(\fS)$. Hence Lemma~\ref{lemma-GKdim} shows that the monomial algebra $\bT(\mQ,\Lambda_\Delta;\RS)$ has GK dimension $r(\fS)$ over $\RS$.

For each $k$, the $\RS$-module $F_k^\Delta({\SSR})$ is free of finite rank. By Lemma~\ref{liftfacts}, the GK dimension of ${\SSR}$ over $\RS$ is also $r(\fS)$. This completes the proof of Theorem  \ref{thm.embed3e}.

\newcommand{\stateS}{\mathscr{S}}
\newcommand{\reduceS}{\mathscr{S}^\rd}

\subsection{Naturality of the quantum traces with respect to triangulation changes} 

\def\bXSbl{{ \bsX^\bal(\fS, \Delta) }}
\def\bXSblR{{ \bsX^\bal(\fS, \Delta;\cR) }}
\def\bYbl{\bsX^\bal  }
A natural question to ask is what the relations between quantum traces associated to different triangulations are. 
The following facts can be proved, and the proofs are almost identical to the ones (for $\pfS$ having no circle components) given in \cite{LY2}.  The following will not be necessary for our purposes so we state it here with the proof omitted and leave the details for the interested reader.

A map $\bn\colon \Delta\to \BZ$ is balanced if $\bn(a) + \bn(b) + \bn(c)$ is even for any triangular triple $(a,b,c)$. 
Let ${\bXSblR}$ be the $\RS$-submodule spanned by $\bx^\bn$ with balanced $\bn$. Then ${\bXSblR}$ is also a quantum torus and hence has a division algebra of fractions $\Fr({\bXSblR})$.

{\bf FACT.} For two ideal triangulations $\Delta,\Delta'$, there is an  algebra isomorphism
\[\bTheta_{\Delta\Delta'}\colon\Fr(\bYbl(\fS;\Delta'{;\cR}))\to \Fr(\bYbl(\fS;\Delta{;\cR}))\]
satisfying the following properties

\begin{itemize}
\item For three triangulations $\Delta, \Delta', \Delta''$ we have
  $\bTheta_{\Delta''\Delta'} \circ \bTheta_{\Delta'\Delta} = \bTheta_{\Delta''\Delta}$ and $\bTheta_{\Delta \Delta}=\id$.

\item The quantum trace $\btr_\Delta$ is compatible with coordinate changes, i.e.,
$$
\bTheta_{\Delta' \Delta} \circ \btr_\Delta= \btr_{\Delta'}.
$$

\end{itemize}

\def\uB{{\underline{\mathrm{B}}}}

\section{The boundary simplified skein algebra}
\label{secMRY}

To prove Theorem \ref{thmSg} and \ref{thmSRY} we will use a version of a quantum trace for a simplified version of a skein algebra, meaning we introduce a subquotient such that all the near boundary arcs are equal to 0. This extra condition behaves well with respect to the filtration defined by the intersection number with the boundary components. In this section we discuss the boundary simplified skein algebra.

Throughout the section $\fS$ is a punctured surface and $\cR$ is a ground ring.

\def\Bp{{B^+}}
\subsection{The Muller-Roger-Yang (MRY) skein algebra } \label{ssMRY}
We define the MRY skein algebra. 

 The $\ZQ$-submodule $ {\SSp}$ of $ {\SS}$ spanned by boundary ordered $\pfS$-tangle diagrams with only positive states is a subalgebra of $ {\SS}$, called the {\bf MRY skein algebra} of $\fS$. Let $\SSpR= \SSp \ot_\ZQ \cR$.
 Since states are positive everywhere, we don't need to specify the states in Figures. 

%, Proposition 2.4 of \cite{Le:triangular}
\blem[Height exchange rule ] \label{rHeight}
(a)   In $\SSp$ one has
\be 
\begin{array}{c}\includegraphics[scale=0.7]{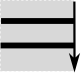}\end{array} = q^{-1} \begin{array}{c}\includegraphics[scale=0.7]{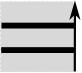}\end{array}
\label{eqHeight}
\ee
(b) Let $\al$ be a boundary ordered simple  $\pfS$-diagram  with connected components $\al_1, \dots, \al_k$. We have,
\be \al \qeq \al_1 \dots \al_k \ \text{ in $\SSp$}. \label{eqMulti}
\ee
There is a unique  $l\in \BZ$ such that we have $\omega(\al) = q^{l} \al$, where $\omega$ is the reflection.

\elem
\bpr  (a) is a special case of \eqref{eq.sign}. (b) follows immediately from (a).
\epr
With $\al$ as in Lemma \ref{rHeight} define the reflection normalization
\be 
[ \al]_\omega = q^{l/2} \al. \label{eqRefl}
\ee
Then $[\al]_\omega$ is reflection invariant.

It is easy to check, though we don't need it in the sequence, that
\be 
[ \al]_\omega = q^{ - \sum \frac{\bn_\al(e) (\bn_\al(e) -1)}{2} } (\al, h^+),
\ee
where $(\al, h^+)$ is $\al$ with the positive boundary order, and the sum is over all boundary edges.

\bpro \label{rMRYbasis}
The subset $\Bp(\fS)$ of $B(\fS)$  of all $\alpha\in B(\fS)$ with positive states is  a free $\cR$-basis of $ {\SSpR}$.
\epro
\bpr  As a subset of a basis, $\Bp(\fS)$ is $\cR$-linear independent. On the other hand, 
the skein relations and the height exchange show that $\Bp(\fS)$ spans $\cS^+(\fS {;\cR})$ over~$\cR$.  Thus $B^+(\fS)$ is a free $\cR$-basis of $\cS^+(\fS {;\cR})$.
\epr

\brem (a) We can define $\cS^+(\fS {;\cR})$ by using the non-stated tangles, and relations which are the only relations in (A)-(F) which do not contain negative states.

(b) When $\fS$ has no circle boundary and no interior punctures, the algebra $\cS^+(\fS {;\cR})$ is isomorphic to the Muller skein algebra \cite{Muller}, as proved in \cite{Le:triangular}.
% When there is no circle boundary component, the algebra $\uSS$ was defined in \cite{Le:Qtrace}, and also in \cite{PS2}.
\erem

\def\utr{{\underline{\tr}}}
\def\bbl{{\mathbf l}}
\def\mP{{\mathsf P}}
\def\uD{{\underline{\Delta}}}
\def\uA{{\underline{\cA}}}
\def\uP{{\underline{\mP}}}
\def\uv{\overleftarrow }
\subsection{A quantum trace } 

\newcommand{\vertexdef}[2]{
\begin{tikzpicture}[scale=0.8,baseline=0.28cm]
\fill[gray!20!white] (-0.2,0) rectangle (1.3,1);
\draw (0,1)--(0.6,0)--(1,1) (-0.2,0)--(1.3,0);
\draw (0.1,0.5)node{\vphantom{$b$}#1} (0.98,0.5)node{\vphantom{$b$}#2};
\draw[fill=white] (0.6,0)circle(2pt);
\end{tikzpicture}
}

In this subsection we show that $ {\SSpR}$, for a certain class of surface, is sandwiched between a quantum space and its quantum torus, defined by an  antisymmetric integral matrix corresponding to a triangulation of $\fS$.

 Let $\fS$ be a punctured surface without interior puncture. Thus
 $\fS=\bfS\setminus \cP$, where $\bfS$ is a compact surface and $\cP$ is a finite subset of the boundary $\pbfS$.

\bdf
A {\bf quasi-ideal arc} is either an ideal arc or an embedding $a: [0,1) \embed \fS$ which can be extended to a proper embedding $\bar a: [0,1] \embed \bfS$ such that $\bar a(1)\in \cP$ while $a(0)$ is in one of the circle boundary component of $\fS$.

A quasi-ideal multiarc is a disjoint union of quasi-ideal arcs. For a quasi-ideal multiarc $\al$ the $\pfS$-tangle diagram $\uv\al$ is obtained by moving the branches of $\al$ near every ideal point  to the left as in Figure \ref{fig:uv}.
\edf
\begin{figure}[h]
    \centering
    \includegraphics[width=130pt]{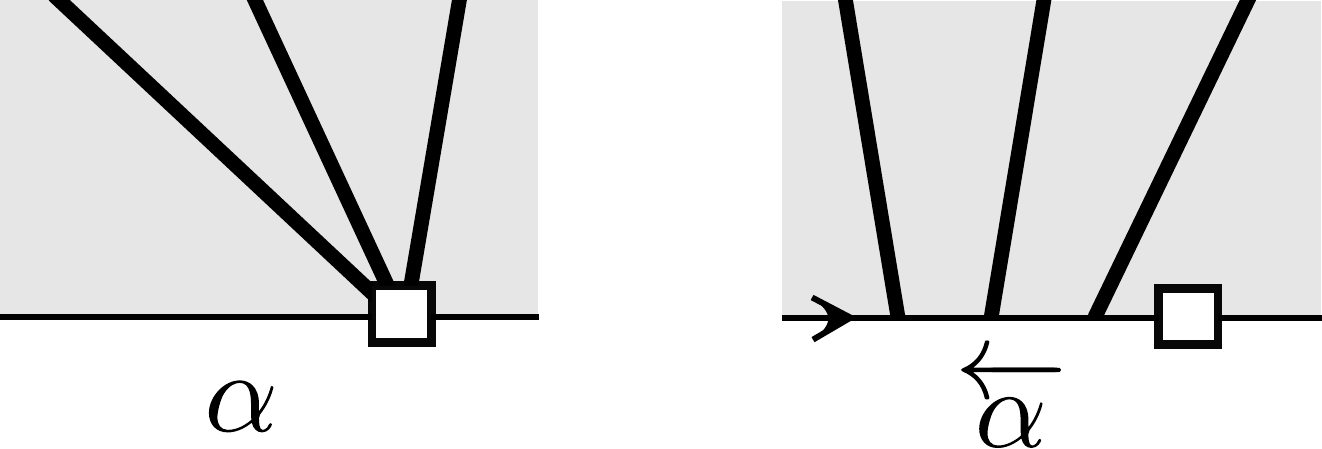}
    \caption{Moving left: From quasi-ideal multiarc $\al$ to $\uv\al$. If a component of $\al$ is a boundary edge, we first slightly isotope it to an interior position and then do the move-left operation.}
    \label{fig:uv}
\end{figure}

 For disjoint quasi-ideal arcs $c$ and $e$ define  $\mP(c,e)$ by 
$$ \mP(c,e) =\ \#\left( \begin{array}{c}\includegraphics[scale=0.16]{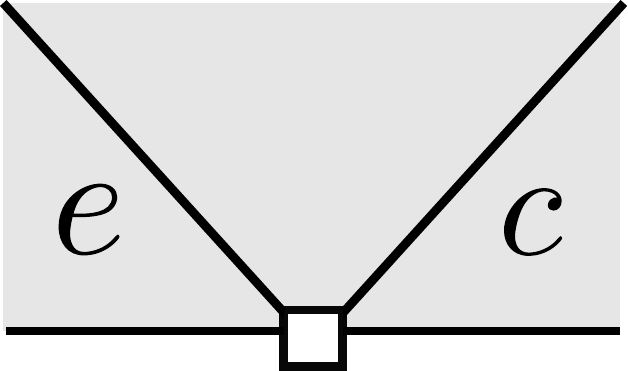}\end{array} \right) - \#\left( \begin{array}{c}\includegraphics[scale=0.16]{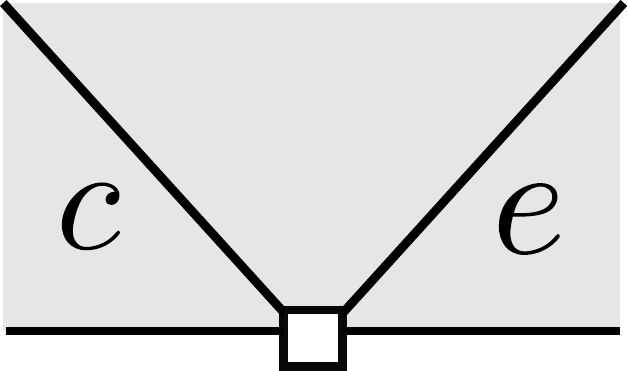}\end{array} \right). $$The right hand side counts the number of times a half-edge of $c$ meets a half-edge of $e$ at an ideal vertex as an algebraic intersection, meaning where the count adds 1 if $c$ is clockwise to $e$ and $-1$ otherwise.

\def\ASD{\cA(\fS,\Delta)}
\def\ASDR{\cA(\fS,\Delta;\cR)}
\def\ASDp{\cA_+(\fS,\Delta)}
\def\ASDo{\Ad(\fS,\Delta)}
\def\trAD{\tr_\Delta^A}
\def\trADR{\tr_{\Delta;\cR}^A}

Assume further that $\fS$  has  a  triangulation $\Delta$. 
Each circle boundary component $\mu\in \cM$ is in a 1-marked monogon bounded by an edge $e_\mu\in \Delta$, called a {\bf monogon edge}, see Figure~\ref{fig:amu}. Choose a quasi-ideal arc $c_\mu$ in the monogon, connecting $\mu$ and the ideal vertex of the monogon as in Figure \ref{fig:amu}. Let $\Dbu$ be the collection $\Delta$ where each monogon edge $e_\mu$ is replaced by the quasi-ideal arc $c_\mu$. Let $\ASD$
 be the quantum torus associated to the restriction of $\mP$ to $\Dbu \times \Dbu$:
\be 
\ASD
=\ZQ \la \aaa_c^{\pm 1}, c \in \Dbu \ra /( \aaa_c \aaa_e= q^{\Pbu(c,e)} \aaa_e \aaa_c) .
\ee
The set of monomials $\{ \aaa^\bk \mid \bk \in \BZ^\Dbu\}$ is a $\ZQ$-basis of $\ASD$. Consider subalgebras
\begin{align*}
\ASDp &= \ZQ\text{-span of} \ \{\aaa^\bk \mid \bk \in \BN^\Dbu   \} \quad \text{(quantum space ) }\\
\ASDo &= \ZQ\text{-span of} \ \{\aaa^\bk \mid \bk \in \BZ^\Dbu, \bk(e) \ge 0\ \text{for boundary edges} \ e   \}
\end{align*}

%Let $\cA_+(\fS,\Delta) = \bT_+(\Pbu)$ be the corresponding quantum space, and $\Ad(\fS, \Delta)$ be the $\cR$-subalgebra of $\cA(\fS,\Delta)$ spanned by all monomials with non-negative degree of each $\aaa_e$, where $e$ is a boundary edge.
\begin{figure}[h]
    \centering
    \includegraphics[width=200pt]{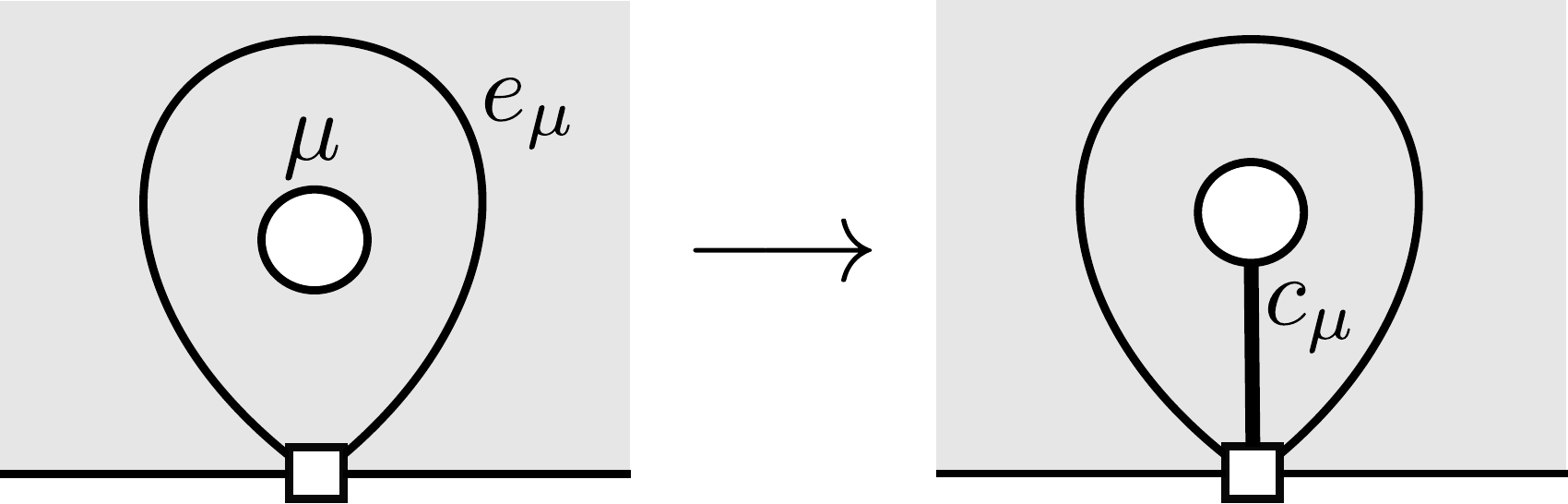}
    \caption{}
    \label{fig:amu}
\end{figure}

\def\ovec{\overrightarrow}

As usual let $\ASDR= \ASD \ot_\ZQ \cR$. Define $\cA_+(\fS,\Delta {;\cR})$ and $\Ad(\fS,\Delta {;\cR})$ similarly.
\bthm \label{thmAtrace}
Suppose $\Delta$ is a triangulation of a triangulable punctured surface $\fS$ having no interior puncture. There is a unique reflection invariant $\ZQ$-algebra embedding 
$$\trAD \colon {\SSp} \embed \ASD$$ such that for all $c\in \Dbu$,
\be \tr_\Delta^A([\uv{  c }]_\omega)= \aaa_c\ee
 Moreover $\trADR := \tr_\Delta^A \ot_\ZQ \cR$ is injective, and
\be 
\cA_+(\fS,\Delta {;\cR}) \subset  \trADR(  {\SSpR} ) \subset \Ad(\fS,\Delta {;\cR}).
\ee

\ethm

\bpr When $\fS$ has no circular boundary component the statement was proved in \cite{Muller}. See also \cite{LP,LY2} for related results. Our proof follows closely  \cite{LY2}.

An important fact is that $ {\SSpR}$  is a domain, as it is a subring of  $ {\SSR}$,  which is a domain by Theorem \ref{thm.embed3e}.

Let $A_c= [\uv{  c }]_\omega$ for $c\in \Dbu$. From the height exchange formula \eqref{eqHeight} we have 
$$A_c A_e = q^{\mP(c,e)} A_e A_c.$$

Hence there is a $\ZQ$-algebra homomorphism
$$\iota \colon \cA_+(\fS,\Delta) \to {\SSp}, \  \iota(\aaa_c)=A_c.$$
Clearly $\iota$ reflection invariant. Let $\iota_\cR= \iota \ot _\ZQ \cR$.
\blem The map $\iota_\cR$ is injective.
\elem
\bpr Let $D$ be the result of moving left of the quasi-miltiarc $\cup _{e\in \Dbu}\, e$.  Then $D= \sqcup _{e\in \Dbu} D_e$, where $D_e = \uv e$.

If $D_e$ is an ideal arc and $k\in \BN$  let $D(e,k)$ be $k$ parallel copies of $D_e$, which lie in a small neighborhood of $D_e$. The height exchange rule implies $(D_e)^k \qeq D(e,k)$ in $ {\SSpR}$.

If $D_e$ is a quasi-ideal arc, $e= c_\mu$, let $D(e,k)$ be the diagram in Figure \ref{fig:Dek},  in a small neighborhood of $D_e \cap \mu$.
From defining relations we also have $(D_e)^k \qeq D(e,k)$ in $ {\SSpR}$. 
\begin{figure}[h]
    \centering
    \includegraphics[width=280pt]{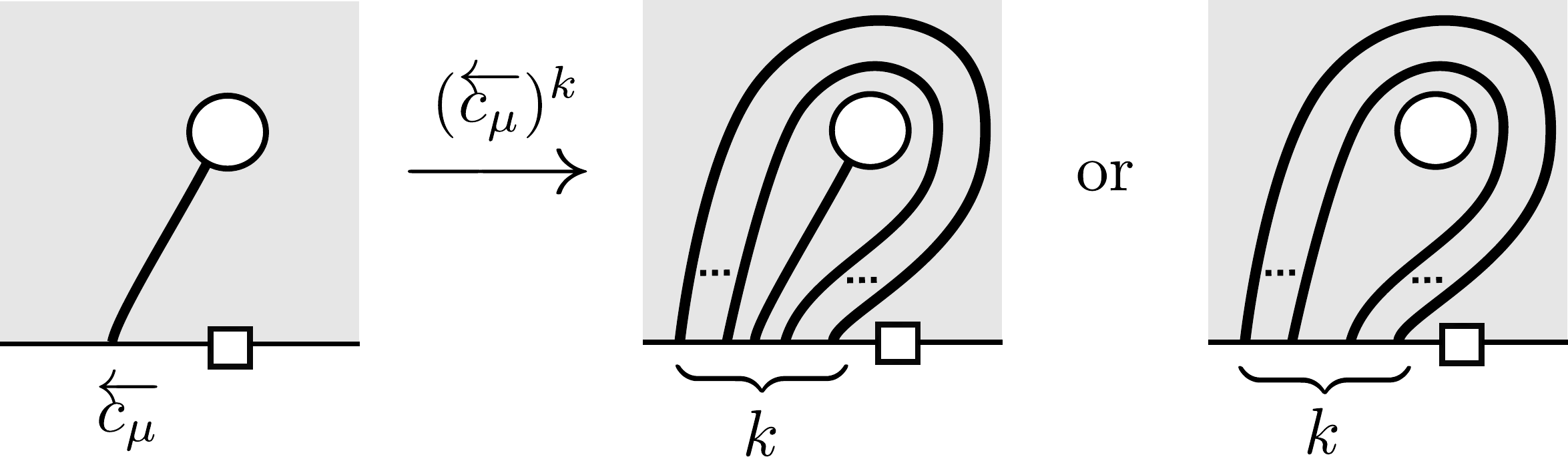}
    \caption{$D(\mu,k)$: left if $k$ odd, right if $k$ even.}
    \label{fig:Dek}
\end{figure}

By the height exchange rule, for $\bk  \in \BN^\Dbu$ the element $\iota(\aaa^\bk)$ is $q$-proportional to $A(\bk):=\sqcup _{e\in \Dbu} D(e, \bk(e))$,  with positive boundary order.  Then $A(\bk)$ is an element of the $\cR$-basis $\Bp(\fS)$ of $ {\SSpR}$ of Proposition \ref{rMRYbasis}. Besides 
$A(\bk)\neq A(\bk')$ as elements of $\Bp(\fS)$ if $\bk\neq \bk'$.
Since $\iota_\cR$ maps the $\cR$-basis $\{ a^\bk \mid \bk \in \BN^\Dbu\}$ injectively into an $\cR$-basis  of $ {\SSpR}$, it is injective.
\epr
Identify $\cA_+(\fS,\Delta {;\cR})$ with its image under $\iota$. For $\al \in  \Bp(\fS)$ let $\ovec \al$ be the quasi-ideal multiarc obtained by reversing the moving left operation of Figure~\ref{fig:uv}.
\blem  \label{rCal2}
For a boundary ordered $\pfS$-tangle diagram $\al$ we have $\aaa^{\bm_\al} \, \al\in  \cA_+(\fS,\Delta)$. Here
 $\bm_\al\colon \Dbu\to \BN$ is defined by $\bm_\al(e)= I( \overrightarrow{\al},e)$.  
\elem
With the lemma and the integrality of $ {\SSpR}$, by \cite[Proposition 2.2]{LY2} 
the algebra embedding $\cA_+(\fS,\Delta {;\cR}) \embed \cA(\fS,\Delta {;\cR})$ extends to  a unique  algebra embedding $  {\SSpR} \embed \cA(\fS,\Delta {;\cR})$, which is reflection invariant when $\cR=\ZQ$.
 Denote this extension by $\trADR$. By construction $\trADR= \trAD \ot _\ZQ \cR$ where  $\trAD =  \trADR$ with $\cR=\ZQ$.
Since $\bm_\al(e)=0$ for each boundary edge $e$, one has $\trADR(\al) \in \aaa^{-\bm_\al} \cA_+(\fS, \Delta  {;\cR})\subset \Ad(\fS, \Delta {;\cR})$. Hence  we have  the theorem.

It remains to prove the lemma. 
%We follow \cite[Section 4]{Muller}.  
Consider %F_k \subset \SSp$ be the submodule
$$ B^+_0: =   \{ \al\in \Bp(\fS) \mid I(e, \ovec\al)=0\ \text{for all } \ e\in \Dbu.\}.$$

Assume $\al\in \Bp(\fS)$. The proof of \cite[Corollary 4.13]{Muller} shows that $\aaa^{\bm_\al}  \al \in \ZQ B^+_0$. In \cite{Muller} only the case of ideal edges is considered, but the proof there does not concern the endpoints, and works as well for the case when there are quasi-ideal arcs. 

Let us consider now $\cR B^+_0$.
Let $\al\in B^+_0$. The maximality of a triangulation implies that each component of $\ovec \al$ is either a quasi-ideal arc or an ideal edge in $\Delta$, possibly a monogon edge. Thus by moving left back we get that $\al = A(\bk)$ for certain $\bk \in \BN^\Dbu$. This shows $\ZQ B_0 = \cA_+(\fS, \Delta)$. We have the lemma, and hence the theorem. \epr
\brem
\label{rCal} Lemma \ref{rCal2} shows how to calculate $\tr_\Delta^A(\al)$, for a boundary ordered $\pfS$-tangle diagram. First we find $\bm_\al$ of Lemma \ref{rCal2}. Then $z=\aaa^{\bm_\al} \al\in \cA_+(\fS,\Delta)$ can be calculated by using the skein relation. Then $\tr_\Delta^A(\al) = \aaa^{-\bm_\al} z$.
\erem
 
\subsection{Boundary simplified skein algebra} Now we define the version of skein algebra that will be used in the proof of Theorem \ref{thmSg} and Theorem \ref{thmSRY}.
Recall that $\fS= \bfS \setminus \cP$.

\bdf \label{defNearB}
 (a) A $\pfS$-arc is {\bf near boundary} if as an arc in $\bfS$ it is homotopic relative its boundary to a subset of $\pbfS$.

(b) A {\bf strongly simple diagram} on $\fS$ is a simple diagram having no near boundary arc.

(b) The {\bf boundary simplified skein algebra} $ {\uSSR}$ is 
$${\uSS} = \cS^+(\fS)/\cI^{\partial},\quad \uSSR= \uSS \ot_\ZQ \cR,
$$ 
where  $\cI^{\partial}$ is the two-sided ideal generated by near boundary arcs.
\edf
Since $\cI^\partial$ is invariant under $\omega$, the reflection $\omega$ descends to a reflection of $\uSS$.

Let $\uB(\fS)$ be the set of all isotopy classes of strongly simple diagrams. Let $h$ be a choice, for each $\al\in \uB(\fS)$, of a boundary order for $\al$. 
\bpro \label{rUbasis}
The set
$(\uB(\fS), h)$  is a free $\cR$-basis of $ {\uSSR}$.
\epro
\bpr With the height exchange rule of Lemma \ref{rHeight}, we can assume  $h$ is the positive order. Using the skein relations it is easy to see that $(\uB(\fS), h)$ spans $ {\uSSR}$.  {\stno{over $\cR$.}}

Let $\cR B'$ be the $\cR$-span of $B'= \Bp(\fS) \setminus \uB(\fS)$. To prove the proposition, it is enough to show that $ \cR B'=\cI^\partial$, which is reduced to: If $\al$ is a non-trivial near boundary arc and $\beta\in \uB(\fS,h)$ then $\al \beta, \beta\al \in \cR B'$. We will use induction on $|\al \cap \beta|$.

Suppose $|\al \cap \beta|=0$. 
By the height exchange rule 
$$ \al \beta \qeq \beta \al\qeq (\al \cup \beta) \in  B'.$$
The case  $|\al \cap \beta| >0$ is reduced to the case with smaller $|\al \cap \beta|$ by (resolving the rightmost crossing)
\begin{align*}
\begin{array}{c}\includegraphics[scale=0.21]{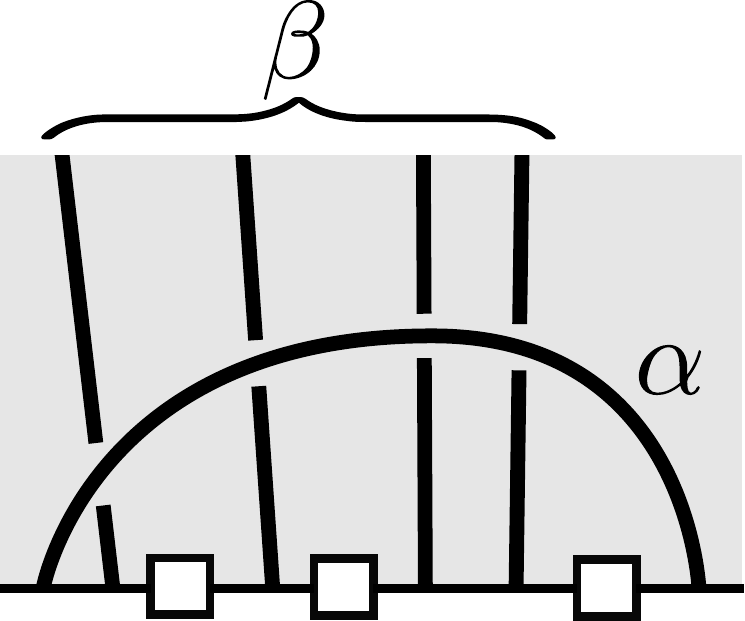}\end{array}=q\begin{array}{c}\includegraphics[scale=0.21]{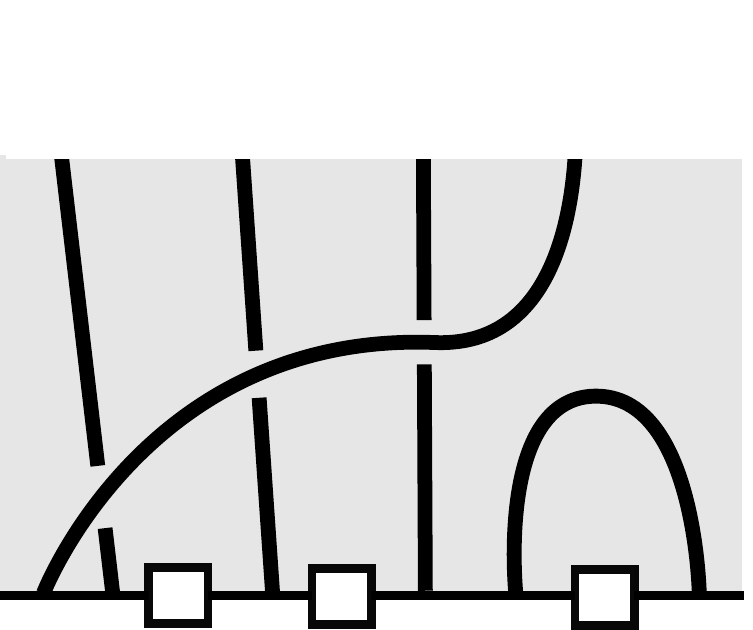}\end{array}+q^{-1}\begin{array}{c}\includegraphics[scale=0.21]{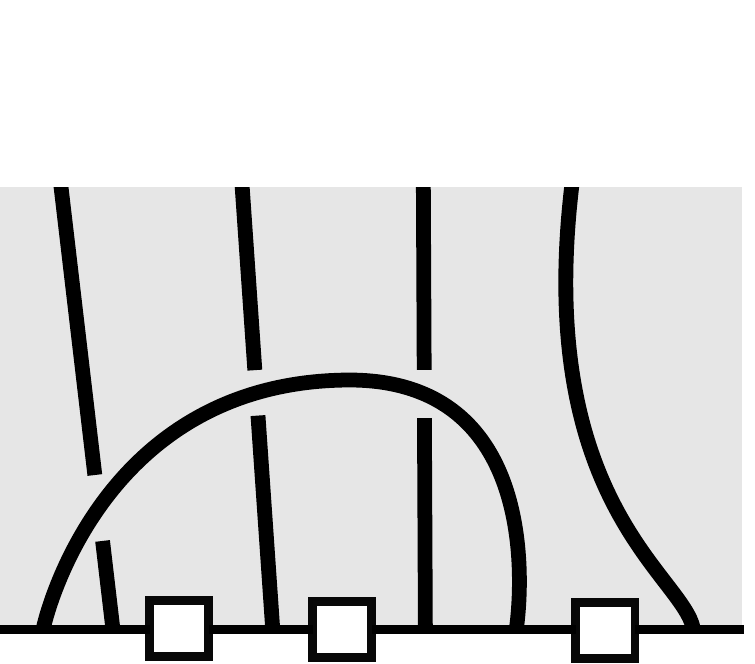}\end{array}
\end{align*}
Here the first term can be further resolved into a linear combination of crossingless diagrams containing a near boundary arc (the lower right arc), and the second diagram is a diagram of the desired form with one fewer crossing.
\epr

\def\uASD {\uA(\fS, \Delta)}

Let $\Dd$ be the set of all boundary edges, and $\uD := \Dbu \setminus \Dd$. Let $\uP$ be the restriction of $\mP$ to $\uD \times \uD$.
Define the smaller quantum torus based on interior quasi-ideal edges:
\be 
\uASD:= \bT(\uP) \subset \ASD, \quad 
\uA(\fS, \Delta {;\cR}):=  \uASD \ot _\ZQ \cR.
\ee

There is a $\ZQ$-algebra surjection $\Ad(\fS,\Delta) \onto \uA(\fS,\Delta)$ given by $\aaa_e \to 0$ for $e\in \Dd$.
\bthm
\label{thm.emb2}
Let $\Delta$ be a triangulation of a triangulable punctured surface $\fS$ having no interior puncture. Assume that  each component of $\pbfS$ has at most one ideal point.

Then 
$\tr^A_\Delta\colon  {\SSp} \to \Ad(\fS,\Delta)$ descends to a $\ZQ$-algebra homomorphism
\be 
\utr^A_\Delta\colon  {\uSS} \to \uASD. \label{equtrr}
\ee
\ethm
\bpr

For  $e\in \Dd$ the moved left $\uv e$ is the only near boundary arc in a neighborhood of $e$. Thus $\cI^\partial$ is the ideal generated by $\uv e$ with $e\in \Dd$. By construction, $\tr^A_\Delta(\uv e)\qeq \aaa_e$. Hence $\tr^A_\Delta$ descends to an algebra homomorphism as given in \eqref{equtrr}.
\epr
\bpro With the assumption of Theorem \ref{thm.emb2}, we have
\be  \utr^A_\Delta(\ell_\mu) = [\aaa_e \aaa_c^{-1} ]_\Weyl +  [\aaa_c \aaa_e^{-1} ]_\Weyl,
\label{eqellmu}
\ee
where $\ell_\mu$ is the loop near a boundary edge $\mu$, and $e, c$ are ideal arcs as in Figure \ref{ellmu}.
\epro
\bpr The product $\ell_\mu \uv e \uv c$ can be calculated, resolving the two crossings and using the skein relations. See the right part Figure \ref{ellmu}, where we need to do the move-left operation. The third term is 0 due a defining relation, while the last term is 0 due to  the presence of a near boundary arc. Thus 
$$ \ell_\mu \uv e \uv c =  q^{k} (\uv e)^2 +  q^l (\uv c)^2,$$
Using  $\utr^A_\Delta(\uv e  )= \aaa_e$, we get, for some $k',l'\in \BZ$,
$$  \utr^A_\Delta(\ell_\mu) = q^{k'} [\aaa_e \aaa_c^{-1} ]_\Weyl +  q^{l'}[\aaa_c \aaa_e^{-1} ]_\Weyl.
$$
We can show $k'=l'=0$ and hence\eqref{eqellmu} by carrying the exact calculation of powers of $q$ in each step. Alternatively, we have $k'=l'=0$ directly from reflection invariance.
\begin{figure}[h]
    \centering
    \includegraphics[width=480pt]{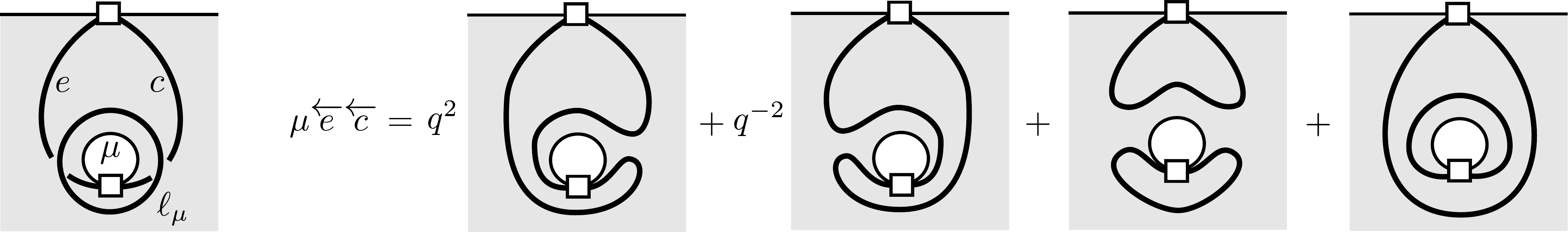}
    \caption{}\label{ellmu}
\end{figure}
\epr

\brem
When there is no circle boundary component, the algebra $ {\uSSR}$ and the quantum trace $\utr^A_\Delta$ were defined in \cite{Le:Qtrace}. The algebra  $ {\uSSR}$  is also studied in \cite{PS2}. 
\erem

\section{Dehn-Thurston coordinates and modification}\label{secDT}

Let $\Sigma_{g,m}$ denote a compact oriented surface of genus $g$ with $m$ boundary components.
To understand the skein algebras of $\fS=\Sigma_{g,m}$ we need a parameterization of the basis $B(\fS)$, similar to the Dehn-Thurston (DT) parameterization. The widely used DT coordinates do not behave well under skein products. In this section we will
 introduce a {\bf modified} version of DT coordinates which will be shown in later sections to pick up the highest degree term in the product of skeins.

Recall that $B(\fS)$ is the set of isotopy classes of unoriented, compact,  1-dimensional  proper submanifolds of $\Sigma_{g,m}$ which intersect each  boundary component at most one point.

We construct our DT coordinates for elementary pieces, which we call basic DT pairs of pants. Then we combine them together to get the global DT coordinates for $B(\fS)$.

\subsection{Three basic DT pairs of pants}

The surface $\Sigma_{0,3}$ is called a pair of pants, see Figure ~\ref{fig:pants}.
 
\begin{figure}[h]
    \centering
    \includegraphics[width=380pt]{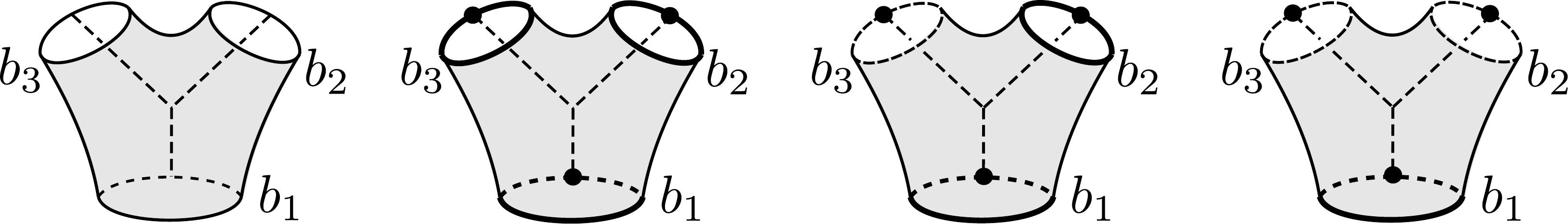}
    \caption{From left to right $\Sigma_{0,3}$, $\PP_3, \PP_2, \PP_1$}
    \label{fig:pants}
\end{figure}

A {\bf DT pair of pants $\PP_j$}, for $j\in \{1,2,3\}$, is the surface $\Sigma_{0,3}$ with boundary components $b_1, b_2, b_3$, equipped with  a $Y$-graph embedded as in Figure \ref{fig:pants}, where  the boundary components $b_i$ with $ i \le j$ are declared to be {\bf bold}, while the remaining ones are {\bf dashed}.   %A boundary component of $\PP_j$ is {\bf bold} (respectively {\bf dashed}) if   it is attached to a bold (respectively dashed) edge of $Y$.
Up to isotopy, the identity is the only self-diffeomorphism preserving the orientation, the $Y$-graph, and the enumeration of the boundary components.  Ignoring the enumeration of the boundary components,  the group of such self-diffeomorphisms of $\PP_3$  is the cyclic group $\BZ/3$, generated by a rotation by $120^o$  in Figure \ref{fig:pants}.

% The endpoint of $Y$ on a bold boundary component is called a {\bf bold vertex}.
 
 If $c$ is a loop, ie a simple closed curve, on a an oriented surface $\Sigma$, then by the result of {\bf cutting $\Sigma$ along $c$} we mean a surface $\Sigma'$ having boundary components $c', c''$ and a diffeomorphism $f: c' \to c''$ such that $\Sigma= \Sigma'/(u= f(u), u\in c')$, where $c$ is the common image of $c'$ and $c''$.

\subsection{DT datum}\label{Sec;DTdatum} We introduce the notion of a DT datum, which will determine DT coordinates. We will make the following assumption\\
\be (g,m)\not \in \{  (0, k), (1,0), k \le 4\}
\label{eqExcep}
\ee

\def\SC{{\fS_\cC}}
\def\cCt{{\cC^{(2)}}}

A {\bf pants decomposition} of $\fS=\Sigma_{g,m}$ is a maximal collection $\cC$ of disjoint non-trivial, non-peripheral loops which are pairwise non-isotopic.
%Fix a pants decomposition $\cC=\{ c_1, \dots, c_r\}$ of $\fS$. 
By cutting $\fS$ along all $c\in \cC$ we get a compact oriented surface $\SC$ whose connected components are copies of $\Sigma_{0,3}$, with a projection $\pr: \SC \onto \fS$. A component $c\in \cC$ lifts  to two boundary components $c', c''$ of $\SC$, and  denote by $\cCt$ the set of all such lifts. 
A connected component of $\SC$ is called  a {\bf face} of the pants decomposition $\cC$.

A triple $(a,b,c)\in \cC$ is {\bf triangular} if they are the images of the three boundary components of a face under the projection $\SC \onto \fS$. Note that  two of the three $a,b,c$ might be equal.

%A face is of {\bf type $j$} if it has $j$ cut boundary components.

A {\bf dual graph $\Gamma$ of $\cC$} is a trivalent graph embedded into $\fS$, transverse to each $c\in \cC$  such that its preimage in each face $\tau$ is a $Y$-graph, denoted by $\Gamma_\tau$, as  in Figure~\ref{fig:pants}. Each face $\tau$, equipped  with the graph $\Gamma_\tau$ is one of $\PP_1, \PP_2, \PP_3$, where the bold boundary components are declared to be the ones in $\cCt$.

\bdf  \label{defLength}

 A DT datum of $\fS$ consists of a pants decomposition $\cC$ and a dual graph $\Gamma$.

Fix a DT datum $(\cC, \Gamma)$. The  {\bf bold vertices} of $(\Gamma, \cC)$ are
elements of  $\Gamma \cap (\bigcup_{c\in \cC} c)$, as well as 
their lifts in $\SC$.

For a simple diagram  $\al$  on $\fS$ its 
{\bf length coordinate}  at $c\in \cC$ is $n_\al(c) :=I(\al, c)$.

Such an  $\al$ is good with respect to $(\cC, \Gamma)$ if $\al$ does not contain any bold vertex and $\al$ is taut with respect to $\cC$, meaning  $|\al \cap c| = n_\al(c)$ for all $c\in \cC$.
\edf

The following is standard, and is the basis of all the definitions of DT coordinates.

\blem \label{rSlides}
 Two good simple diagrams are isotopic in $\fS$ if and only if they are related by a sequence of t-slides and loop-slides as seen in Figure \ref{fig:slide}.

\elem

\begin{figure}[h]
    \centering
    \includegraphics[width=330pt]{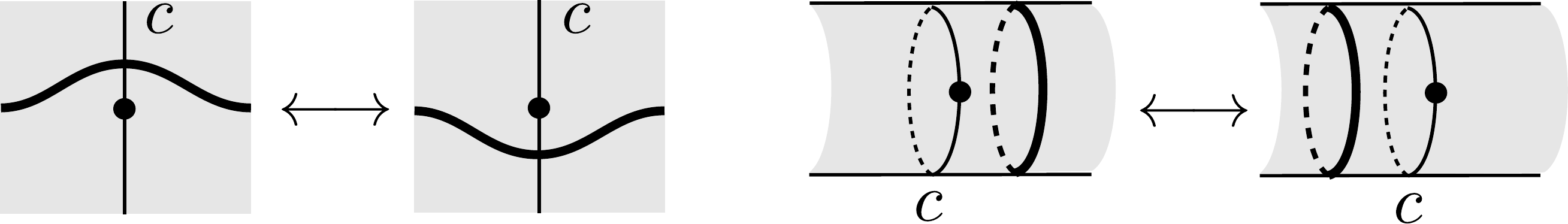}
    \caption{Left: $t$-slide, Right: loop-slide}
    \label{fig:slide}
\end{figure}

\def\hC{{ \hat \cC}}
\def\tiC{{ \tilde \cC}}

An easy Euler characteristic count shows that
\be 
|\cC|= 3g-3+m .
\ee

\subsection{Difference between our twist coordinates and the usual ones}
Before giving the detailed definition, let us point out the difference between our twist coordinates and the ones used for examples in \cite{Luo,FKL,HP}.

In the DT pair of pants $\PP_3$ with boundary components $b_1, b_1, b_3$, a curve is {\bf standard}
 if it is one of the three curves $\ell_1,a_{23}, a_{11}$ of Figure \ref{fig:standard03} or their images under the actions of $\BZ/3$.
 
\begin{figure}[h]
    \centering
    \includegraphics[width=380pt]{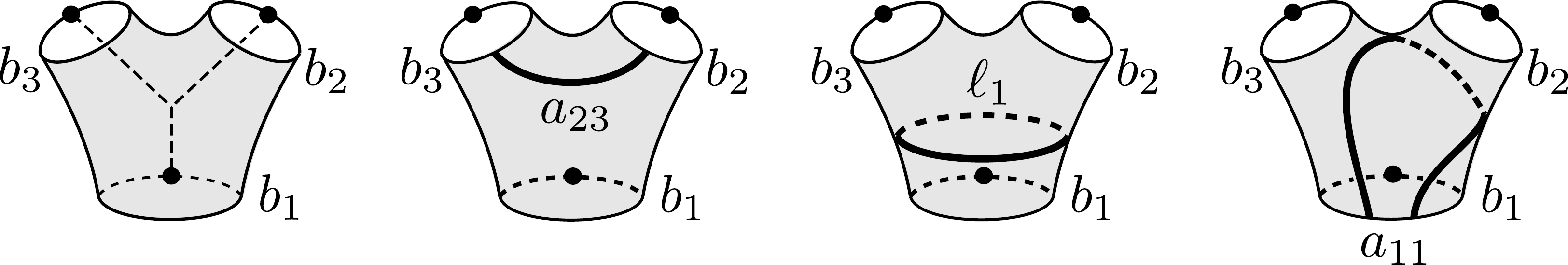}
    \caption{Standard curves on $\PP_3$.}
    \label{fig:standard03}
\end{figure} 
 
 For $i=1,2,3$, the arc $a_{ii}$ will be called a {\bf standard return arc}; it has two endpoints on $b_i$, and we say that it 
 {\bf approaches $b_{i+1}$}. Here indices are taken modulo 3.

In the usual definition of \cite{Luo,FKL}, the twist coordinate $t'_i(\al)$, where $\al$ is a simple diagram,  is defined so that the contribution of each standard return arc  is 0. Our modified version of the twist coordinate is the following
\be 
t_i(\al)= 2 \times  [ t'_i(\al) + \text{number of  return arcs approaching}\ b_{i}].
\ee
This modification is very important as it will equal to the highest degree term  in certain filtrations of the skein algebra. The factor 2 is used so that we can later accommodate the case of $\PP_1$, where the twist coordinates might be odd.

\subsection{Patching together strongly simple diagrams} 
\label{ssStrong} We show how to patch together strongly simple diagrams on faces to get a strongly simple diagram on $\fS$.

\def\Pb{\check \tau}
\def\pPb{\partial{\Pb  }}
\def\PPb{{\check{\PP  }}}

For a face $\tau$ of the DT datum let $\Pb$ be obtained from $\tau$ by removing all the bold vertices.

Recall that $\uB(\Pb)$ is the set of all isotopy classes of strongly simple diagrams.

For a bold boundary component $c$ we define the {\bf twist on $c$  map} $\theta_c: \uB(\Pb) \to \uB(\Pb)$, which is a permutation of $\uB(\Pb)$ as follows. If $\al \cap c=\emptyset$ then $\theta_c(\al)= \al$; otherwise $\theta_c(\al)$ is given by
$$\begin{array}{c}   \incl{1.22cm}{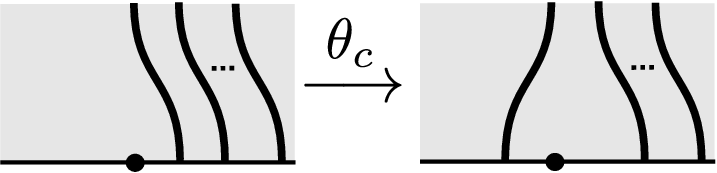}\end{array}.$$

\def\patch{{\mathsf{patch}}}
Suppose $\al= (\al_\tau)_{\tau\in \cF}\in \prod_{\tau \in \cF} \uB(\Pb)$, where $\cF$ is the set of all faces. Let $\al^\sqcup=\sqcup _\tau \al_\tau$.
 Then $\al$ is  {\bf matched} if $|\al^\sqcup \cap c'| = |\al^\sqcup \cap c''|$ for each $c\in \cC$.
For matched $\al$  we can patch the $\al_\tau$ together to get  $\patch(\al) \in B(\fS)$. Let $(\prod_{\tau \in \cF} \uB(\Pb))^*$ be the set of all matched elements. From Lemma \ref{rSlides} we get

\blem \label{rSlides2}
Two elements $\al, \al'\in (\prod_{\tau \in \cF} \uB(\Pb))^*$ are patched to the same element of $B(\fS)$ if they are related by a sequence of 
\begin{itemize}\item  t-moves:
$$
\al \leftrightarrow \theta_{c'}^{-1} \circ \theta_{c''}(\al) ,\ c\in \cC.
$$
\item loop-moves: Assume $\al$ contains $\ell_{c'}$, then the loop-move at $c$ is
$$ \al \leftrightarrow (\al \setminus \ell_{c'} ) \cup \ell_{c''},$$

\end{itemize}
\elem

\subsection{DT coordinates for $\PP_3$} \label{ssP3}
\def\rot{{\mathrm{rot}   }}
\def\loo{\ell}
\def\arcc{\omega}
\def\Add{{\mathsf{Add}}}

We now introduce DT coordinates for strongly simple diagrams on $\PPb_3$.

Recall that $\PP_3$ has 3 bold boundary components $b_1, b_2, b_3$.
Let  $\theta_i$ be the twist map on $b_i$. For each $n\in \BZ$, we  call $\theta_i^n(a_{ii})$  a {\bf return arc approaching} $b_{i+1}$.

For $i=1,2,3$ define $\Add_i: \BZ^3 \to \BZ$ by (with indices taken mod 3)
\be
\Add_i(n_1, n_2, n_3  ) = \max(0,  {n_{i-1} - n_{i} - n_{i+1}} ). \label{eqAdd3}
\ee
For a strongly simple diagram $\al$, if $n_i= |\al \cap b_i|$ then
 the number of return arcs approaching $b_{i}$ is $\frac 12 \Add_i(n_1, n_2, n_3  ).$

A strongly simple diagram is {\bf standard} if each connected component of it is standard, i.e. one of  $a_{ij}, \ell_i$.  We use the convention $a_{ij}= a_{ji}$.

\bpro 
\label{rP3}
(a) There is a unique injective  map
$$ \nu: \uB( \PPb_3) \to \BN^{3} \times \BZ^{3}, \quad \nu(\al)= (n_1(\al),n_2(\al), n_3(\al), t_1(\al), t_2(\al), t_3(\al)), $$
satisfying the following conditions $(1)-(6)$:

\begin{enumerate}

\item length coordinates: For bold $b_i$ we have
$ 
n_i(\al) = |\al \cap b_i|.
$
\item disjoint additivity: If $\al_1,\dots, \al_k$ are components of $\al$ then $\nu(\al) = \sum_{j=1}^k \nu(\al_j)$.
\item The twist increases the twist coordinate by 2 : The twist $\theta_i$ on a bold $b_i$ changes only the coordinate $t_i$, and the change is given by
\be 
t_{ i }(\theta_i (\al)) =  \begin{cases} t_i( \al) +2,  & n_i(\al) >0 \\
t_i( \al),  & n_i(\al) =0.\end{cases}
\ee
\item Twist coordinates of boundary curves: For bold $b_i$,
\be 
t_j(\ell_i)=  2 \delta_{ij}. \label{eqloop}
\ee

\item Standard straight arcs: For all applicable $i,j,k$ with $j\neq  k$,
\begin{align}
t_i(a_{jk}) &= 0  
\end{align}

\item Standard return arcs: With indices taken modulo 3,
\begin{align}
t_i(a_{jj}) &=  2 \delta_{i-1,j}. \label{eqDT3}
\end{align}
\end{enumerate}

(b) The image of $\Lambda_3 :=\nu(\uB(\PPb_3))$ inside $\mathbb{N}^3\times \mathbb{Z}^3$ is  the subset  constrained by 

\begin{itemize}
\item[(i) ] $n_1+n_2+n_3$  and all $t_i$ are even, and

 \item[(ii)] if $n_i=0$ then $t_i\geq \Add_i(n_1, n_2, n_3)$.
 
\end{itemize}

Moreover, the set $\Lambda_3$ is a submonoid of $\BZ^6$.
\epro

Note that in $\PP_3$,  each of $b_1, b_2,$ and $b_3$ are bold.   To ensure a consistent statement of the proposition for $\PP_2$ and $\PP_1$ later, we use ``For bold $b_i$" in the formulation.

\def\Lo{{\mathring \Lambda}}

\bpr 

(a) Define $n_i(\al)$ by condition (1).
The standard twist coordiantes, $t'_i(\al)$ (in e.g. \cite{Luo,FKL}), are defined as follows.  There are integers $k_1, k_2, k_3$ such that $(\theta_1)^{-k_1} (\theta_2)^{-k_2} (\theta_3)^{-k_3}(\al)  $ is standard. If $n_i(\al)\neq 0$ then define $t'_i(\al)= k_i$. Otherwise, $t'_i(\al)$ is the number of components of $\al$ isotopic to $b_i$ in $\PP_3$.

The standard DT coordinates $(n,t')$  satisfy  all the requirements, except that the right hand side of 
\eqref{eqDT3}  is $0$ and the right hand side of \eqref{eqloop} is $\delta_{ij}$, and the slide increases the twist coordinate by 1.   Defining the new twist coordinates by 
\be t_i(\al) = 2 t'_i(\al) + \Add_i(n_1(\al), n_2(\al), n_3(\al)).
\label{eqNewt}
\ee
Then $\nu$ is injective and  satisfies (1) and (3)-(6). Let us prove (2). Clearly the $n_i(\al)$ coordinates are additive. Let us prove $t_i$ is additive. Note that $d= \Add_i(n_1(\al), n_2(\al), n_3(\al))$ is twice the number of return arcs approaching $b_i$. By definition
$$ t_i(\al_l) =   \begin{cases} 2t'_i(\al_l) + 2  &  \al_l \ \text{is a return arc approaching} \ b_i \\ 
2t'_i(\al_l) & \text{otherwise}.
\end{cases}
$$

Hence, from the additivity property of $t'_i$, $$t_i(\al) = 2t'_i(\al) + d =   {\sum_{l=1}^k 2t_i'(\al_l)+d= \sum_{l=1}^k t_i(\al_l).}$$

(b) The image of the old coordinates $(\vec{n},\vec{t'})$ is the submonoid $\Lambda'$ satisfying : $n_1+ n_2 + n_3\in 2\BZ$, and $t_i'\ge 0$ whenever $n_i=0$. These translate to the conditions (i) and (ii) for $(\vec{n},\vec{t})$.

Let us now prove $\Lambda_3$ is a monoid.  There are  two observations: \\ 
First, 
$\{ (\bn, \bt)  \in (\BN_{>0})^3\times(2\BZ)^3 \mid n_1 + n_2 + n_3 \in 2 \BN\} $ is  a subset of $\Lambda_3$.\\
Second,  the function $\Add_i$ satisfies $\Add_i(\bn' + \bn'') \le \Add_i(\bn') + \Add_i(\bn'')$.

 Let $(\bn', \bt'), (\bn'', \bt'')\in \Lambda_3$. We have to show $(\bn, \bt)= (\bn'+ \bn'', \bt'+ \bt'') \in \Lambda_3$.

  If $\bn\in (\BN_{>0})^3$ then $(\bn, \bt) \in \Lambda_3$ by the first observation.
 Assume, say $n_3=0$. Then $n_3'= n''_3=0$. Then $t_3'\ge  \Add_3(\bn'), t_3''\ge \Add_3(\bn'')$. Hence $ t_3= t'_3+ t''_3 \ge  \Add_3(\bn)$ by
  the second observation. This completes the proof.
\epr

% \brem The added term $\Add_i(n_1(\al), n_2(\al), n_3(\al))$ in \eqref{eqNewt} is twice the number of return arcs approaching $b_i$. \erem

\subsection{Coordinates in $\PP_2$} \label{secP2}

We now introduce  DT coordinates for strongly simple diagrams on $\PPb_2$, whose boundary components $b_1, b_2$ are bold while $b_3$ is dashed. 
 See Figure~\ref{fig:standard02}. 

\begin{figure}[htpb!]
    %\centering
    \includegraphics[height=1.7cm]{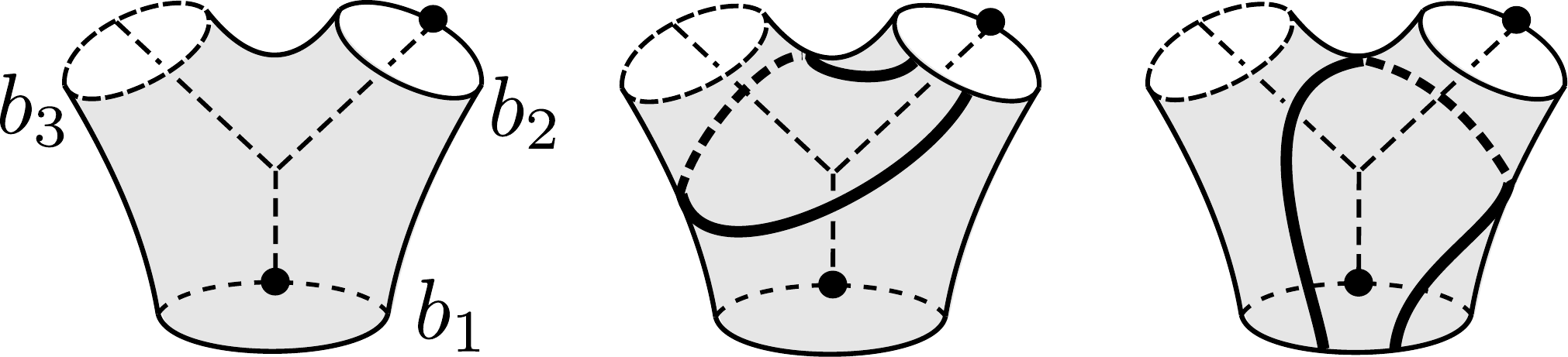}
    \caption{left: $\PP_2$ \quad middle :$a_{22}$\quad right: $a_{11}$}
    \label{fig:standard02}
\end{figure}

We use the same standard arcs and curves $a_{ij}, \ell_i$ as in the case of $\PP_3$. Since $b_3$ is dashed, a strongly simple diagram has at most one endpoint on $b_3$. In particular, we never encounter $a_{33}$ nor $\ell_3$. The parameterized set $\uB(\PPb_2)$ is not the one usually considered in the literature. 
%Note that $\gamma_1= \theta_1(\gamma'_1)$.

%\FIGjpg{standard2}{Standard curves on $\PP_2$}{2.5cm}

\bpro\label{rP2}
 (a) There is a unique injective  map
$$ \nu: \uB( \PPb_2) \to \BN^{2} \times \BZ^{2}, \quad \nu(\al)= (n_1(\al),n_2(\al),  t_1(\al), t_2(\al)), $$
satisfying the conditions (1)-(5) of Proposition \ref{rP3}, and in addition, the twist values  of the standard return arcs are
\begin{align}
  {(t_1(a_{11}), t_2(a_{11})))=(2,0)}, \quad (t_1(a_{22})), t_2(a_{22})))=(0,0). 
\end{align}

  {(b) The image  $\Lambda_2:=\nu(\uB(\PP_2))$ is the submonoid of $\mathbb{N}^2\times \mathbb{Z}^2$ constrained by  all $t_i$ are even, and  if $n_i=0$ then $t_i\geq 0$ $(i\in \{1,2\})$.}

\epro
\bpr 
(a) Let $\al \in \uB(\PP_2)$. Note that $\PP_2 \subset \PP_3$. Considering $\al$ as a strongly simple diagram in $\PPb_3$, we can define the standard DT coordinates $n_i(
\al), t'_i(\al)$ for $i=1,2,3$. The definition of strongly simple implies that $t'_3=0$ and $n_3(\alpha)\in \{0,1\}$. Since $n_1(\alpha)+n_2(\alpha)+n_3(\alpha)$ is even,  $n_3\in \{0,1\}$ is uniquely determined. Thus $\al$ is totally determined by $n_1(\al), n_2(\al), t'_1(\al), t'_2(\al)$.

Define 
\be 
t_1(\al)= {2}t'_1(\al)+\max(0, n_1(\al)-n_2(\al)-n_3(\al)),\quad t_2(\al)= {2}  t'_2(\al).
\label{eqPP2}
\ee
(This modification is different from the $\PP_3$ case. This is forced upon us because of the later compatibility between the product and the lead term.)

We can easily check that  $\nu$ satisfies all desired properties.

(b) The claim immediately follows from (a) and properties of $n_i$ and $t'_i$. \epr

Equation \eqref{eqPP2} can be rewritten as $t_i= 2 t'_i + \Add_i(n_1, n_2)$, where
 $\Add_i\colon \BN^2 \to \BN$, for $i=1,2$, is defined as follows. First, let $n_3=0$ or $1$ according as $n_1+n_2$ is even or odd. Then
\be 
\label{eqAdd2}
\Add_1(n_1, n_2) = \max (0, n_1- n_2-n_3), \  \Add_2(n_1, n_2) = 0.
\ee
Note that if $n_i=0$ for an index $i=1,2$ then $\Add_i(n_1, n_2)=0$.

\subsection{Coordinates in $\PP_1$} \label{ssP1}

We now introduce  the  DT coordinates for strongly simple diagrams on $\PPb_1$, where $b_1$ is bold and $b_2, b_3$ are dashed. See Figure \ref{fig:standard01}.  We use the same notation $a_{ij},\ell_i$ for the curves as in the case of $\PP_3$.  We don't have $a_{22}, a_{33}, \ell_2, \ell_3$.

\begin{figure}[htpb!]
    %\centering
    \includegraphics[height=2.1cm]{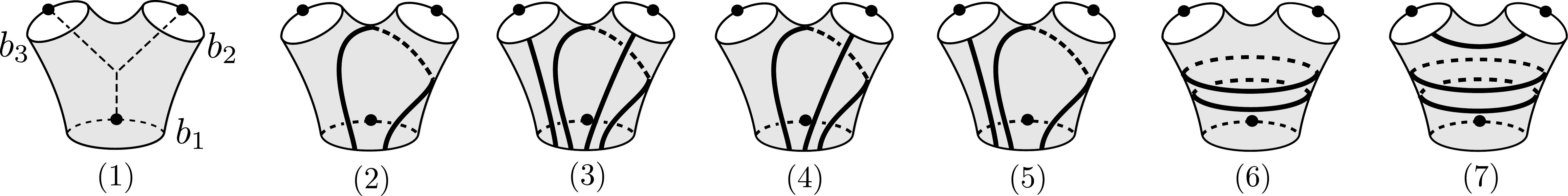}
    \caption{ $\PP_1$ followed by the curves with coordinates $(n,t)$ which are (even, even), (even, odd), (odd, even), (odd, odd),  (0, even), and (0, odd)}
    \label{fig:standard01}
\end{figure}
 
\no{
\FIGc{PP1}
{$\PP_1$ followed by the curves with coordinates $(n,t)$ which are (even, even), (even, odd), (odd, even), (odd, odd),  (0, even), and (0, odd)}{2.3cm}
}

\bpro\label{rP1}
 (a) There is a unique injective  map
$$ \nu: \uB( \PPb_1) \to \BN \times \BZ, \quad \nu(\al)= (n_1(\al),  t_1(\al)), $$
satisfying the conditions (1)-(4) of Proposition \ref{rP3}, and additionally,
\begin{align}\label{eq:t_1}
 t_1(a_{11}) =t_1(a_{12}) =  0, t_1(a_{13}) = -1, \ 
 t_1(a_{23}) = 1
\end{align}

(b) The image  $\Lambda_1:=\nu(\uB(\PP_1))\subset \mathbb{N}\times \mathbb{Z}$ is a submonoid defined by: If   {$n_1=0$ then $t_1\ge 0$}.

\epro
\bpr 
 Let $\al \in \uB(\PPb_1)$. Define $n_1(\al)= |\al\cap b_1|$. 
We give a geometric definition of $t_1$.    This geometric definition will give direct justification for properties (1)-(4) of the lemma.  We proceed in two separate cases:

Case 1: $n_1(\al)= 0$. See (6)-(7) of Figure \ref{fig:standard01}.  Here $\al$ consists of $l$ copies of $\ell_1$, and $d$ copies of $a_{23}$ where $d=0$ or 1.  Define $t_1(\al)= 2l+d$.

Case 2:  $n_1\ge 1$.  There is a unique $k\in \BZ$ such that  $\theta_1^k (\al) $ is standard. Thus $\theta_1^k (\al) $ consists of $l$ copies of $a_{11}$, $d_{2}$ copies of $a_{12}$ and $d_{3}$ copies of $a_{13}$, where $d_2, d_3 \in \{0,1\}$. See (2)-(5) of Figure \ref{fig:standard01}. Define $t_1(\al) = -2k - d_3$.

With these explicit values, one can easily check all the claims of the Proposition.
\epr

We note that the definition implies
\be 
t_1(\al) \equiv |\al \cap b_3| = n_3 \mod 2,
\label{eqn3}
\ee
which uniquely determines $n_3\in \{0,1\}$, and hence also  $n_2\in \{0,1\}$, since $n_1+n_2+n_3\in 2\BZ$. 
Besides
\be t_1(\al) - 2t'_1(\al) = \#(a_{23}) - \# (a_{13})= \frac 12 [\max(0, n_2+n_3-n_1) - \max (0, n_1 + n_3 -n_2)   ]
\label{eqAdd1}
\ee

Unlike the case of $\PP_2$ and $\PP_3$, the right hand side of \eqref{eqAdd1} does not depend solely on $n_1$.

\subsection{Dehn-Thurston Coordinates} We now can define DT coordinates for simple diagrams on $\fS= \Sigma_{g,m}$, equipped with a DT datum $(\cC, \Gamma)$. Recall that we are assuming
\be (g,m) \neq (1,0),(0,k), \ \text{for} \ k \le 4.\label{eqAss1}
\ee

  Recall that $\SC$ is the disjoint union of faces of the pants decomposition, and each $c\in \cC$ lifts to two boundary components $c', c''$ of $\SC$. The set of all $c', c''$ is denoted by $\cCt$. A boundary component $b\in \cCt$ is said of type $\PP_i$ if the face containing $b$ is of type $\PP_i$. Since $(g,m) \neq (0,4)$ one of $c',c''$ must not be of type $\PP_1$, and we will assume the $c'$ is not of type $\PP_1$ for all $c\in \cC$.

Let $\tau$ be a face of the pants decomposition. 
An identification of $(\tau, \Gamma \cap \tau)$ with one of the three DT pairs of pants $\PP_1, \PP_2, \PP_3$ is called a {\bf characteristic map} of $\tau$. For type $\PP_1$ and $\PP_2$ there is a unique characteristic map up to isotopy.  
For type $\PP_3$ there are three possibilities of the characteristic maps, but the definition given below is independent of the choice. 

Let $b\in \cCt$ be a bold component of $\partial\tau$. Let $\al\subset \fS$ be a  simple diagram in good position with respect to $(\cC, \Gamma)$. 
Define $\al_\tau= \pr^{-1}(\al) \cap \tau$. 
Assuming that under the characteristic identification $b$ is the $i$-th boundary component,
let
\be 
%n(\alpha;b)= |\al \cap c|=n_i(\al_\tau), \quad 
t(\al;b)= t_i(\alpha_\tau).
\ee 
Define $\nu(\al) := (n_\al, t_\al) \in \BN^\cC \times \BZ^\cC$, where
\begin{align}
& n_\al: \cC\to \BN,& &\text{given by}&  \ n_\al(c) &= |\al \cap c| \notag   \\
&t_\al: \cC \to \BZ, &  &\text{given by}& t_\al(c)&= t(\al;c') + t(\al;c''). 
\label{eqtalc}
\end{align}

\def\fA{{\mathfrak A}}
\def\fB{{\mathfrak B}}

\bpro \label{rDTcoord}
Let $(\cC, \Gamma)$ be a DT datum of $\fS= \Sigma_{g,m}$ where $(g,m) \neq (1,0),(0,k)$ for $k \le 4$. Then $\nu$ gives a well-defined injective map
$ \nu: B(\fS) \to \BN^\cC \times \BZ^\cC$.

\epro

\begin{proof} To prove the well-definedness  we show that if $\al$ and $\al'$ are isotopic then $\nu(\al) = \nu(\al')$. 

First, since $n_\al(c) = I(\al, c)$, clearly $n_\al(c)= n_{\al'}(c)$.
By Lemma \ref{rSlides} the two diagrams $\al$ and $\al'$ are related by a sequence of t-slides and loop-slides, which preserve  $t_\al$ by Properties (3) and (4) of Proposition \ref{rP3}. This shows $\nu$ is well-defined.

Now we prove the injectivity. Assume $\nu(\al)= \nu(\beta)$. %\red{\stno{Consider the collections $\fA=(\al_\tau)_{\tau\in \cF}$ and $\fB=(\beta_\tau)_{\tau\in \cF}$.}} 
Let $c\in \cC$. Then $ {s}:= t(\al; c') - t(\beta;c')$ is even since both $t(\al; c')$ and $t(\beta;c')$ are even, as $c'$ is not of type $\PP_1$.
From
$$t(\al;c') + t(\al; c'') = t_\al(c)= t_\beta(c)= t(\beta;c') + t(\beta; c'')$$ we have  $ {s}= t(\beta;c'')-  t(\al; c'')$.
 Then after $ {s}/2$ twists on $c$, we can bring $\al$ to $\beta$ with $t(\beta;c')= t(\al;c')$ and $t(\beta;c'')= t(\al;c'')$. Repeating this procedure to all $c\in \cC$, and observe that twists are isotopies, we can isotope $\beta$ to $\gamma$ with $t(\gamma;c')= t(\al;c')$ and $t(\gamma;c'')= t(\al;c'')$ for all $c\in \cC$. This shows $\gamma_\tau$ is isotopic to $\al_\tau$ for all faces $\tau$. Hence $\al$ is isotopic to $\gamma$, and moreover to $\beta$.
\end{proof}

To describe the image of $\nu$ let us call $c\in \cC$ {\bf even} if both $c'$ and $c''$ are not of type $\PP_1$. The reason is if $c$ is even, then $t_\al(c)$ is even for all $\al \in B(\fS)$. Note that if $c$ is not even then $c''$ must be of type 1, since $c'$ is never of type  $\PP_1$.

For $\bn: \cC\to \BN$ and  $b\in \cCt$,  
we define $\Add(b; \bn)\in 2 \BZ$ as follows. Let $\tau$ be the face containing $b$. If $\tau= \PP_1$
 let $\Add(b; \bn)=0$. Assume  $\tau= \PP_2$ or $\tau= \PP_3$. Under the characteristic map the bold boundary components of $\tau$ are $b_1,$ and $b_2$ if $\tau= \PP_2$ and $b_1, b_2$, and  $b_3$ if $\tau= \PP_3$.   Additionally, let $b= b_i$ for some index $i$. 
 Then define 
\begin{align}
\Add(b; \bn) = \begin{cases} \Add_i(\bn(\pr(b_1)), \bn(\pr(b_2))   )  & \text{if} \ \tau = \PP_2 \\ 
\Add_i(\bn(\pr((b_1)), \bn(\pr(b_2)), \bn(\pr(b_3))   )   & \text{if} \ \tau = \PP_3 
\end{cases}  
\end{align} 
If $\bn(\pr(b))=0$ and $b$ is of type $\PP_1$ or $\PP_2$, then $\Add(b; \bn) = 0$.

\bpro \label{rDTcoord2} Assume the assumptions of Proposition \ref{rDTcoord}. 

(a) The image $ \nu( B(\fS)  )$  is the submonoid $\Lambda_{\cC,\Gamma}$ of $ \BN^\cC \times \BZ^\cC$ consisting of $(\bn, \bt)$ satisfying
\begin{enumerate}

\item [(i)] if $(a,b,c)\in \cC$ is a triangular triple then $\bn(a)+ \bn(b) + \bn(c) \in 2 \BZ$,

\item [(ii)] if $c$ is even then $\bt(c) \in 2\BZ$, and

\item [(iii)] if $\bn(c) =0$ then $\bt(c) \ge \Add(c'; \bn) + \Add(c''; \bn)$.

\end{enumerate}

(b) For each $i=1,\dots, r$, there is a piecewise linear function $h_i: \BR^{\cC \cup \cM}\to \BR$ such that for all $\al \in B(\fS)$ we have
\be t_i(\al) - 2t'_i(\al) = h_i(I_\al),\  \text{where}  \ I_\al(c) = I(\al, c).
\ee

\epro

Here a function $h: \BR^k \to \BR$ is {\bf piecewise linear} if it is continuous and there are a finite number of co-dimension 1 linear subspaces $H_1, \dots, H_d\subset \BR^k$ such that in each connected component of the complement of all $H_1, \dots, H_d$ the function is linear. 

Also $t'_i$ is the standard twist coordinate defined as in Subsection \ref{ssP3}.
\def\tbn{{\tilde\bn}}
\def\tbt{{\tilde\bt}}

\begin{proof} (a) First we have $ \nu( B(\fS)  ) \subset \Lambda_{\cC,\Gamma}$. This follows from the descriptions of $\Lambda_j$ in Propositions \ref{rP3}, \ref{rP2}, and \ref{rP1}. 
In particular, (iii) follows from \eqref{eqtalc}. 

\no{
 For each element in $B(\fS)$ choose the a fixed representative $\al$ which  is a simple diagram good with respect to $\cC, \Gamma$. Recall that   $\al_\tau= \pr^{-1} (\al) \cap \tau$ for a face $\tau$.
Then $\nu$ is the composition of the following maps
\be 
B(\fS) \xrightarrow{  } \prod_{\tau \in \cF}\Lambda_\tau \subset  \BN^ \cCt \times \BZ^\cCt   \xrightarrow{f  } \BN^\cC \times \BZ^\cC.
\ee
Here, if  $\nu_\tau$ denotes our modified DT coordinates on the face $\tau$, then 
\begin{itemize}
\item the first map sends $\al$, the chosen representative, to $\prod_{\tau} \nu_\tau(\al_\tau)$,
\item $\Lambda_\tau =\Lambda_j$ for $\tau = \PP_j$, and
\item $f: \BN^ \cCt \to \BN^ \cC, g: \BZ^ \cCt \to \BZ^ \cC $ are the $\BN$-linear map defined by
\be 
f(\bn)(c)= \bn(c'), \quad g(\bt)(c)= \bt(c')+ \bt(c'').
\ee
\end{itemize}
The linearity of $f,g$ implies that the $\Lambda_{\cC,\Gamma}= \nu( B(\fS)  )$  is a submonoid of $ \BN^\cC \times \BZ^\cC$. The descriptions of $\Lambda_j$ in Propositions \ref{rP3}, \ref{rP2}, and \ref{rP1} implies the properties (1), (2), and (3). In particular, (3) follows from \eqref{eqtalc}.
}

Let us now prove the converse $ \Lambda_{\cC,\Gamma}\subset \nu( B(\fS)  )$.
Assume $(\bn, \bt) \in \Lambda_{\cC,\Gamma}$. Recall that $c'$ is not of type $\PP_1$ for 
$c\in \cC$.  Define $\tbn, \tbt : \cCt \to \BZ$ by
\begin{align*}
\tbn(c')= \tbn(c'')= \bn(c), \quad \tbt(c')= \Add(c';\tbn), \ \tbt(c'') = \bt(c) - \tbt(c').
\end{align*}
From (iii) , $\tbt(c')= \Add(c';\tbn),  and \tbt(c') + \tbt(c'')= \bt(c)  $ we have, for all $b\in \cCt$,
\be 
\text{if $\tbn(b)=0$ then} \  \tbt(b) \ge  \Add(b;\tbn).
\label{edIneq5}
\ee

For each face $\tau$ of the {pants decomposition}, let $(\tbn, \tbt)_\tau$ be
the restriction  of $(\tbn, \tbt)$ to the bold boundary components of $\tau$. Let us show that  $(\tbn, \tbt)_\tau \in \Lambda_\tau$, where $\Lambda_\tau$ is identified with $\Lambda_j$ under the characteristic map of $\tau$. 

Assume $\tau=\PP_3$.  From (i) we have $\tbn(b_1)+ \tbn(b_2) + \tbn(b_3)\in 2 \BZ$. If $b_i= c'$ for some $c\in \cC$ then $\tbt(b_i) = \Add(c'; \bn)\in 2 \BZ$. If $b_i= c''$ then, since $b_i$ is not of type $\PP_1$, $c$ must be even. Hence $\tbt(b_i)= \bt(c) - \Add(c';\tbn)$ is even.
Finally, when $\tbn(b_i)=0$, then we have $\tbt(b_i) \ge  \Add_i(\tbn)$ due to \eqref{edIneq5}. All the requirements for being an element of $\Lambda_\tau$ are satisfied, and  we have $(\tbn, \tbt)_\tau \in \Lambda_\tau$.

Assume  $\tau=\PP_2$. The exact same  argument for the $\PP_3$ case  shows that $\tbt(b_i)$ are even.   When $\tbn(b_i)=0$,  due to \eqref{edIneq5} we have $\tbt(b_i) \ge  \Add(b_i;\tbn)=0$. Thus we have $(\tbn, \tbt)_\tau \in \Lambda_\tau$.

Assume $\tau=\PP_1$, which has only one bold edge $b_1$. When $\tbn(b_i)=0$,  we have $\tbt(b_i) \ge  \Add(b_i;\tbn)=0$. Thus we have $(\tbn, \tbt)_\tau \in \Lambda_\tau$.

\no{
\blue{We proceed by directly checking the claim for each possibility of $\tau$.  

First assume that $\tau=\PP_3$ and recall that $\Lambda_3$ is the submonoid of $\mathbb{N}^3\times\mathbb{Z}^3 $ constrained by $n_1+n_2+n_3$ and all $t_i$ are even and if $n_i=0$ then $t_i\geq \Add_i(n_1,n_2,n_3)$.   Then we have a triangular triple$(x_1,x_2,x_3)$ where under the characteristic map $(x_1^{\epsilon_1},x_2^{\epsilon_2},x_3^{\epsilon_3})$ are all bold and cobound $\PP_3$,  where $\epsilon_i$ can denote either one or two primes.  Thus $(\tbn,\tbt)_\tau=(\bn(x_1^{\epsilon_1}),\bn(x_2^{\epsilon_2}),\bn(x_3^{\epsilon_3}),\tbt(x_1^{\epsilon_1}),\tbt(x_2^{\epsilon_2}),\tbt(x_3^{\epsilon_3}))$.   Then by definition $\tbn(x_1^{\epsilon_1})+\tbn(x_2^{\epsilon_2})+\tbn(x_3^{\epsilon_3})=\bn(x_1)+\bn(x_2)+\bn(x_3)$ is even.    Then we have both $\Add(x_i';\bn)$ and $\bt(x_i)-\Add(x_i';\bn)$ will be even, meaning $\tbt(x_i^{\epsilon_i})$ is even.   Now assume that $\tbn(x_i^{\epsilon_i})=\bn(x_i)=0$.  Then $\bt(x_i)\geq \Add(x_i';\bn)+\Add(x_i'';\bn)$.   Thus we have $\tbt(x_i^{\epsilon_i})\geq \Add( x_i^{\epsilon_i};\bn)=\Add_i(\bn(pr((x_1)),\bn(pr(x_2)),\bn(pr(x_3)))$ as desired.

Next assume that $\tau=\PP_2$ and recall that $\Lambda_2$ is the submonoid of $\mathbb{N}^2\times\mathbb{Z}^2 $ constrained by all $t_i$ are even and if $n_i=0$ then $t_i\geq 0$.   Then we have a pair $(x_1,x_2)$ where under the characteristic map $(x_1^{\epsilon_1},x_2^{\epsilon_2})$ are bold and cobound $\PP_2$,  where $\epsilon_i$ can denote either one or two primes.  Thus $(\tbn,\tbt)_\tau=(\bn(x_1^{\epsilon_1}),\bn(x_2^{\epsilon_2}),\tbt(x_1^{\epsilon_1}),\tbt(x_2^{\epsilon_2}))$.   Then we have both $\Add(x_i';\bn)$ and $\bt(x_i)-\Add(x_i';\bn)$ will be even, meaning $\tbt(x_i^{\epsilon_i})$ is even.   Now assume that $\tbn(x_i^{\epsilon_i})=\bn(x_i)=0$.  Then $\bt(x_i)\geq \Add(x_i';\bn)+\Add(x_i'';\bn)$.   Thus we have $\tbt(x_i^{\epsilon_i})\geq \Add( x_i^{\epsilon_i};\bn)=\Add_i(\bn(pr((x_1)),\bn(pr(x_2)))\geq 0,$ as desired.

Finally assume that $\tau=\PP_1$ and recall that $\Lambda_1$ is the submonoid of $\mathbb{N}\times \mathbb{Z}$ constrained by if $n_1=0,$ then $t_1\geq 0$.  Then we have $x_1$, where under the characteristic map $x_1''$ is bold and bounds $\PP_1$.  Then $(\tbn,\tbt)_\tau=(\bn(x_1),t(x_1)-\Add(x_1';\bn)$.    Then assume that $\bn(x_1)=0$, so $\bt(x_1)\geq \Add(x_i';\bn)+0.$  Thus we have $\tbt(x_1)\geq \Add(x_i';\bn)-\Add(x_i';\bn)=0$, as desired.  }
}

Hence $(\tbn, \tbt)_\tau$ determines a strongly simple diagram $\al_\tau$ in $\tau$. These $\al_\tau$'s can be patched together to give $\al \in B(\fS)$, with $\nu(\al) = (\bn, \bt)$.

(b)This follows immediately from the construction, using \eqref{eqAdd1}, \eqref{eqAdd2}, and \eqref{eqNewt}.
\end{proof}

\section{The three basic DT pairs of pants}
\label{secPPP}
Roughly speaking, to construct the embedding of Theorem \ref{thmSg} and Theorem \ref{thmSRY} we first cut the surface $\Sigma_{g,m}$ into pairs of pants, construct a similar map for each pair of pants, then patch them up to get the global map.  In this section we construct the ``local" maps for pairs pants, by composing the map $\utr^A_\Delta$ with a simple multiplicatively linear map.

\def\ddd{\mathbbm{d}}
\def\bnu{{\boldsymbol{\nu}}}
\def\YPj{{\cY(\PPb_j)}}

\subsection{Quantum tori associated to $\PP_j$}\label{ssQtPj} To each DT pair of pants $\tau$ we associate a quantum torus $\cY(\Pb {;\cR})$, which will be the target space of a quantum trace map.

 Recall that $\PP_j$, with $j=1,2,3$, is the DT pair of pants with the embedded $Y$ graph, with boundary components $b_1, b_2, b_3$, see Figure \ref{fig:standard03}. Here $j$ means the first $j$ boundary components are bold. Also $\PPb_j$ is obtained from $\PP_j$ by removing the endpoints of the $Y$-graph lying on the bold components. 
The $Y$-graph gives rise to the coordinate map 
$$ \nu : \uB(\PPb_j) \xra{\cong} \Lambda_j \subset \BN^j \times 
\BZ^j$$

Define the quantum torus $\cY(\PPb_3{})$, where  indices are taken mod 3:
\be 
\cY(\PPb_3{})=  \ZQ \la x^{\pm 1}_i, u_i^{\pm 1}, i=1,2,3 \ra/ (x_{i+1} x_i = q x_i x_{i+1}, u_i u_j = u_j u_i, u_i x_j = q^{\delta_{ij}} x_ju_i   ).
\ee

Define $\cY(\PPb_2{})$  by dropping $x_3$ and $ u_3$,  and define $\cY(\PPb_1{})$  by  dropping $x_2, x_3, u_2, u_3$:
\begin{align}
\cY(\PPb_2{})& =  \ZQ\la x_1^{\pm 1},x_2^{\pm 1},  u_1^{\pm1},  u_2^{\pm1}\ra /( u_i x_j = q^{\delta_{ij}} x_j u_i, x_2 x_1 =q x_1 x_2, u_1 u_2 = u_2 u_1)
\\
\cY(\PPb_1{}) &=   \ZQ\la x_1^{\pm 1}, u_1^{\pm1} \ra /( u_1 x_1= qx_1 u_1).
\end{align}

For $\bnu=(\bn, \bt)\in \BZ^j \times \BZ^j$ let $Y^\bnu=[x^\bn u^\bt]_\Weyl$. Then 
\be B(\cY(\PPb_j{})):=\{Y^\bnu \mid \bnu \in \BZ^{2j}\}
\label{eqBaseY}
\ee
is a $\ZQ$-basis of $\cY(\PPb_j{})$. 
For $i\le j$ let
$$ \cY(\PPb_j{})_{\deg_i =k} = \cR\text{-span of monomials having  degree of $x_i =k$}.$$

We will use the following function $\ddd_\tau$  to define a preorder (a linear order without the antisymmetric property) on $\uB(\PPb_j)$ and related filtrations.  Define
 $\ddd_{\PP_j}:\BN^j \times \BZ^j \to \BZ^{3}$  by
\begin{align}
\ddd_{\PP_3} ((n_1, n_2, n_3, t_1, t_2, t_3))&= (n_1+n_2+n_3, t_1+t_2+t_3, 0)\\
\ddd _{\PP_2}((n_1, n_2, t_1, t_2))&= (n_1+n_2, t_1+t_2, t_1)\\
\ddd_{\PP_1}((n_1, t_1)) &= (n_1, t_1, 0)
\end{align}
The reason  why $\ddd_\tau$ is defined as above will be clear in the proof of Theorem \ref{thmbtr}. Note that only $\ddd_{\PP_2}$ needs the third coordinate in its image. This is dictated by the complexity of the quantum trace for $\PP_2$  that will be constructed.
 
Using the lexicographic order on $\BZ^3$, for $\bk\in \BZ^3$ define
\be 
F_\bk (\cY(\PPb_j{})) = \cR\text{-span of } \{ Y^\bbl \mid \ddd_{\PP_j}(\bbl) \le \bk\}.
\ee

For two monomials $Y^\bnu,Y^{\bnu'}\in B(\cY(\PPb_j{}))$, we say $Y^\bnu$ has {\bf higher $\ddd_{\PP_j}$-degree} if $\ddd_{\PP_j}(\bnu) > \ddd_{\PP_j}(\bnu')$. We emphasize that this is not a linear order, but a preorder on $B(\cY(\PPb_j{}))$.

Let  $\cY(\PPb_3; \cR)= \cY(\PPb_j) \ot_\ZQ \cR$. Define $ \cY(\PPb_j; \cR)_{\deg_i =k}$ and $F_\bk (\cY(\PPb_j;\cR)) $ similarly.

\subsection{Quantum traces for $\uS(\PP_j)$ } \label{ssUtr}

We now formulate the main results of this section.

\def\YPj{{ \cY(\PPb_j )  }}
\def\YPjR{{ \cY(\PPb_j;\cR)  }}

\bthm \label{thmbtr}

For $j=1,2,3$ there is a reflection invariant $\ZQ$-algebra homomorphism
$$ \utr: \uS(\PPb_j{}) \to \cY(\PPb_j {})$$
having the following properties: For $i\le j$ and $\al \in \uB(\PPb_j)$ with $|\al \cap b_i| = k$,

\begin{enumerate}

\item Boundary grading:
\be \utr 
\left( 
\al 
\right) 
\ \in \  {\YPj}_{\deg_i=k} .  \label{eqGrade_utr}
\ee

\item Near boundary loop: The value of the loop $\ell_i$ is given by
\be \utr(\ell_i) = u_i^2 + u_i^{-2}. \label{eqLoop_utr}
\ee

\item Twist:  if $|\al \cap b_i|=k\neq 0$ then
\be 
\utr([\theta_i(\al)]_\omega)=  q^{-k} \, (u_i)^2\, \utr([\al]_\omega) = [  (u_i)^2\,\, \utr([\al]_\omega) ]_\Weyl.
\label{eqTwist_utr}
\ee

\item Highest order term: 
\be 
\utr(\al) \eqbu Y^{\nu(\al)} + F_{<\, \ddd_{\PP_j}(\nu(\al)) }( {\YPj}). \label{eqHighdeg_utr}
\ee
\end{enumerate}

\ethm
Part(1)--(3) will be used to glue the maps $\utr$ along bold boundary components. We will prove the Theorem in the subsequent subsections.

\brem %The map $\utrR:= \utr_\ot_ZQ \cR$ satisfies all the 
From the theorem it is easy to show that $\utr$ is injective.
\erem

\def\vcup{{ \overrightarrow{\cup}}}

\subsection{Reduction to 1-component case} 
%It is enough to prove Theorem \ref{thmbtr} when $\cR=\ZQ$, which is assumed now.
%Here we give the proof of parts which are common for all the DT pairs.

\blem \label{rReduct}
Assume there is a reflection invariant $ {\BZ_q}$-algebra homomorphism 
$ \utr\colon \uS(\PPb_j) \embed \cY(\PPb_j)$
satisfying all the conditions (1)-(4) of Theorem \ref{thmbtr} for all 1-component $\al \in \uB(\PPb_j)$. Then we also have (1)-(4) for all $\al\in \uB(\PPb_j)$, i.e. we have the Theorem \ref{thmbtr}.
\elem

\bpr Since \eqref{eqLoop_utr} concerns only 1-component elements, it holds true by assumption.

Let $\al\in \uB(\fS)$ have  connected components $\al_1, \dots, \al_k$ with $k \ge 2$. By \eqref{eqMulti}
\be \al \qeq \al_1 \dots \al_k \label{eq1comp}
\ee 

The multiplicative nature of $\utr$, the additive nature of the monomial degree and the $\ddd_{\PP_j}$-degree show that \eqref{eqGrade_utr}  and \eqref{eqHighdeg_utr}
%, which hold for each $\al_l$,
also hold for $\al$.

It remains to prove \eqref{eqTwist_utr}. 
Recall that $u_i$ commutes with all  variables except for $x_i$, for which $u_i x_i= q x_i  u_i$. Hence from \eqref{eqGrade_utr} we get $u_i ^2\  \utr(\al) = q^{2k} \utr(\al)\, u^{2k}_i$. By definition 
$$[u_i^2\,\utr([\al]_\omega)) ]_\Weyl = q^{-k} u_i^2\, tr([\al]_\omega)).$$
Since $\utr([\al]_\omega))$ is reflection invariant, $[u_i^2\,\utr([\al]_\omega)) ]_\Weyl$ is reflection invariant by Lemma \ref{rReflection}. 

Since $k=|\al \cap b_i| >0$ one of components of $\al$, say $\al_1$, intersects $b_i$.
 The components of $\theta_i(\al_1\al_2 \dots \al_k)$ are $\theta_i(\al_1), \al_2, \dots,\al_k$, as seen in Figure \ref{fig16}.

\begin{figure}[htpb!]
    %\centering
    \includegraphics[height=2.2cm]{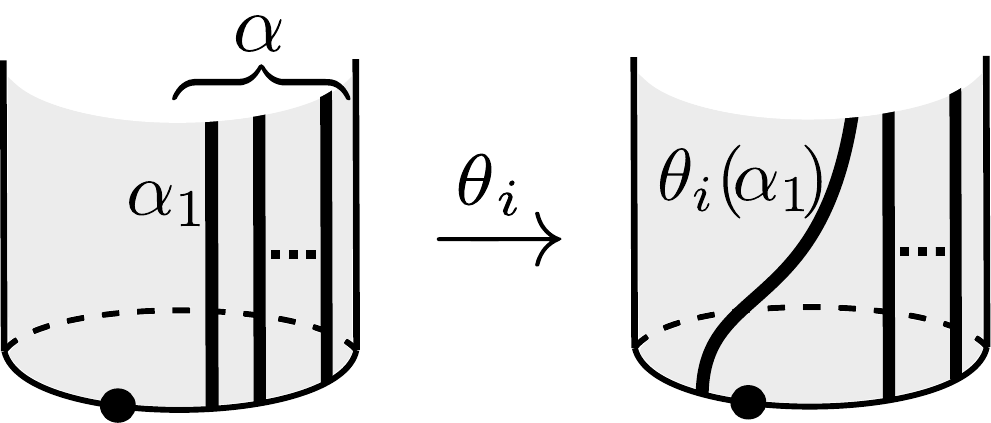}
    \caption{} \label{fig16}
\end{figure}

By  the height exchange rule \eqref{eqMulti}, we get 
$$ \theta_i(\al_1\al_2 \dots \al_k)\qeq \theta_i(\al_1) \al_2 \dots \al_k.$$
Using \eqref{eq1comp}, the above identity,  then Identity \eqref{eqTwist_utr} for $\al_1$, we get
$$ \utr([\theta_i(\al)]_\omega) \qeq \utr(\theta_i(\al_1)) \utr( \al_2 \dots \al_k) \qeq [ u_i^2\, \utr([\al]_\omega))]_\Weyl.$$
From the reflection invariance, we get 
$ \utr([\theta_i(\al)]_\omega) =[ u_i^2\, \utr([\al]_\omega))]_\Weyl,$
which proves \eqref{eqTwist_utr}.
\epr

\subsection{Near boundary arc and slide over bold vertex} Consider $\PP_j$. Fix $i\le j$.
\blem \label{rIn} Let   $\al$ be boundary ordered $\partial\PPb_j$-tangle diagrams with $|\al \cap b_i|= k\neq 0$.

(a) One has

\begin{align}
\ell_i \, \al &= q^k \theta_i(\al) + q^{-k} \theta_i^{-1}(\al)
\label{eqReso2}
%\\ \theta_i(\al \al') &\qeq \theta_i(\al)\  \al'  . \label{eqtheta}
\end{align}

(b) If a $ {\BZ_q}$-algebra homomorphism $f: \uS(\PPb_j) \to \cY(\PPb_j)$ satisfies $f(\ell_i) = u_i^2 + u_i^{-2}$, 
then we have the following recursive relation: For all $m\in \BZ$,
\be 
q^k f(\theta_i^{m+2} (\al)) -  ( u_i^2 + u_i^{-2} )  f(\theta_i^{m+1} (\al))+  q^{-k} f(\theta_i^{m} (\al)) = 0. 
\label{eqRecur}
\ee

\elem
\bpr (a) Use the skein relation (A) to resolve the crossings of $\ell_i \vcup \al$, where $\ell_i \vcup \al$ denotes a union of $\ell_i$ and $\al$ with all over/under information assigned to be $\ell_i$ is higher than $\al$. See Figure \ref{fig:loopPr}. To have no near boundary arcs there are only  two ways to resolve all the $k$ crossings, and they give \eqref{eqReso2}. 

\begin{figure}[htpb!]
    %\centering
    \includegraphics[height=1.5cm]{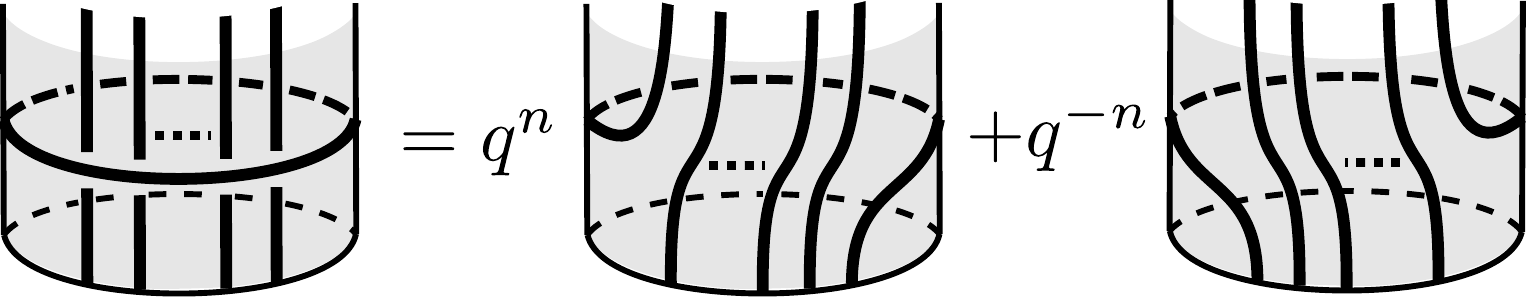}
    \caption{The only two surviving terms of $\ell_i \vcup \al$}
    \label{fig:loopPr}
\end{figure}

(b) Using $ u_i^2 + u_i^{-2}= f(\ell_i)$ then 
Identity \eqref{eqReso2}, we get
 $$  ( u_i^2 + u_i^{-2} )  f(\theta_i^{m+1} (\al)) = f(\ell_i\,\theta_i^{m+1}(\al)) = q^k f(\theta_i^{m+2} (\al)) +  q^{-k}f(\theta_i^{m} (\al)),$$
which  proves\eqref{eqRecur}.
\epr

\subsection{Proof of Theorem \ref{thmbtr} for $\PP_3$}
\begin{proof} By Theorem \ref{thm.emb2}, the  triangulation $\Delta$ of $\PPb_3$, with interior edges set $\Do=\{ e_1, e_2, e_3, e'_1, e'_2, e'_3\}$  given in Figure \ref{fig:trian_P3}, gives rise to a $ {\BZ_q}$-algebra map
$\utr^A_\Delta\colon \uS(\PPb_3) \to \uA(\PPb_3, \Delta),$
where 
$$ \uA(\PPb_3, \Delta) =  {\BZ_q}\la \aaa_e^{\pm1}, e \in \Do  \ra / \aaa_e \aaa_c = q^{\uP(a,c)}  \aaa_c \aaa_e $$
Here $\uP: \Do \times \Do$ is the unique antisymmetric function defined by 
$$\uP(e'_i, e_i)=2,\ \uP ( e'_i,  e'_{i+1} ) = \uP ( e'_{i}, e_{i+1} ) = \uP (  e'_{i+1}, e_{i})=1,  \ \text{indices are taken mod } \ 3. $$
\begin{figure}[htpb!]
    %\centering
    \includegraphics[height=3.5cm]{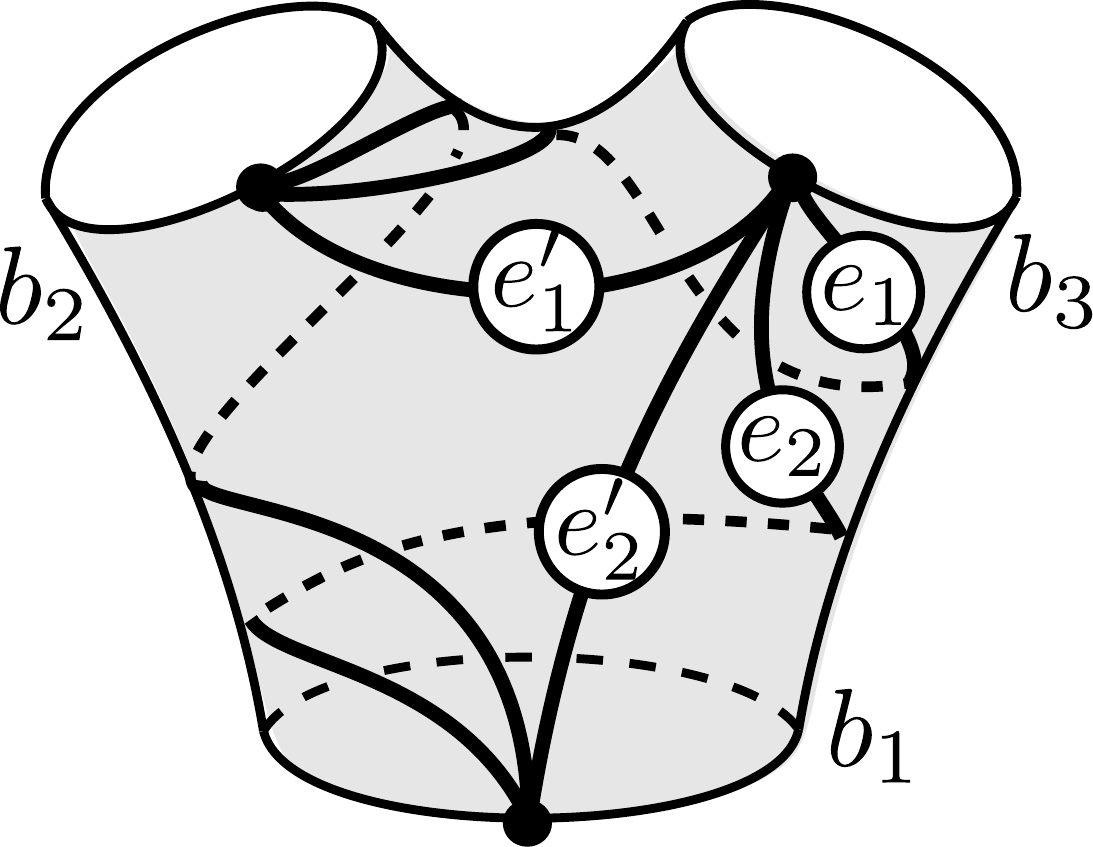}
    \caption{}
    \label{fig:trian_P3}
\end{figure}
There is a $ {\BZ_q}$-algebra homomorphism 
$g: \uA(\PPb_3, \Delta) \to \cY(\PPb_3)  $ given by 
$$ g(\aaa_{e_i})= [x_{i+1} x_{i+2}]_\Weyl, \ g(\aaa_{e'_i})= [u_{i+2}^2x_{i+1} x_{i+2}]_\Weyl.$$ 
We claim that  $\utr:= g \circ \utr^A_\Delta\colon \uS(\PPb_3) \to \cY(\PPb_3)$ satisfies all the requirements of the theorem. Clearly $\utr$ is reflection invariant.

\blem \label{rIndent1} 
For $i\neq k$ in $\{1,2,3\}$ and $m\in \BZ$  we have
%a) One has Identity \eqref{eqLoop_utr}. That is $\utr(\ell_i) = u_i^2 + u_i^{-2}$.(b) 
\begin{align} 
\utr(\ell_i) &= u_i^2 + u_i^{-2} \label{eqaell33}\\
\utr(\theta_i^m( a_{ik}))&= [u_i^{2m}x_i x_k]_\Weyl \label{eqa23} \\
 \utr( [\theta_i^m(a_{ii})]_\omega)) &=  [u_i^{2m} u_{i+1}^2 x_i^2]_\Weyl+[u_{i}^{2m+2} u_{i+2}^{-2} x_i^2]_\Weyl. \label{eqgamma1}
\end{align}
\elem

\bpr  Let us prove \eqref{eqaell33}.
 By \eqref{eqellmu} we have 
$$ \utr^A_\Delta (\ell_i) = [\aaa_{e'_2} \aaa_{e_2}^{-1}]_\Weyl + [\aaa_{e'_2}^{-1} \aaa_{e_2}]_\Weyl. $$
Applying $g$ to both sides, we get \eqref{eqaell33}.

Let us prove \eqref{eqa23}. 
Calculate $\utr^A_\Delta(a_{ik})$ and $\utr^A_\Delta(\theta_i(a_{ik}))$ using the method described in Remark \ref{rCal}, then apply $g$, we get
$$\utr(a_{ik}))= [x_i x_k] ,\quad   \utr(\theta_i(a_{ik})) = [u_i^2x_i x_k].$$

\no{Give pictures, more details for $a_{11}$, then say that the others are similarI think we could just say following a similar calculation to that for Equation \ref{eqellmu}}

Thus \eqref{eqa23} is true for $m=0$ and $m=1$.
Using the recursion relation \eqref{eqRecur} we can easily see that  if \eqref{eqa23} is true for $m$ and $m+1$, then it is true for $m+2$ and $m-1$. Hence \eqref{eqa23} is true for all $m\in \BZ$.

\no{Give  more details here.}

The same proof works for \eqref{eqgamma1}.
\epr

Let now prove the Theorem. Suppose  $\al\in \uB(\PPb_3)$ has one component. Then $\al$ is one of the curves on the left hand sides of the identities in Lemma \ref{rIndent1}. Identities of Lemma \ref{rIndent1} show that all conditions (1)-(3) of Theorem \ref{thmbtr} are satisfied for $\al$. We will show now that (4) is also satisfied. There are two observations:

(i) In each right hand side of \eqref{eqaell33}-\eqref{eqgamma1}, the first monomial is $\ddd_{\PP_3}$-dominant in the sense that it 
has $\ddd_{\PP_3}$ higher than any other terms. For example, the first and second monomials  of the right hand side of
\eqref{eqgamma1} are respectively $Y^\bnu$ and $Y^{\bnu'}$, where
$$ \bnu = ( 2,0,0,2m, 2,0 ) , \ \bnu' = ( 2,0,0,2m+2, -2,0 ). $$
By definition  $\ddd_{\PP_3}(\bnu) > \ddd_{\PP_3}(\bnu')$. The definition of $\ddd_{\PP_3}$ was designed so that this is true.

(ii) The exponent of that $\ddd_{\PP_3}$-dominant monomial is exactly the $\nu(\al)$. The modified DT coordinate was designed for this purpose.

 Thus we also have condition (4) for $\al$. By Lemma \ref{rReduct} we have the theorem.
\end{proof}

\subsection{Proof of Theorem \ref{thmbtr} for $\PP_2$} 
\begin{proof} By Theorem \ref{thm.emb2}, the  triangulation $\Delta$ of $\PPb_2$, with interior edges set $\Do=\{ a,b,c,e\}$  given in Figure \ref{fig:trian_P2}, gives rise to
  a $ {\BZ_q}$-algebra map
$$\utr^A_\Delta: \uS(\PPb_2) \to \uA(\PPb_2, \Delta)=\bT(\uP),$$
where $\uP: \{ a,b,c,d\} \times \{ a,b,c,d\}$ is the antisymmetric function given by
$$ \uP(d,a) = \uP(b,a), \ \uP(c,d) = \uP(a,c) = \uP(b,c)  =1,\ \uP(d,b) =0.  $$
\begin{figure}[htpb!]
    %\centering
    \includegraphics[height=2.7cm]{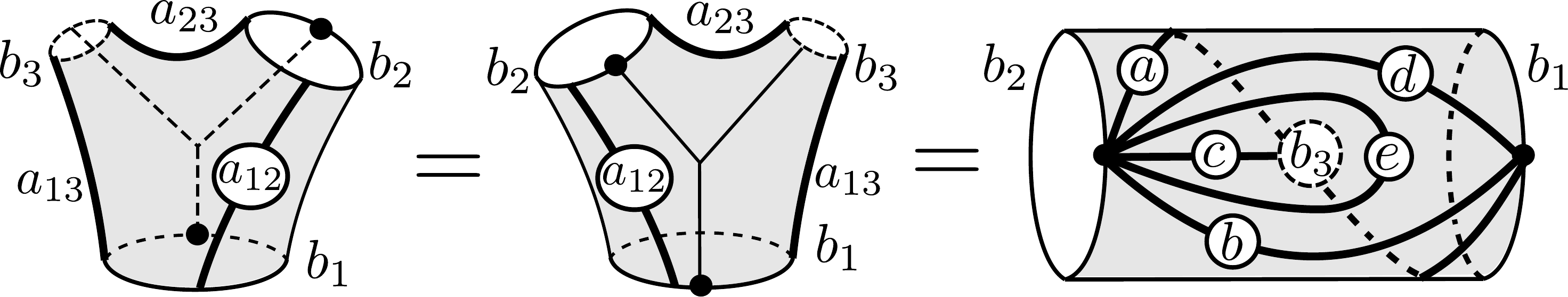}
    \caption{}
    \label{fig:trian_P2}
\end{figure}
There is a $ {\BZ_q}$-algebra homomorphism 
$g\colon \uA(\PPb_2, \Delta) \to \cY(\PPb_2)  $ given by 
$$ g(\aaa_{a})= [x_1 x_2]_\Weyl, \ g(\aaa_b)= [u_{2}^2 x_{1}x_{2}]_\Weyl, \ g(\aaa_c) =x_2, \ g(\aaa_d)= [u_1^2 x_1 x_2 ]_\Weyl.$$ 
Let  $\utr:= g \circ \utr^A_\Delta\colon \uS(\PPb_2) \to \cY(\PPb_2)$. Clearly $\utr$ is reflection invariant.

\blem \label{rIndent2}   For $i=1,2$ and $m\in \BZ$  we have

\begin{align} 
\utr(\ell_i) & = u_i^2 + u_i^{-2} \label{eqell12} \\
\utr(\theta_i^m( a_{12}))&= [u_i^{2m}x_1 x_2]_\Weyl  \label{eqa12}
 \\ \utr(\theta_1^m( a_{13}))&= [u_1^{2m}x_1]_\Weyl + [u_1^{2m-2} u_2^2x_1]_\Weyl  \label{eqa13} \\
  \utr(\theta_2^m( a_{23}))&= [u_2^{2m}x_2]_\Weyl   \label{eqa23bis}\\
\utr( [\theta_1^m(a_{11})]_\omega) &=  [u_1^{2m+2} x_1^2]_\Weyl+[u_{1}^{2m-2} u_{2}^{4} x_1^2]_\Weyl + (q+ q^{-1})[u_{1}^{2m} u_{2}^{2} x_1^2]_\Weyl. \label{eqa11}\\
\utr( [\theta_1^m(a_{22})]_\omega) &=  [u_2^{2m} x_2^2]_\Weyl. \label{eqa22}
\end{align}
\elem

\bpr 
Let us prove  \eqref{eqell12}, which is \eqref{eqLoop_utr}. By \eqref{eqellmu} we have 
$$ \utr^A_\Delta (\ell_1) = [\aaa_{a} \aaa_{d}^{-1}]_\Weyl + [\aaa_{a}^{-1} \aaa_{d}]_\Weyl, \ \utr^A_\Delta (\ell_2) = [\aaa_{a} \aaa_{b}^{-1}]_\Weyl + [\aaa_{a}^{-1} \aaa_{b}]_\Weyl. $$
Applying $g$ to both sides, we get \eqref{eqell12}.

Proof of \eqref{eqa12}-\eqref{eqa23bis}. 
Calculate $\utr^A_\Delta(a_{ik})$ and $\utr^A_\Delta(\theta_i(a_{ik}))$, for applicable indices, using the method described in Remark \ref{rCal}, then apply $g$, we get 
 \eqref{eqa12}-\eqref{eqa23bis}  for $m=0$ and $m=1$.
Using the recursion relation \eqref{eqRecur} we get \eqref{eqa12}-\eqref{eqa23bis}  for all $m\in \BZ$.

The same proof works for \eqref{eqa11} and \eqref{eqa22}. Alternatively, we can use the identity $[a_{11}]_\omega =  [\theta_1(a_{13}^2)]_\omega$ and $[a_{22}]_\omega= a_{23}^2$, and Identities \eqref{eqa12}-\eqref{eqa13} to get \eqref{eqa11} and \eqref{eqa22}.
\epr

Let now prove the Theorem. Suppose  $\al\in \uB(\PPb_2)$ has one component. Then $\al$ is one of  the curves on the left hand sides of \eqref{eqell12}-\eqref{eqa22}. On each right hand side of  \eqref{eqell12}-\eqref{eqa22}, the first monomial is $\PP_2$-dominant and has exponent equal to the DT coordinates of the curve of the left hand side. 
 The definitions of $\ddd_{\PP_2}$ and DT coordinates were designed for this property. 

The explicit values of $\utr(\al)$ given in 
Lemma \ref{rIndent2} show that all conditions (1)-(4) of Theorem~\ref{thmbtr} are satisfied for $\al$. Hence by Lemma \ref{rReduct} we have the theorem.
\end{proof}

\subsection{Proof of Theorem \ref{thmbtr} for $\PP_1$}

\begin{proof} By Theorem \ref{thm.emb2}, the triangulation $\Delta$ of $\PPb_1$, with interior edges set $\Do=\{ a,b\}$  given in Figure \ref{fig:trian_P1}, gives rise to
  a $ {\BZ_q}$-algebra map
$$\utr^A_\Delta\colon \uS(\PPb_1) \to \uA(\PPb_1, \Delta)= {\BZ_q}\la \aaa_c^{\pm1}, \aaa_d^{\pm1}\ra/( \aaa_d \aaa_c= q \aaa_c \aaa_d).$$
\begin{figure}[htpb!]
    %\centering
    \includegraphics[height=2.6cm]{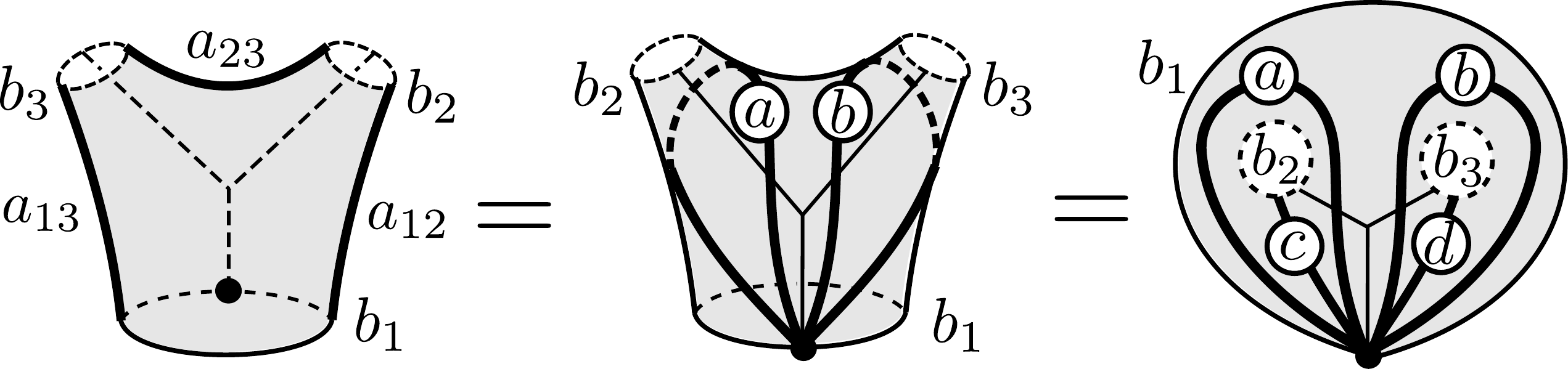}
    \caption{}
    \label{fig:trian_P1}
\end{figure}

There is a $ {\BZ_q}$-algebra homomorphism 
$g\colon \uA(\PPb_1, \Delta) \to \cY(\PPb_1)  $ given by 
$$ g(\aaa_{c})= x_1, \ g(\aaa_d)= [u_{1} x_{1}]_\Weyl.$$ 
Let  $\utr:= g \circ \utr^A_\Delta:\colon\uS(\PPb_1) \to \cY(\PPb_1)$. Clearly $\utr$ is reflection invariant.

\blem \label{rIndent11}
 For  $m\in \BZ$  we have
\begin{align} 
\utr(\ell_1) & = u_1^2 + u_1^{-2}  \label{eqell1111}\\
\utr(a_{23})&= u_1+ u_1^{-1}  \label{eqa2323}
 \\ \utr(\theta_1^m( a_{12}))&= [u_1^{2m}x_1]_\Weyl, \
  \utr(\theta_1^m( a_{13}))= [u_1^{2m-1}x_1]_\Weyl   \label{eqq1213}\\
\utr( [\theta_1^m(a_{11})]_\omega) &=  [u_1^{2m} x_1^2]_\Weyl. \label{eqa1111}
\end{align}
\elem

\bpr Let us prove \eqref{eqell1111}.
 By \eqref{eqellmu} we have 
$$ \utr^A_\Delta (\ell_1) = [\uv{a} (\uv{b})^{-1}]_\Weyl +[\uv{b} (\uv{a})^{-1}]_\Weyl. $$
Using $\uv{a}= \aaa_c^2, \uv(b) = \aaa_d^2$ and the map
 $g$, we get \eqref{eqell1111}.

Calculate $\utr^A_\Delta(a_{ik})$ and $\utr^A_\Delta(\theta_1(a_{ik}))$, for applicable indices, using the method described in Remark \ref{rCal}, then apply $g$, we get 
 \eqref{eqa2323}  and \eqref{eqq1213}  for $m=0$ and $m=1$.
Using the recursion relation \eqref{eqRecur} we get \eqref{eqa2323}  and \eqref{eqq1213} 
   for all $m\in \BZ$.

The same proof works for \eqref{eqa1111}. Alternatively, we can use the identity $[a_{11}]_\omega =  [\theta_1(\aaa_{a_{12}}^2)]_\omega$ and \eqref{eqq1213} to get \eqref{eqa1111}.
\epr

Let now prove the Theorem. Suppose  $\al\in \uB(\PPb_2)$ has one component. Then $\al$ is  one of the curves on the left hand sides of \eqref{eqell1111}-\eqref{eqa1111}. The explicit values of $\utr(\al)$ given in 
Lemma \ref{rIndent11} and the definitions of $\ddd$ and DT coordinates show that all conditions (1)-(4) of Theorem \ref{thmbtr} are satisfied for $\al$. Again the definition of $\ddd$ and the twist DT coordinates are defined for this purpose. By Lemma \ref{rReduct} we have the theorem.
\end{proof}

\def\xxx{\orange{xxx}}
\section{Degenerations of $\cS(\Sigma_{g,m} {;\cR})$ and quantum tori}
\label{secDegen}

We will prove a refinement of Theorem \ref{thmSRY} concerning the skein algebra $\cS(\Sigma_{g,m} {;\cR})$. In particular we show that $\cS(\Sigma_{g,m} {;\cR})$ is a domain,  has a degeneration equal to a monomial algebra, and calculate its Gelfand-Kirillov dimension.

Throughout this section we fix $\fS= \Sigma_{g, m}$, with  $(g,m) \neq (0,k), (1,0)$, for $k \le 4$.  These exceptional cases will be considered in the Appendix. 
%For simplicity we will choose a special DT datum $(\cC, \Gamma)$. 
% Denote \be  r = 3g-3 +m.  \label{eqrrr} \ee

\subsection{Quantum torus associated with a DT datum $(\cC,\Gamma)$} We now define a main object appearing in the result. Let $\mQ\colon \cC \times \cC\to \BZ$ be
 the antisymmetric matrix given by 

\begin{align*}
\mQ(a,c)=
\# \left(\begin{array}{c}\includegraphics[scale=0.17]{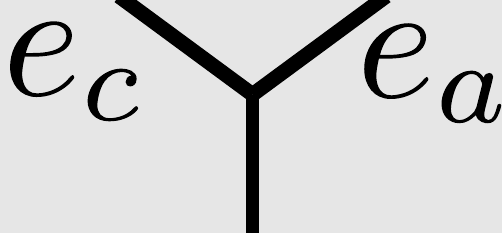}\end{array}\right)-\#\left(\begin{array}{c}\includegraphics[scale=0.17]{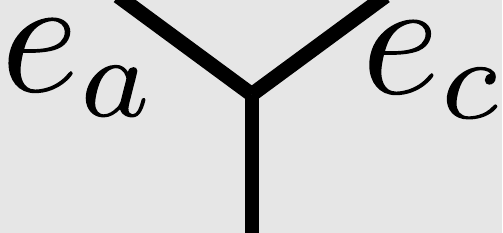}\end{array}\right).
\end{align*}

Here $e_a$ is the edge of $\Gamma$ dual to $a\in \cC$, and the right hand side is the signed number of times a half-edge of $e_a$ meets a half-edge of $e_c$ at a vertex of $\Gamma$, where the sign is $+1$ if $e_a$ is clockwise to $e_c$, and $-1$ otherwise. Let $\tmQ$ be the symplectic double of $\mQ$:
($\Id_r$ and $0_r$ are respectively the $r\times r$ identity matrix and the $r\times r$ 0 matrix):
$$ \tmQ =  \begin{bmatrix} \mQ &  -\Id_r \\ \Id_r & 0_r
\end{bmatrix}$$
Define the
 quantum torus $\cY(\fS,\cC,\Gamma{})= \bT(\tmQ{})$ and $\cY(\fS,\cC,\Gamma; \cR) = \cY(\fS,\cC,\Gamma{}) \ot_\ZQ \cR$. Thus,
$$ \cY(\fS,\cC,\Gamma{})= \ZQ\la y_c^{\pm 1}, u_{c}^{\pm 1}, c \in \cC \ra / ( y_a y_c = q^{\mQ(a,c)} y_c y_a, u_a u_c= u_c u_a, u_a y_c= q^{\delta_{ac}} y_c u_a  ). $$
For $\bnu=(\bn, \bt)\in \BZ^r \times \BZ^r$ let $Y^\bnu= [y^\bn u^\bt]_\Weyl$. The following is a $\ZQ$-basis of $\cY(\fS,\cC,\Gamma{})$:
\be B(\cY(\fS,\cC,\Gamma))= \{ Y^\bnu \mid \bnu \in \BZ^{2r} \}. 
\label{eqBasisY}
\ee

\def\tiF{{\tau \in \cF}}

\subsection{Order of $\cC$ and filtrations} \label{ssN1}  To simplify some technical steps, we will choose $\cC$, $\Gamma$ and an order on $\cC$, and use it to define filtrations on ${\SSR}$.

The curves in $\cC$ cut $\fS$ into $\SC= \bigsqcup_{\tiF} \tau$, where each  face $\tau$  is one of the three DT pairs of pants $\PP_1, \PP_2$, and $\PP_3$.  
There is a projection $\pr: \SC \onto \fS$ identifying pairs of bold boundary components in $\cCt$. 
For $c\in \cC$ we denote $c', c''\in \cCt$ the lifts of $c$.

Choose  $\cC=(c_1, c_2, \dots, c_r)$ such that the first $m+1$ elements (for $g\ge 1$) or $m-3$ (for  $g=0$) are as in the following 

\begin{align*}
\begin{array}{c}\includegraphics[scale=0.3]{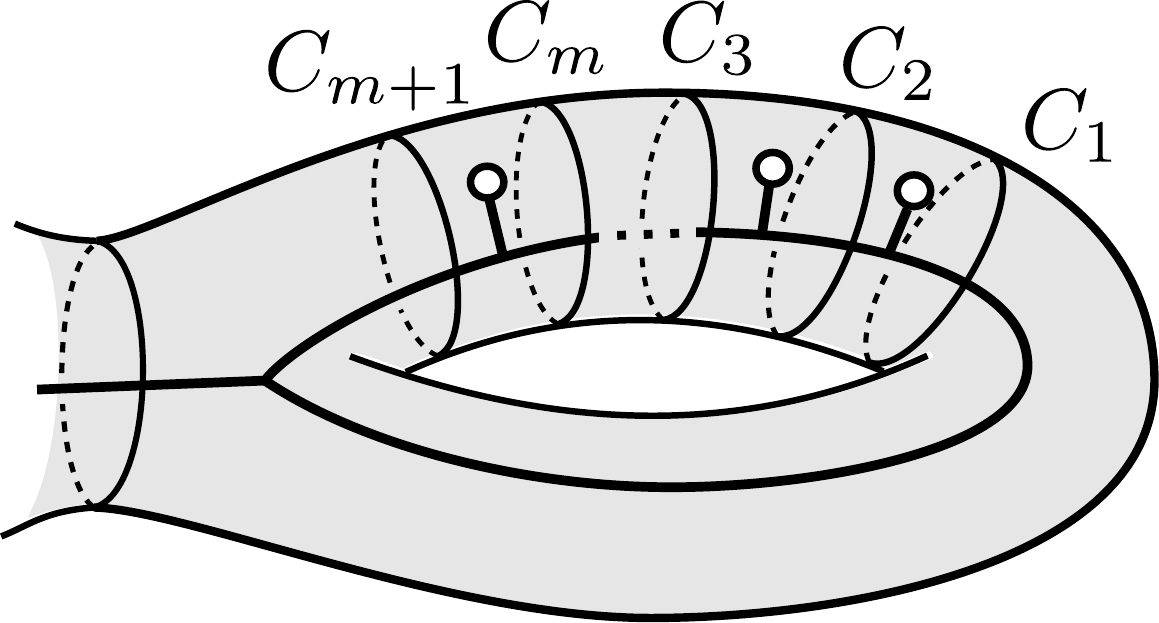}\end{array}\qquad
\begin{array}{c}\includegraphics[scale=0.3]{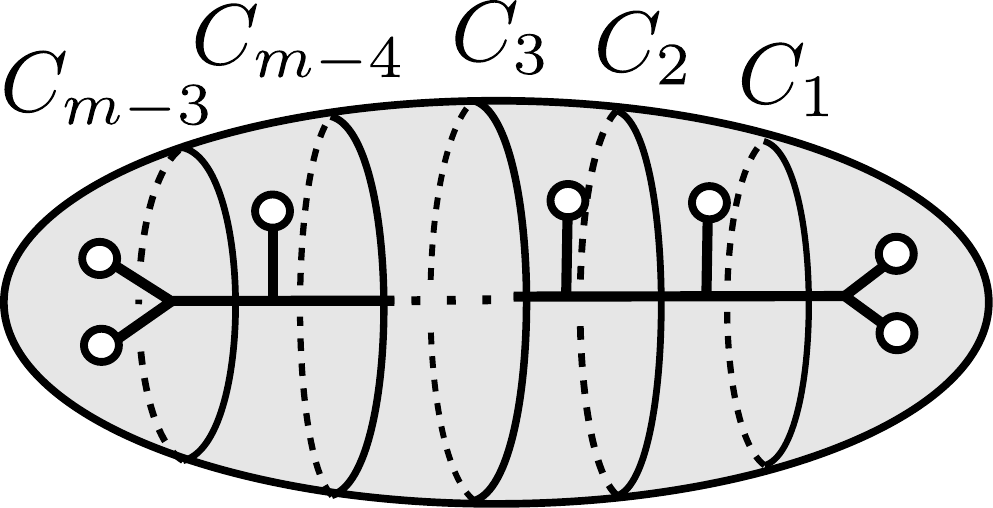}\end{array}.
\end{align*}
This choice of order satisfies the following condition

{\bf (N1)} In a $\PP_2$ face, the curve $\pr(f(b_1))$ comes before $\pr(f(b_2))$ if they are different. Here $f$ is the characteristic map of the face.

Condition {\bf (N1)} is used to transfer certain inequalities from faces to inequalities for $\fS$. 

There are $\barm$ faces of type $\PP_2$, where $\barm =m$ if $g\ge1$ and $\barm =m-4$ if $g=0$.

Let  $\mu: \Lambda_{\cC, \Gamma} \to B(\fS)$ be the inverse of the coordinate map $\nu: B(\fS) \xra{\cong} \Lambda_{\cC,\Gamma}\subset \BN^{\cC}\times \BZ^{\cC}$.  To compare elements of $\Lambda_{\cC,\Gamma}$ or $B(\fS)$, we define $\ddd_1$ and $\ddd$ as follows
 \begin{align}
\ddd_1&: \BN^\cC \times \BZ^{ \cC} \to \BN, \quad \ddd_1(\bn, \bt) = \sum_{c\in \cC} \bn(c)\\
\ddd&: \BN^\cC \times \BZ^{ \cC} \to \BZ^{\barm+2}, \quad \ddd(\bn, \bt) = ( \sum_{c\in \cC} \bn(c), \sum_{c\in \cC}  \bt(c), \bt(c_1), \dots, \bt(c_\barm)).
\end{align}
For $\al, \beta\in B(\fS)$ define $\ddd_1(\al)=\ddd_1(\nu(\al)),\ddd(\al) = \ddd(\nu(\al))$, and $\la \al, \beta \ra_\tmQ= \la \nu(\al), \nu(\beta )\ra_\tmQ$.

Using the lexicographic order on $\BN \times \BZ^{\barm+1}$, for $k\in \BN$ and $\bk\in \BN \times \BZ^{\barm+1}$ define 
\begin{align}
F_k( {\SSR}) &:= \cR\text{-span of } \ \{ \al \mid \ddd_1(\al) \le k \} \label{eqFil1}\\
E_\bk( {\SSR}) &:= \cR\text{-span of } \ \{ \al \mid \ddd(\al)  \le \bk\} \label{eqFil2}\\
E_\bk(\cY(\fS,\cC,\Gamma {;\cR})) &:= \cR\text{-span of } \ \{ Y^\bl \mid \ddd(\bl )  \le \bk\}
\label{eqFil3}
\end{align}
The first \eqref{eqFil1} defines an $\BN$-filtration on $ {\SSR}$  compatible with the product.
We denote the associated graded algebra by $\Gr^F( {\SSR})$. Using the lead term map, see Proposition~\ref{r.basisGr}, we will consider $B(\fS)$ also as an $\cR$-basis of $\Gr^F( {\SSR})$. 

The second \eqref{eqFil2} defines a $(\BN \times \BZ^{\barm+1})$-filtration on $ {\SSR}$, which will be shown to respect the product later.

The third \eqref{eqFil3} defines an $(\BN \times \BZ^{\barm+1})$-filtration on $\cY(\fS,\cC,\Gamma {;\cR})$ compatible with the product, and its associated graded algebra is canonically isomorphic to $\cY(\fS,\cC,\Gamma {;\cR})$. This is because $\cY(\fS,\cC,\Gamma {;\cR})$ is a graded algebra, graded by the exponents of the monomials.

When $\cR=\ZQ$ we drop $\cR$ from the notation. Using the bases of $\SS$ and $\cY(\fS,\cC,\Gamma)$, we get
\begin{align}
F_k( {\SSR})&= F_k(\SS) \ot_\ZQ \cR    \label{eqR1} \\
 E_\bk( {\SSR}) &= E_\bk( {\SS})  \ot_\ZQ \cR  \label{eqR2}   \\
  E_\bk(\cY(\fS,\cC,\Gamma {;\cR})) &= E_\bk(\cY(\fS,\cC,\Gamma)) \ot_\ZQ \cR. \label{eqR3}
\end{align}

%%XXXXXXXXXXXX 

\subsection{Main results} We now formulate the main result of this section.

\bthm \label{thmMain2} Let $\fS= \Sigma_{g, m}$ with a DT datum $(\cC, \Gamma)$, where $\cC=(c_1,\dots, c_r)$ satisfying  condition {\bf (N1)} of subsection \ref{ssN1}.
 Here $(g,m) \neq (0,k), (1,0)$, for $k \le 4$.  The ground ring $\cR$ is a commutative $\ZQ$-domain.

(a)  There is an $\cR$-algebra embedding 
\be 
\phi\colon \Gr^F( {\SSR}) \embed \cY(\fS,\cC,\Gamma {;\cR})
\ee
with the following property: For $\al \in B(\fS)$ we have
\be
 \phi(\al) = Y^{\nu(\al)} + E_{< \ddd(\nu(\al) )  }(\cY(\fS,\cC,\Gamma {;\cR})). \label{eqtopdeg}
\ee
Consequently $ {\SSR}$ is a domain.

(b) For $\bk, \bl  \in \Lambda_{\cC,\Gamma}$ we have the following identity in $ {\SSR}$:
\be 
\mu(\bk) \mu(\bl) =  q^{\frac 12 \la \bk, \bl\ra _{\tmQ}  } \mu (\bk + \bl) 
\mod E_{ < \ddd(\bk + \bl) }( {\SSR}),  \label{eqProdTop}
\ee
where $\mu$ was defined in Section \ref{ssN1}.

(c) The $(\BN\times \BZ^{\barm+1})$-filtration $(E_\bk( {\SSR}))$ is compatible with the product, and its associated graded algebra is isomorphic to the monomial algebra $\bT(\tmQ,\Lambda_{\cC,\Gamma } {;\cR})$:
\be 
\Gr^E( {\SSR}) \xar{ \cong} \bT(\tmQ,\Lambda_{\cC,\Gamma } {;\cR}).
\ee

(d) If the ground ring $\cR$ is Noetherian, then so is $ {\SSR}$.

(e) The Gelfand-Kirillov dimension of $ {\SSR}$ is $2r$, where $r=3g-3+m$.
\ethm

\def\DDD{\mathsf{D}}

\def\uTR{{\underline{\mathsf{TR}}}}
\def\UniF{\bigsqcup _{\tau\in \cF}}
\def\tY{{\tilde \cY}}
\def\cSC{{ \check\fS_\cC}}
\def\EEE{{ \cY^\diamond}}

Theorem \ref{thmMain2} is proved in Subsections \ref{ssFaces}-\ref{ssProofL}. Let us now prove Theorem \ref{thmSRY}. 

\bpr
[Proof of Theorem \ref{thmSRY}] We assume $(g,m) \neq (0,4)$, which will be handled in Appendix. We use the notation of Subsection \ref{ssRY}.

(b) By Theorem \ref{thmMain2}(a),  $\SSR$ is a domain. By Proposition \ref{rRYRY}(b), $\SRYSR$ is a domain.

By Theorem \ref{thm-noether1}, $\SSR$
is orderly finite generated. Hence $\SRYSR$ is orderly finitely generated by Proposition \ref{rRYRY}(d), 

Assume $\cR$ is Noetherian. Then $\SSR$ is Noetherian by Theorem \ref{thmMain2}(d). By Proposition \ref{rRYRY}(c), $\SRYSR$ is Noetherian.

(d) follows from Theorem \ref{thmMain2}(e)  and Proposition \ref{rRYRY}(e).

(a) \& (c):  Using the filtration $F$  (or $E$) for the ground ring $\cR_\fS'$, then restricting it to the 0-component of the grading \eqref{eqGrade2}, we get the desired filtration. \epr

\subsection{Combining the faces} \label{ssFaces}
We analyze the decomposition $\SC = \UniF \tau$.

To simplify the notation, we will drop $\cC, \Gamma$ and 
 write $B, \Lambda, \cY(\fS {;\cR})$ for respectively $B(\fS),\Lambda_{\cC, \Gamma},\cY(\fS,\cC,\Gamma {;\cR})$.
For $a,b\in B$  their product in $\Gr^F( {\SSR})$  is  denoted by $a*b$.

Let $\cSC= \UniF \Pb$. Then $\uS(\cSC {;\cR}) = \bigotimes_{\tau \in \cF} \uS(\Pb {;\cR})$. Taking the tensor product, we have an $\cR$-algebra homomorphism
$$ \uTR= \bigotimes_{\tau \in \cF} \utr_\tau \, :  \uS(\cSC {;\cR}) \to \tY:=  \bigotimes_{\tau \in \cF} \cY(\Pb {;\cR}).$$
Recall that $\cY(\Pb;\cR)$ is the $\cR$-algebra of Laurent polynomials in $x_a, u_a$, where $a$ runs the set of bold boundary components of $\tau$, and they $q$-commute as described in Section \ref{ssQtPj}.
We identify  $\tY$ with
the $\cR$-algebra of polynomials in  variables $x^{\pm1}_a, u_a^{\pm1}, a \in \cCt$, where the variables corresponding to boundary components of a face $\tau$ are $q$-commuting by the rules of $\cY(\tau {;\cR})$, otherwise they are commuting. An $\cR$-basis of $\tY$ is $ \{Y^\bn \mid \bn \in \BZ^  \cCt\times \BZ^  \cCt\}$.

Let $\EEE\subset \tY $ be the $\cR$-span of $Y^\bn$ in which the degrees of $x_{c'}$ and $x_{c''}$ are equal for all $c\in \cC$. Equivalently, $\EEE$ is  generated by all $u^{\pm 1}_a$ with $ a\in \cCt$, and all $(x_{c'} x_{c''})^{\pm1}$ with $c\in \cC$. The following explains where the matrix $\mQ$ comes from.

\blem \label{rIsoY}
There is an $\cR$-algebra isomorphism
\be 
\cY(\fS {;\cR}) \xar{\cong} \EEE/(u_{c'} = u_{c''}, c\in \cC), \ \text{with } x_c \to [x_{c'} x_{c''}]_\Weyl, \ u_c\to u_{c'} = u_{c''}.
\ee
\elem
\bpr This follows right away from the definition of $\cY(\fS {;\cR})$. 
\epr
We will identify $\cY(\fS {;\cR})$ with  $\EEE/(u_{c'} = u_{c''}, c\in \cC)$ via the isomorphism of Lemma \ref{rIsoY}.

An $\cR$-basis of $\uS(\cSC {;\cR}) $ is $\uB(\cSC)=\prod_{\tau \in \cF} \uB(\Pb)$. Recall that an element $\al \in \uB(\cSC)$ is matched if $|\al \cap c'| =|\al \cap c''|$ for all $c\in \cC$. 
Let $(\uS(\SC {;\cR}))^\diamond $ be the $\cR$-submodule of $\uS(\cSC {;\cR})$ spanned by all boundary ordered {\em matched} elements. Theorem \ref{thmbtr}(1) implies that $\uTR((\uS(\cSC {;\cR}))^\diamond)\subset \EEE$.
Consider the composition
\be 
\Phi\colon (\uS(\cSC {;\cR}))^\diamond \xra{\uTR} \EEE \onto \cY(\fS {;\cR}).
\ee

\subsection{Definition of $\phi$} In this subsection we define $\phi\colon \Gr^F( {\SSR}) \to \cY(\fS {;\cR})$. First we will define  $\phi(\al)$ for  $\al$ in the basis $B= B(\fS)$. Then we extend $\phi$ linearly, and show that it is an algebra homomorphism.

\def\uTR{{\underline{\mathsf{TR}}}}
\def\Lift{{\mathsf{Lift}}}
\def\cSC{{\check\fS_\cC}}
\def\pcSC{\partial \cSC}
\def\pSC{\partial \cSC}
\def\tiF{{\tau \in \cF}}
\def\ttD{{\tilde D}}

Let $\al\in B$. Choose a diagram $D$ in good position representing $\al$. Let $h$ be a choice of a linear order on each set $D \cap c, c\in \cC$. We can lift $(D,h)$ to an element $\Lift(D, h)$, a boundary ordered $\pSC$-tangle diagram, consisting of the lift of $D$ and the boundary order lifting $h$. As $\Lift(D,h)$ is a matched $\pSC$-tangle diagram, we can define $\phi(D,h):=\Phi(\Lift(D,h))$.

\blem $\phi(D,h)$ depends only on $\al$. Consequently we have a map $\phi\colon B \to \cY(\fS {;\cR})$.
\elem
\bpr
Let us show $\phi(D,h)$ does not depend on $h$. It is enough to show that $\phi(D,h)$ does not change if we exchange the $h$-order of two consecutive points on $\al \cap c, c \in \cC$.
 The height exchange formula \eqref{eqHeight} tells us the change in height on $c'$ results in a factor $q^\epsilon$ for an $\epsilon\in \{\pm 1\}$, while the change in $c''$ is in the opposite direction and results in the factor  $q^{-\epsilon}$. Thus $\phi(D,h)$ does not change. We will drop $h$ in the notation in $\phi(D,h)$.
 
Let us show  $\phi(D)= \phi(D')$ when $D,D'$ are  good diagram representatives $\al$. Lemma~\ref{rSlides2} reduces the proof to the case when $D'$ is obtained from $D$ by a t-move or a loop-move.
 
$\bullet$ the loop-move at $c\in \cC$. On the lift this move is $\ell_{c'} \leftrightarrow \ell_{c''}$.
 By \eqref{eqLoop_utr}, we have 
 $$\uTR(\ell_{c'})= u_{c'}^2 + u_{c'}^{-2}, \ \uTR(\ell_{c''})= u_{c''}^2 + u_{c''}^{-2}. $$ 
 As $u_{c'}= u_{c''}$ in $\cY(\fS {;\cR})$, we get that $\uTR(\ell_{c'})= \uTR(\ell_{c''})$, implying $\phi(D) = \phi(D')$.
 
$\bullet$ t-move at $c\in \cC$. We have 
 $$ \Lift(D',h) = \theta_{c'}^{\epsilon}  \theta_{c''}^{-\epsilon} ( \Lift(D,h)),$$
 where $\epsilon\in \{ \pm 1\}$, depending on the direction of the t-move. Identity \eqref{eqTwist_utr} shows that $\uTR(\Lift(D',h)) = \uTR( \Lift(D,h))$, implying $\phi(D)= \phi(D')$.
\epr
 Extend   $\phi$ linearly to  map, also denoted by $\phi\colon \Gr^F( {\SSR}) \to \cY(\fS {;\cR})$.

\blem The  map $\phi\colon \Gr^F( {\SSR}) \to \cY(\fS {;\cR})$ is an $\cR$-algebra homomorphism.
\elem

\def\Cr{\mathsf{Cr}}
\def\Res{{\mathsf {Res}}}
\def\dcup{{\vec \cup   }}
\bpr
Let $\al, \al'\in B$. We need to show
\be
\phi(\al * \al')=\phi(\al) \phi(\al').
\label{eqAlgPhi}
\ee 
Let $D$ and $D'$ be respectively good position representatives of $\al$ and $\al'$. We can assume that $D\cap D'$ does not intersect any $c\in \cC$. Choose an order $h$ (respectively $h'$) on  $D \cap c$ (respectively $D'\cap  c$), for each $c\in \cC$.

The product $\al\al'$ in $ {\SSR}$ is presented by the diagram $D \dcup D'$, which is the union of $D$ and $D'$ with $D$ above $D'$. Similarly, the product $\Lift(D,h)\,  \Lift(D', h')$ in $\uS(\cSC {;\cR})$ is given by the diagram $\Lift(D,h) \dcup \Lift(D', h')$.

Now we use the defining relations \eqref{eq.skein}-\eqref{eqF} to  present $D \dcup D'$ as a linear sum of basis elements. Recall that $\cM$ is the set of boundary components of $\fS= \Sigma_{g,m}$. As $D,D'$ are simple, each $s \in \cM$ has at most 2 endpoints of $D \dcup D'$.

If $\gamma$ is the diagram of the left hand side of either \eqref{eq.skein} or \eqref{eqE}  then the first and the second diagrams on the right hand side are called, respectively, the positive and negative resolutions of $\gamma$. 
Let $S$ be the set consisting of (i) all crossings of $D \dcup D'$ and (ii) all $s\in \cM$ having 2 endpoints of $D \dcup D'$. For $\rho\colon S \to \{ \pm \}$, let $(D \dcup D')_\rho$ be the result of first applying to $D \dcup D'$  the  $\rho(s)$-resolution at $s$ for all $s\in S$, then applying the trivial loop relation \eqref{eq.loop} and the peripheral loop relation \eqref{eqF}  to any component that lies entirely in one of the faces. Note that  $(D \dcup D')_\rho$ is 0 in $\Gr^F( {\SSR})$ if in a face $\tau$ it has an arc $a$ which can be homotoped relative $\partial a$ into  $\partial \tau$, i.e. the lift of $a$ is a near boundary arc in $\cSC$. Call such a  $\rho$ non-admissible. Thus in $\Gr^F( {\SSR})$ we have
\be 
\al * \al' = \sum_{\rho \ \text{admissible}} \Coeff(\rho)\,  (D \dcup D')_\rho , \quad \Coeff(\rho)\in \cR. \label{eq98}
\ee

The same procedure associated to $\rho$ applies to $\Lift(D,h) \dcup \Lift(D', h')$, and if $\rho$ is non-admissible then the resulting diagram is 0 due to the presence of near boundary arcs. Hence,
 in $\uS(\cSC {;\cR})$, 
\begin{align}
\Lift(D,h)\,  \Lift(D', h') &=  \sum_{\rho \ \text{admissible}} \Coeff(\rho)\,   (\Lift(D,h) \dcup \Lift(D', h'))_\rho \, .\label{eq97a}
\end{align}
 From \eqref{eq98} and \eqref{eq97a} we get \eqref{eqAlgPhi}.
\epr

It should be noted that if we work in the ordinary skein algebra, $ {\SSR}$, then a non-admissible $\rho$ would contribute non-trivially to $\al  \al'$ in the analog of \eqref{eq98},  but would contribute 0 to the product $\Lift(D,h)\,  \Lift(D', h')$ in $\uS(\cSC {;\cR})$. This is the reason we work with the associated graded algebra $\Gr^F{( {\SSR})}$.

\subsection{Proof of Theorem \ref{thmMain2}} \label{ssProofL}
 
 Let us  prove \eqref{eqtopdeg}, which follows fairly easily from~\eqref{eqHighdeg_utr} and Property {\bf (N1)} of the numeration of $\cC$ given in Subsection \ref{ssN1}. Here are the details.
 
Due to \eqref{eqR3}, it is enough to prove \eqref{eqtopdeg} for $\cR=\ZQ$, which is assumed now.

\def\tord{{\widetilde{\ddd}}} 
 \def\Ctau{{\cC_\tau^{(2)}   }}
 \def\tbn{{\tilde\bn}}
 \def\tbt{{\tilde\bt}}
 For a face $\tau $ let $\Ctau\subset \cCt$ be the set of bold boundary components of $\tau$. For a function $\bn: \Ctau\to \BZ$ let $\tbn: \cCt \to \BZ$ be its 0 extension, which takes value 0 on $\cCt\setminus \Ctau$.

Consider the following  functions
 \begin{align*}
& \ddd_\tau: \BN^{\Ctau} \times \BZ^{\Ctau}  \to \BZ^3, \quad \text{defined in Subsection \ref{ssQtPj}} \\
& \tord: \BN^{\cCt} \times \BZ^{\cCt}  \to (\frac 12 \BZ)^{\barm+2}, \ (\bn, \bt) \to (\frac 12\sum_{a\in \cCt} \bn(a), \sum_{a\in \cCt} \bt(a), \bt(c_1') + \bt(c_1''), \dots, \bt(c_\barm') + \bt(c_\barm'')).
% \\ \ddd&: \BZ^{\cC \sqcup \cC } \to \BZ^{m+2}, \ (\bn, \bt) \to (\sum_{a\in %\cC} \bn(a), \sum_{a\in \cC} \bt(a), \bt(c_1), \dots, \bt(c_m) ).
 \end{align*}
Property (N1) guarantees the extension from $\Ctau$ to $\cCt$ preserves the orders defined by $\ord$:
\blem \label{rOrder2}
If
$\ddd_\tau((\bn, \bt)) \le \ddd_\tau((\bn', \bt'))$ then $\tord((\tbn, \tbt)) \le  \tord((\tbn', \tbt'))$.
\elem
 Assume $\al \in B$. Let $\al_\tau$ be the lift of $\al$ in face $\tau$, with some boundary order. By \eqref{eqHighdeg_utr},
$$ \utr(\al_\tau) \qeq (Y_\tau)^{\nu(\al_\tau)} + E_{ < \ddd(\nu(\al_\tau)  ) }(\cY(\Pb)).$$
Consider $\cY(\Pb)$ as a subalgebra of $\tY$. By  Lemma \ref{rOrder2}
order is preserved by extension. Hence
$$ \utr(\al_\tau) \qeq  Y^{\widetilde{\nu(\al_\tau)}} + E_{ < \tord(\widetilde{\nu(\al_\tau)}  ) }(\tY).$$
Taking the product over all faces and denoting $N(\al) = \sum _{\tau \in \cF} \widetilde{\nu(\al_\tau)}$, we get
\be  \uTR(\al) \qeq Y^{N(\al)  } +  E_{ < \tord(N(\al) ) }(\tY).
\ee
The projection $\EEE \onto \cY(\fS)$ sends $Y^{N(\al)  }$ to $Y^{\nu(\al)}$. Besides $\tord(N(\al))= \ddd(\nu(\al))$. Thus,
$$ \phi(\al) \qeq Y^{\nu(\al)  } +  E_{ < \ddd(\nu(\al) ) }(\cY(\fS)).$$
Using the reflection invariance of Lemma \ref{rReflection}, we can replace $\qeq$ by $=$, and obtain \eqref{eqtopdeg}.

Return to the general ground ring $\cR$. Define the $(\BN\times \BZ^{\barm +1})$-filtration of  $\Gr^F( {\SSR})$ by
$$ E_\bk(\Gr^F( {\SSR})) = 
\cR\text{-span of } \ \{ \al\in B(\fS) \mid \ddd(\nu(\al))  \le \bk\}.  $$

Equation \eqref{eqtopdeg} implies $\phi$ respects the $(\BN\times \BZ^{\barm +1})$-filtrations of its domain and target space, and $\Gr^E(\phi)(\al) = Y^{\nu(\al)}$. 
Since $\Gr^E(\phi)$ maps the basis $B(\fS)$ of
 $\Gr^F( {\SSR})$ bijectively onto the basis 
$\{Y^\bk \mid \bk \in \Lambda\}$ of $\bT(\tmQ,\Lambda {;\cR})$, the map $\Gr^E(\phi)$ is a linear $\cR$-isomorphism. By Lemma~\ref{rLift1}, the map $\phi$ is injective. 

 Since $\Gr^F( {\SSR})$ embeds into the domain $\bT(\tQ,\Lambda {;\cR})$, it is a domain. By Proposition \ref{liftfacts}, $ {\SSR}$ is a domain. This completes the proof of part (a).

\def\bj{{\mathbf j}}
\def\bi{{\mathbf i}}
\def\Supp{{\mathsf{Supp}}}

(b) Due to \eqref{eqR2}, it is enough to prove \eqref{eqProdTop} for $\cR=\ZQ$, which is assumed now.
Let $\bk, \bl \in \Lambda$. 
By \eqref{eqtopdeg},
\begin{align*}
\phi(\mu(\bk)) & = Y^{\bk} + E_{< \ddd(\bk )  }(\cY(\fS {})) \\
\phi(\mu(\bl)) & = Y^{\bl} + E_{< \ddd(\bl )  }(\cY(\fS {})) \\
\phi(\mu(\bk+ \bl)) & = Y^{\bk+ \bl} + E_{< \ddd(\bk+ \bl )  }(\cY(\fS {}))
\end{align*}
Using  $Y^\bk Y^\bl = q^{\frac 12 \la \bk, \bl\ra   } Y^{\bk+\bl}$, we get
\be 
\phi(\mu(\bk) * \mu(\bl) - q^{\frac 12 \la \bk, \bl\ra  } \mu (\bk + \bl)) \ \in E_{< \ddd(\bk+ \bl )  }(\cY(\fS {}))
\ee
As $\phi^{-1} (E_{ < \ddd(\bk+ \bl ) } (\cY(\fS {}))) = E_{ < \ddd(\bk+ \bl ) } (\Gr^F({\SS}))$ by Lemma \ref{rLift1}, we have, in $\Gr^F({\SS})$,
\be 
\mu(\bk) * \mu(\bl) =  q^{\frac 12 \la \bk, \bl \ra  } \mu(\bk+\bl) +   \sum c_\bj\, \mu(\bj), \quad c_\bj\in \cR, \ddd(\bj) < \ddd(\bk+ \bl).
\label{eq101a}
\ee 
By definition of $\Gr^F({\SS})$, we have $\ddd_1(\bj)= \ddd_1(\bk+\bl)$. Lifting \eqref{eq101a} to ${\SS}$, we have
\be 
\mu(\bk) \mu(\bl) =  q^{\frac 12 \la \bk, \bl \ra  } \mu(\bk+\bl) +   \sum c_\bj\, \mu(\bj) + \sum c_\bi\, \mu(\bi), \quad c_\bi\in \cR, \ddd_1(\bi) < \ddd_1(\bk+ \bl). \label{eq101h}
%F_{ < \ddd_1(\al \oplus \beta)}(\SS).
\ee
Since $ \ddd_1(\bi) < \ddd_1(\bk+ \bl)$ implies $\ddd(\bi) < \ddd(\bk+ \bl)$, the last two sums in \eqref{eq101h} are in $E_{ < \ddd(\bk+\bl)}( {\SSR})$, and we get 
 \eqref{eqProdTop}.

%ccc

(c) follows right away from Identity \eqref{eqProdTop}.

(d) was proved by Theorem \ref{thm-noether1}.

(e)   {
 By (d) and Proposition \ref{liftfacts}, and then Lemma \ref{lemma-GKdim}, we have
$$ \GKdim ( {\SSR}) \ge \GKdim (\Gr^E( {\SSR}))= 2r.$$

To prove the converse inequality, we use the following lemma. Define the norm on $\BR^k$ by
$$ \|(x_1, \dots, x_k) \|_\infty := \max | x_i|. $$
\blem There exists  a finite collection $\cD$ of simple closed curves on $\fS$ such that ${\|\nu(\al)\|_\infty} \le \sum_{c \in \cD} I(\al, c)$ for all $\al \in B(\fS)$.  
\elem
\begin{proof} Recall the standard DT twist coordinates $t'_i(\al)$, as seen in Section \ref{secDT}. 
From \cite[Proposition 4.4]{Luo}, there exist a collection  $\cD'$ of simple closed curves and continuous, homogeneous functions $f_i, g_i': \BR^{\cD'} \to \BR$ 
for each $i=1,\dots, r$ such that 
$$ n_i(\al) = f_i (I_\al) , \quad t_i'(\al) = g'_i (I_\al ), \ \text{where} \ I_\al(c) = I(\al, c).$$
The set $\cD'$ contains $\cC$ and  $\cM$.
Since $t_i(\al) -2 t_i'(\al) $ is a piece-wise linear function in the variables $I_\al( c),   c\in \cC \cup \cM$, and piece-wise linear functions are continuous and homogeneous, there are continuous, homogeneous functions $g_i: \BR^{\cD'} \to \BR$ such that $t_i(\al) = g_i (I_\al )$.

Continuity of $f_i, g_i$ implies there is $d\in \BN$ such that for all $x\in \BR^{\cD'}$ with $\|x \|_\infty \le 1$ we have $|f_i(x)|, |g_i(x)| \le d$.  {Homogeneity} implies that for all $x\in \BR^{\cD'}$ we have 
$$ |f_i(x)|\le d \|x\|_\infty,\quad   |g_i(x)| \le d \|x\|_\infty. $$
Replacing each curve in $\cD'$ by $d$ parallel copies of it, from $\cD'$ we get $\cD$ satisfying the requirement of the lemma.
\end{proof}

As we calculate the GK dimension, we can assume $\cR$ is a field. By Theorem \ref{thm-noether1}, $ {\SSR}$ has a finite set $S$ of generators. As $B(\fS)$ is an $\cR$-basis, each element of $x\in S$ is a linear combination of elements of a finite set $\Supp(x) \subset B(\fS)$. Replacing $S$ by $\bigcup_x \Supp(x)$, we can assume that  $S \subset B(\fS)$. Let $V_1$ be the $\cR$-span of $S$ and $V_n = (V_1)^n$. 
 
The lemma implies $W_n \subset U_n$, where
\begin{align*}
W_n :=\cR\text{-span} \{ \al \in B(\fS) \mid \sum_{c \in \cD} I(\al, c) \le n\}, \
U_n :=\cR\text{-span}  \{ \al \in B(\fS) \mid \|\nu(\al)\|_\infty \le n\}.
\end{align*} 
Note that $W_n W_l \subset W_{n + l}$, see Subsection \ref{sec.fil}.
Since $S$ is finite, there is $k\in \BN$ such that $V_1 \subset W_k$. Hence,
$$   V_n = (V_1)^n \subset (W_k)^n \subset W_{kn} \subset U_{kn}.$$
Since the number of  points $x\in \BZ^{2r}$ with $\|x \|_\infty \le k$ is $(2k+1)^{2r}$, we have
$$ \dim_\cR (V_n) \le \dim_\cR(U_{kn})  \le (2kn+1)^{2r},$$
from which we conclude $\GKdim(\SS) \le 2r$. \qed

\appendix
\def\RfS{{\cR_\fS}}

\section{Exceptional cases}
  {

In this appendix we show that the Roger-Yang skein algebras $\SRY(\Sigma_{g,m})$, for the exceptional cases $(g,m)=(0,k), (1,0)$ with $k \le 4$, are domains. Recall that 
$$\RfS= \cR[ v^{\pm 1}, v \in \cM], \ \RfS'=  \cR[ v^{\pm 1/2}, v \in \cM] .$$

\def\SSY{{\SRY(\fS)}}
\subsection{The case $(g,m)=(1,0), (0,1), (0,2), (0,3)$} These are simpler cases.

Let $\fS= \Sigma_{1,0}$. Then $\SSY$ is  a domain, as it embeds into a quantum torus \cite{FrG}.

Let $\fS= \Sigma_{0,1}$. Then $\SSY= \RfS$ by Theorem \ref{basisDiamond}. Hence it is a domain.

Let $\fS=\Sigma_{0,2}$.  The algebra $ {\SSR}$ is commutative and has a presentation \cite{Wong0}
$$ {\SSR} =\RfS [\al ]/ (p(\al)), \quad p(\al) =  v_1 v_2\al^2+(q-q^{-1})^2 \in \RfS [\al ].  $$
Here $\cM=\{ v_1, v_2\}$. Since $p(\al)$ is irreducible in $\RfS [\al ]$, the quotient $ {\SSR}$ is a domain.

Let $\fS= \Sigma_{0,3}$. The algebra $ {\SSR}$ is commutative and has a presentation \cite{Wong0}
$$ {\SSR} =\RfS [\al_1, \al_2, \al_3 ]/ (v_{i+1}\al_i \al_{i+1} = \delta \al_{i+2}, v_{i+1} v_{i+2} \al_i^2 = \delta^2), \delta = q^{1/2} + q^{-1/2}.  $$
Here $\cM= \{v_1, v_2, v_3\}$, indices are taken mod 3, and $\al_i$ is the arc connecting $v_{i+1}$ and $v_{i+2}$. From the presentation we get an algebra homomorphism $f\colon {\SSR} \to \RfS'$ given by $f(\al_i)= \delta v_{i+1}^{-1/2}v_{i+2}^{-1/2} $. From Theorem \ref{basisDiamond},
$$ \SSY = \RfS \oplus \RfS \,  \al_1 \oplus \RfS \,  \al_2 \oplus \RfS \,  \al_3.$$
Since $1, f(\al_1), f(\al_2), f(\al_3)$ are linearly independent over $\RfS$, the map $f$ is injective. Hence, as a subalgebra of a domain, $\SSY$ is a domain.
\brem It should be noted that $\cS(\Sigma_{0,3} {;\cR})$ is not a domain, as $(\al_1-\delta)(\al_1+\delta)=0$ in $\cS(\Sigma_{0,3} {;\cR})$, but neither factor is 0.
\erem

\subsection{The case $\fS=\Sigma_{0,4}$} This case is much more difficult since  $\SSY$ is not commutative. Even though an explicit presentation exists \cite{Moon}, it does not seem to help to prove that $\SRY(\Sigma_{0,4})$ is a domain. Instead, we modify the proof of Theorem \ref{thmMain2}. We sketch here the main steps, leaving the details for the dedicated reader.

Let $\cC=\{b\}=\{b_1\}$ be the pants decomposition, with circle boundary component set $\cM=\{b_2, b_3, b_4, b_5\}$ as in Figure \ref{fig:04}. The upper pair of pants $\tau$ is the standard $\PP_1$. Let $\tau'$ be the lower pair of pants. The graph $\Gamma$ is on the back, like in Figure \ref{fig:trian_P1}.

\begin{figure}[htpb!]
    %\centering
    \includegraphics[height=3.2cm]{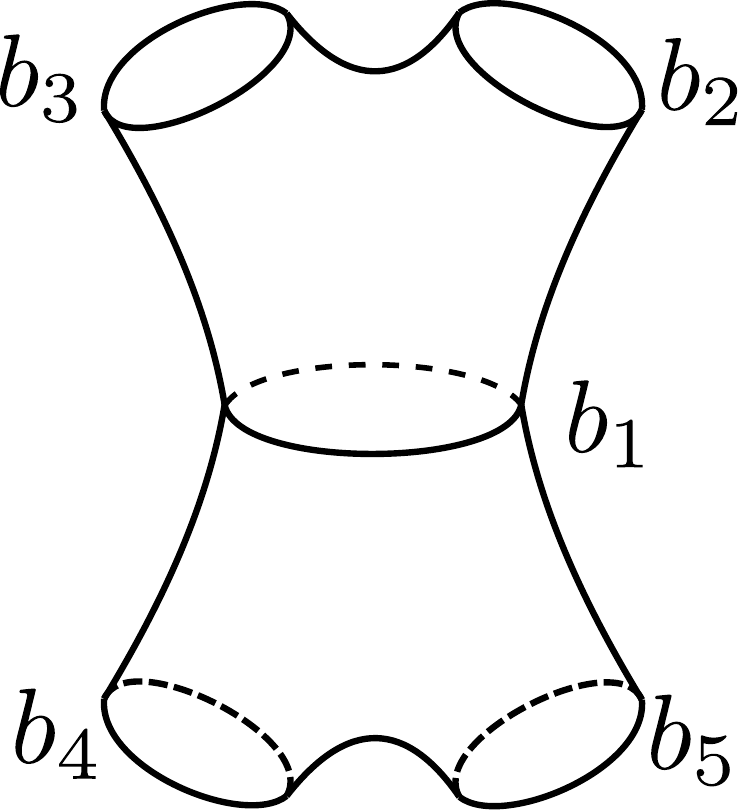}
    \caption{}
    \label{fig:04}
\end{figure}

First we define DT coordinates for $\al\in\uB(\fS)$. Assume $\al$ is in good position. We can define coordinates $(n(\al_\tau), t(\al_\tau))$ and $(n(\al_{\tau'}, t(\al_{\tau'}))$ as in Subsection \ref{ssP1}. As usual let
$$ \nu(\al) = (n(\al), t(\al)) \in \BN \times \BZ, 
n(\al)= n(\al_\tau) = n(\al_{\tau'}), \ t(\al)= t(\al_{\tau})+ t(\al_{\tau'}).$$
  The proof of Proposition \ref{rDTcoord} does not quite work since if $\beta$ is isotopic to $\al$ and also in good position, then it might happen that $t(\al_\tau ) - t(\beta_\tau)$ is odd.  To get rid of the ambiguity, define $\bar \nu(\al) =0$ or $1$ according to whether  $t(\al_\tau)$ is even or odd.
The proof  Proposition \ref{rDTcoord} gives

\blem The map $(\bar \nu, \nu): \uB(\Sigma_{0,4}) \to \{0,1\} \times (\BN \times \BZ)$  is injective, and its image is $\{0,1\} \times \Lambda_1$.
\elem

Let $\cY^\RY= \cY(\fS,\cC,\Gamma {;\cR})\ot_\cR \RfS'=
 \RfS' \la y^{\pm 1} , 
 u^{\pm 1}\ra /( uy = q yu)$.
 Define the $F$-filtration on $\SSY$ using  the same \eqref{eqFil1}, with $\cR$ replaced by $\RfS$. Thus, for $k\in \BN$,
 $$F_k(\SSY)= \RfS\text{-span of } \ \{ \al \in \uB(\fS) \mid n(\al) \le k \}.$$
Similarly, define the $E$-filtration on $\cY^\RY$ using \eqref{eqFil3}, with $\cR$ replaced by $\RfS'$.

 \def\eqbuu{\, {\overset \bullet =}\, }
 
Define the $\RfS$-linear map $\phi^\RY: \Gr^F( \SSY) \to \cY^\RY$ so that for an element $\al$ of the basis $\uB(\fS)$ we have
$$ \phi^\RY(\al) = \left( \prod_{m \in \cM } m^{-\frac 12 |m \cap \al  |   }  \right) \phi(\al).$$
Using \eqref{eqSRY} and that $\phi$ is an algebra homomorphism, one sees that $\phi^\RY$ is an $\RfS$-algebra homomorphism. The proof of \eqref{eqtopdeg} can be repeated for this new setting, with \eqref{eqn3}, and the result is
\be 
\phi^\RY(\al) \eqbuu b_3^{ \bar \nu (\al)/2  } Y^{\nu (\al) } + E_{ \ddd(\nu(\al) } (\cY^\RY),
\ee 
where $x \eqbuu z$ means there are integers $k_i\in \BZ$ such that $x \qeq   (b_3^{ k_3} b_2^{k_2/2} b_4^{k_4/2} b_5^{k_5/2})z$.  From here we see that the map $\phi^\RY$ is injective. It follows that $\Gr^F( \SSY)$, as a subalgebra of a domain, is a domain. Hence $\SSY$ is a domain.
}

\bibliographystyle{hamsalpha}
\bibliography{biblio}

 \end{document}